\documentclass[11pt]{amsart}
\usepackage{amsmath,amsfonts,amssymb,mathrsfs}
\usepackage{amssymb,mathrsfs,graphicx,enumerate,mathabx}
\usepackage{amssymb,amscd,amsthm,bbm}
\usepackage{mathtools}
\mathtoolsset{showonlyrefs}
\usepackage[retainorgcmds]{IEEEtrantools}
\usepackage{colortbl}
\usepackage{lipsum}
\usepackage{graphicx, subfigure}
\usepackage{graphicx}
\usepackage{hyperref}
\mathtoolsset{showonlyrefs}

\title[]{Global energy minimizers for a diffusion-aggregation model on sphere }

\author[Fetecau]{Razvan C. Fetecau}
\address[Razvan C. Fetecau]{\newline Department of Mathematics, Simon Fraser University, 8888 University Dr., Burnaby, BC V5A 1S6, Canada}
\email{van@math.sfu.ca}

\author[Park]{Hansol Park}
\address[Hansol Park]{\newline Department of Mathematics, National Tsing Hua University, Section 2, Kuang-Fu Road, Hsinchu 30013, Taiwan}
\email{hansolpark@math.nthu.edu.tw}

\author[Vaidya]{Vishnu Vaidya}
\address[Vishnu Vaidya]{\newline Department of Mathematics, Indian Institute of Science Education and Research, Dr. Homi Bhabha Road, Pune 411008, India}
\email{vishnu.vaidya@students.iiserpune.ac.in}

\newtheorem{theorem}{Theorem}[section]
\newtheorem{lemma}{Lemma}[section]

\newtheorem{proposition}{Proposition}[section]
\newtheorem{remark}{Remark}[section]

\newcommand{\bbr}{\mathbb R}

\newcommand{\bbs}{\mathbb S}

\newcommand{\calA}{\mathcal{A}}
\newcommand{\calF}{\mathcal{F}}

\newcommand{\calP}{\mathcal{P}}

\def\d{\mathrm{d}}
\newcommand{\dS}{\mathrm{d}S} 
\newcommand{\dx}{\mathrm{d}S (x)}
\newcommand{\dy}{\mathrm{d}S (y)}

\newcommand{\rhou}{\rho_{\text{uni}}}
\newcommand{\rhoeps}{\rho^\epsilon}

\newcommand{\dm}{d} 

\begin{document}


\begin{abstract}
We investigate the ground states of a free energy functional on sphere. The energy consists of an entropy and a nonlocal interaction term that are in competition with each other, as they favour spreading and aggregation, respectively.  Specifically, the entropy corresponds to slow nonlinear diffusion and the interaction term is modeled by a quadratic interaction potential. We investigate the transitions that occur in the equilibria and the global minimizers of the energy, in terms of the strength of the nonlocal attractive interactions. We consider separately various ranges of the diffusion exponent, which give qualitatively different behaviours of equilibria and ground states. In terms of applications, we note that the energy we consider here is a generalization to nonlinear diffusion of the Onsager free energy with dipolar potential, used to study phase transitions in polymer orientation.
\end{abstract}

\maketitle 

\section{Introduction}
\label{sect:intro}
In this paper we investigate the minimizers of the free energy functional
\begin{equation}
\label{energy-sphere}
E[\rho]=\frac{1}{m-1}\int_{\bbs^\dm}\rho(x)^m \dx +\frac{\kappa}{4}\iint_{\bbs^\dm\times \bbs^\dm}\|x-y\|^2 \rho(x)\rho(y)\dx \dy,
\end{equation}
defined on the space $\mathcal{P}_{ac}(\bbs^\dm)$ of absolutely continuous probability measures\footnote{Note that throughout the paper we refer to an absolutely continuous measure directly by its density, and by an abuse of notation we write $\rho\in \calP_{ac}(\bbs^\dm)$ instead of $\d \rho  = \rho \,\d S \in \calP_{ac}(\bbs^\dm)$.} on the unit sphere $\bbs^\dm$.  Here, $m>0$ is the diffusion exponent ($m=1$ requires special consideration), $\kappa>0$ represents the interaction strength, and $\|\cdot\|$ denotes the Euclidean distance in $\bbr^{\dm+1}$. Also, the integration is with respect to the surface area measure $\dS$ of the sphere, and the notion of absolute continuity is with respect to the measure $\dS$.

The functional \eqref{energy-sphere} falls in a class of aggregation-diffusion energies extensively studied in various contexts, in particular in self-organizing phenomena such as swarming or flocking of biological organisms \cite{keller1970initiation, patlak1953random, TBL}, emergent behaviour in robotic swarms \cite{Markdahl_etal2018, OlfatiSaber2006}, self-assembly of nanoparticles \cite{HoPu2006}, and opinion formation \cite{MotschTadmor2014}.  In its general form, the energy \eqref{energy-sphere} is defined on probability measures on a generic Riemannian manifold $M$: 
\begin{equation}
\label{eqn:energy-M}
E[\rho]=\frac{1}{m-1} \int_M \rho(x)^m \d x +\frac{1}{2}\iint_{M \times M} W(x, y)\rho(x)\rho(y)\d x \d y,
\end{equation}
where $W: M \times M \to \bbr$ is an interaction potential, and integration is with respect to the Riemannian volume measure $\d x$.  The energy \eqref{eqn:energy-M} relates to the following nonlinear nonlocal evolution equation for the population density $\rho$:
\begin{equation}
\label{eqn:model}
\partial_t\rho(x)- \nabla_M \cdot(\rho(x)\nabla_M W*\rho(x))=\Delta \rho^m(x),
\end{equation}
where
\[
W*\rho(x)=\int_M W(x, y)\rho(y) \d y,
\]
and $\nabla_M \cdot$ and $\nabla_M $ represent the Riemannian divergence and gradient, respectively \cite{AGS2005}. Specifically, critical points of the energy functional correspond to steady states of \eqref{eqn:model}. We aso note that for $m=1$, which corresponds to linear diffusion, the entropy term in \eqref{eqn:energy-M} is replaced by $\int_M \rho(x) \log \rho(x) \d x$. 

Both the free energy functional \eqref{eqn:energy-M} and its corresponding gradient flow \eqref{eqn:model} have been extensively studied on the Euclidean space $M=\bbr^\dm$. A partial list of issues addressed by analysts includes the existence of global energy minimizers for interaction potentials of Riesz type \cite{CalvezCarrilloHoffmann2017,CaHoMaVo2018,CaDePa2019}, uniqueness and qualitative properties (such as monotonicity and radial symmetry) of energy minimizers or steady states of \eqref{eqn:model}  \cite{BuFeHu14,Kaib17,CaHiVoYa2019,DelgadinoXukaiYao2022}, and well-posedness and long time behaviour of solutions to the evolution equation \eqref{eqn:model} \cite{CarrilloCraigYao2019, CaHiVoYa2019, CaFeGo2024}. We also refer here to the review papers \cite{CaMcVi2006,CarrilloCraigYao2019} and the influential monograph \cite{AGS2005}. The literature on general manifolds is much less developed however. The case when $M$ is a Cartan--Hadamard manifold was studied recently in \cite{FePa2024b, FePa2024a} for linear diffusion and in \cite{CaFePa2025a, CaFePa2025b} for nonlinear diffusion. As far as the manifold setup is concerned, most of the research has focused in fact on the model without the diffusion term (when only nonlocal interactions are considered). The well-posedness of the model without diffusion on manifolds  was studied in \cite{FePaPa2020, FePa2021, WuSlepcev2015}, and emergent behaviours were studied on a variety of specific manifolds such as sphere \cite{HaChCh2014, FeZh2019, FePaPa2020}, unitary matrices \cite{HaKoRy2018, Lohe2009}, hyperbolic space \cite{FePa2023a, FeZh2019, Ha-hyperboloid}, the special orthogonal group \cite{FeHaPa2021} and Stiefel manifolds \cite{HaKaKi2022}, as well as on general Riemannian manifolds of bounded curvature \cite{FePa2023b}.

In the present paper we consider the case
\begin{equation}
\label{eqn:M-and-W}
M=\bbs^{\dm},\qquad\text{ and }\quad W(x,y)=\frac{\kappa}{2}\|x-y\|^2.
\end{equation}
The interactions corresponding to this potential are purely {\em attractive}, i.e., any two points experience a pairwise attractive interaction. Consequently, the interaction energy favours points to aggregate together. On the other hand, the entropy component favours spreading, so the two components of the energy \eqref{energy-sphere} compete with each other. On compact manifolds such as the sphere, diffusion by itself leads to global energy minimizers that are uniform densities on the entire space. Together with the attractive interactions, diffusion may still dominate (and lead to uniform states) if the attraction is sufficiently weak. The main interest in this paper is to study the competition between attraction and diffusion in terms of the size of the interaction strength $\kappa$.

Apart from its own intrinsic interest as an aggregation-diffusion energy on the sphere, one of the main motivations for this work is to consider nonlinear diffusion in the Onsager free energy on the sphere \cite{Onsager1949}. The free energy \eqref{energy-sphere} for the special case $m=1$ -- i.e., with the entropic term given by $ \displaystyle \int_{\bbs^\dm} \rho(x) \log \rho(x) \dx$ -- has been used to study isotropic-nematic phase transitions in rod-like polymers and has an extensive literature on its own (see \cite{ConstantinKevrekidisTiti2004, fatkullin2005critical} and the references therein). In this context, $\rho(x)$ represents the probability distribution function for the orientation of a polymer viewed as a rigid rod of unit vector $x \in \bbs^2$, and the interaction potential in \eqref{eqn:M-and-W} is called the dipolar potential; note that since points on $\bbs^\dm$ have unit norm, an equivalent expression of the potential used in this literature is $-\kappa \, x \cdot y$.  It was shown for $\dm=2$ \cite{fatkullin2005critical}, and later generalized to arbitrary dimension in \cite{FrouvelleLiu2012, DegondFrouvelleLiu2014} that a phase transition occurs at a certain critical value of $\kappa$, where the isotropic state (given by the uniform distribution) loses stability to a nematic equilibrium density.

In this paper we work with general exponent $m>1$, which corresponds to slow nonlinear diffusion, and investigate the ground states of the energy functional \eqref{energy-sphere} in terms of the interaction strength $\kappa$. Depending on the values of $m$ (we distinguish between the cases $1<m<2$, $m=2$ and $m>2$), we identify various critical values of $\kappa$ that lead to phase transitions in the ground states, similar to the isotropic-nematic transitions for the Onsager energy. We also note that, different from the case of linear diffusion, we also account here for equilibria that have support strictly included in $\bbs^\dm$. This adds a new type of transition, where ground states change from being fully supported to being strictly supported on the sphere. For the range $1<m\leq 2$ we characterize the global energy minimizers and their phase transitions, in general dimension $\dm \geq 1$. For $m>2$, we carry out our investigations only in dimension $\dm =2$, which in fact is the most important case as far as applications to phase transitions in polymer orientations are concerned.



We also note that on top of the numerous applications to self-collective behaviour in sciences and engineering, there have been recent interest in such models in the context of inverse problems and machine learning \cite{Bongini_etal2017,huilearning, RotskoffEijnden2022}. In particular, some of these works include manifold setups as in the present paper \cite{Maggioni-etal21}. Very recently, an interacting particle system was used in the context of artificial intelligence, more specifically for large language models \cite{Geshkovski_etal2025}. In this application, the phenomena of clustering and synchronization is important for learning tasks. We point out that the system used in \cite{Geshkovski_etal2025} is set up on the unit sphere of general dimension, and the dependence of solutions on the relative strength of the noise and diffusion is listed as an interesting question that remains to be investigated. 


{\em Summary.} The summary of this paper is as follows. In Section \ref{sect:cp} we identify the possible critical points of the energy functional. In Section \ref{sect:unif-distr} we focus on the uniform density and study its stability. In Section \ref{sect:bifurcations-mg1} we investigate the critical values of $\kappa$, the equilibria, and the global energy minimizers when the diffusion exponent is in the range $1<m<2$. Sections \ref{sect:m2} and \ref{sect:bifurcations-mg2} are similar in scope, and they cover the cases $m=2$ and $m>2$, respectively. Some technical details of the results are deferred to the Appendix.

\section{Critical points of the energy}
\label{sect:cp}
In this section we identify the critical points of the energy functional, which are the candidates for ground states. For any $x,y\in \bbs^\dm$, we have 
\[
\|x-y\|^2=\|x\|^2-2\langle x, y\rangle+\|y\|^2=2-2\langle x, y\rangle,
\] 
which allows to simplify the energy as
\begin{equation}
\label{eqn:energy-cmform}
E[\rho]= \frac{1}{m-1}\int_{\bbs^\dm}\rho(x)^m \dx-\frac{\kappa}{2}\iint_{\bbs^{\dm}\times \bbs^{\dm}}\langle x, y\rangle \rho(x)\rho(y)\dS(x)\dS(y)+\frac{\kappa}{2}.
\end{equation}
The centre of mass $c_\rho \in \bbr^{\dm+1}$ of a density $\rho\in \calP_{ac}(\bbs^\dm)$ is defined as
\begin{equation}
\label{eqn:cm}
c_\rho=\int_{\bbs^{\dm}}x\rho(x) \dS(x).
\end{equation}
It holds that
\begin{equation}
\begin{aligned}
\|c_\rho\|^2 &=\left\langle
\int_{\bbs^{\dm}}x\rho(x) \dS(x), \int_{\bbs^{\dm}}y\rho(y) \dS(y)
\right\rangle \\[2pt]
&=\iint_{\bbs^\dm\times \bbs^\dm}\langle x, y\rangle \rho(x)\rho(y)\dS(x)\dS(y),
\end{aligned}
\end{equation}
and hence, the energy \eqref{energy-sphere} can be written as
\begin{equation}
\label{eqn:energy-s}
E[\rho]= \frac{1}{m-1}\int_{\bbs^\dm}\rho(x)^m \dx -\frac{\kappa}{2}\|c_\rho\|^2+\frac{\kappa}{2}.
\end{equation}

The Euler--Lagrange equation for the functional \eqref{eqn:energy-s} is given by
\begin{equation}
\label{eqn:equil}
\left(\frac{m}{m-1}\right)\rho(x)^{m-1}-\kappa \langle c_\rho, x\rangle = \lambda, \qquad x\in \mathrm{supp}(\rho),
\end{equation}
for some constant $\lambda$.  We first note that the uniform distribution on the sphere, 
\begin{equation}
\label{eqn:rho-uni}
\rhou(x) = \frac{1}{|\bbs^\dm|}, \qquad \forall x \in \bbs^\dm,
\end{equation}
where $|\bbs^\dm|$ denotes the area of $\bbs^\dm$, is a solution of \eqref{eqn:equil} for all $\kappa>0$. Indeed, in this case, $c_{\rhou} =0$, and the constant $\lambda$ is given by
\begin{equation}
\label{eqn:lambda-uni}
\lambda_{\text{uni}} = \left(\frac{m}{m-1}\right) \frac{1}{|\bbs^\dm|^{m-1}}.
\end{equation}

Our main goal is to identify the global energy minimizers, among the critical points satisfying \eqref{eqn:equil}. The global minimizers must be radially symmetric, a property that can be inferred from the following argument. Consider a global minimizer $\rho$, and take a rotation $R:\bbs^\dm\to\bbs^\dm$ that preserves its centre of mass $c_\rho$. Note that the center of $\rho'=R_\#\rho$ is also $c_\rho$ (i.e., $c_{\rho'}=c_{\rho}$), and $E[\rho']=E[\rho]$. Then, any convex combination of $\rho$ and $\rho'$,
\[
\rho_\alpha=(1-\alpha)\rho+\alpha\rho', \qquad \text{ for } 0<\alpha<1,
\]
yields
\[
E[\rho_\alpha]<(1-\alpha)E[\rho]+\alpha E[\rho']=E[\rho],
\]
whenever $\rho\neq \rho'$; here we used the strict convexity of the entropy and the fact that $c_\rho=c_{\rho'}=c_{\rho_\alpha}$ for all $0<\alpha<1$. Therefore, to be a global minimizer, we need $\rho=\rho'$ 
 and hence, $\rho$ is invariant under rotations that preserve its centre of mass. We infer that any global minimizer $\rho$ must have rotational symmetry.
  

We now investigate the radially symmetric solutions of \eqref{eqn:equil}. We find that \eqref{eqn:equil} admits two types of solutions: equlibria supported on the entire sphere ($\mathrm{supp}(\rho)=\bbs^\dm$), and equilibria that are supported on a strict subset of the sphere ($\mathrm{supp}(\rho) \subsetneq \bbs^\dm$). Write the equilibria as
\[
\rho(x)=\begin{cases}
\left(\frac{m-1}{m} \right)^{\frac{1}{m-1}} \left(\lambda+\kappa\langle c_\rho, x\rangle\right)^{\frac{1}{m-1}},\quad &\text{for }x\in \mathrm{supp}(\rho),\\[5pt]
0,&\text{otherwise}.
\end{cases}
\]
Note that for $x\in \mathrm{supp}(\rho)$, $\lambda+\kappa\langle c_\rho, x\rangle\geq0$.

Without loss of generality, we can assume $c_\rho=\|c_\rho\| x_0$ for some unit vector $x_0 \in \bbs^\dm$. Define $\theta_x=\arccos\langle x_0, x\rangle \in [0,\pi]$, to write $\langle c_\rho, x\rangle=\|c_\rho\|\cos\theta_x$. The condition $\lambda+\kappa\langle c_\rho, x\rangle\geq0$ on $\mathrm{supp}(\rho)$ can be simplified into
\begin{equation}
\label{ineq:cos}
-\frac{\lambda}{\kappa\|c_\rho\|}\leq\cos\theta_x, \qquad \forall x \in \mathrm{supp}(\rho).
\end{equation}
Given that $\cos \theta_x$ ranges in $[-1,1]$, we now distinguish the two types equilibria, in terms of the relative size of $\lambda$. Note that \eqref{ineq:cos} has no solutions if $\lambda\leq -\kappa\|c_\rho\|$.

{\em a. Equilibria fully supported on $\bbs^\dm$ ($\lambda \geq  \kappa \|c_\rho\|)$.} In this case $-\frac{\lambda}{\kappa\|c_\rho\|}\leq -1$, so $\theta_x$ has full range $[0,\pi]$. The equilibrium is given by
\begin{equation}
\label{eqn:equil-fs}
\rho(x)=\left(\frac{m-1}{m}\right)^{\frac{1}{m-1}}\left(\lambda+\kappa\|c_\rho\|\cos\theta_x\right)^{\frac{1}{m-1}},\qquad\forall x\in \bbs^\dm.
\end{equation}
To determine $\rho$ one needs to find $\lambda$ and $\|c_\rho\|$. These can be found by requiring that $\rho$ has unit mass and centre of mass at $c_\rho$, as following.

From the definition of $c_\rho$ we compute
\begin{equation}
\label{eqn:cm-sq}
\begin{aligned}
\|c_\rho\|^2&=\left\langle c_\rho, \int_{\bbs^\dm}x\rho(x)\dS(x)\right\rangle\\
&= \int_{\bbs^{\dm}}\langle c_\rho, x\rangle \rho(x)\dS(x)\\
&=\dm w_{\dm}\|c_\rho\|\int_0^\pi \cos\theta \left(\frac{m-1}{m}\right)^{\frac{1}{m-1}}\left(\lambda+\kappa\|c_\rho\|\cos\theta\right)^{\frac{1}{m-1}}\sin^{\dm-1}\theta\d\theta,
\end{aligned}
\end{equation}
where for the last equal sign we used hyperspherical coordinates and that $|\bbs^{\dm-1}|=\dm w_{\dm}$, where $w_{\dm}$ denotes the volume of the $\dm$-dimensional unit ball. Together with the unit mass condition, this leads to the following two equations to be solved for $\lambda$ and $\|c_\rho\|$:
\begin{equation}
\label{eqn:system-fs-mg1}
\begin{aligned}
1&=\dm w_\dm\int_0^\pi \left(\frac{m-1}{m}\right)^{\frac{1}{m-1}}\left(\lambda+\kappa\|c_\rho\|\cos\theta\right)^{\frac{1}{m-1}}\sin^{\dm-1}\theta \, \d\theta, \\
\|c_\rho\|& =\dm w_\dm\int_0^\pi \left(\frac{m-1}{m}\right)^{\frac{1}{m-1}}\left(\lambda+\kappa\|c_\rho\|\cos\theta\right)^{\frac{1}{m-1}}\sin^{\dm-1}\theta \cos\theta \, \d\theta.
\end{aligned}
\end{equation}
\medskip

{\em b. Equilibria supported on a strict subset of $\bbs^\dm$ ($-\kappa \|c_\rho\| < \lambda < \kappa \|c_\rho\|$).} In this case we have $-1<-\frac{\lambda}{\kappa\|c_\rho\|}<1$, and \eqref{ineq:cos} restricts $\theta_x$ to $\left[0,\arccos\left(-\frac{\lambda}{\kappa\|c_\rho\|}\right)\right]$. The equilibrium $\rho$ can be expressed as
\begin{equation}
\label{eqn:equil-cs}
\rho(x)=\begin{cases}
\left(\frac{m-1}{m}\right)^{\frac{1}{m-1}}\left(\lambda+\kappa\|c_\rho\|\cos\theta_x\right)^{\frac{1}{m-1}},\qquad&\text{ if }0\leq\theta_x\leq \arccos\left(-\frac{\lambda}{\kappa\|c_\rho\|}\right),\\[5pt]
0,\qquad&\text{ otherwise}.
\end{cases}
\end{equation}

Again, $\lambda$ and $\|c_\rho\|$ need to be found. Denote by $\phi = \arccos\left(-\frac{\lambda}{\kappa\|c_\rho\|}\right)$. Then, similar to how system \eqref{eqn:system-fs-mg1} was derived, we find in this case:
\begin{equation}
\label{eqn:system-cs-mg1}
\begin{aligned}
1&=\dm w_\dm\int_0^\phi \left(\frac{m-1}{m}\right)^{\frac{1}{m-1}}\left(\lambda+\kappa\|c_\rho\|\cos\theta\right)^{\frac{1}{m-1}}\sin^{\dm-1}\theta \, \d\theta, \\
\|c_\rho\|& =\dm w_\dm\int_0^\phi \left(\frac{m-1}{m}\right)^{\frac{1}{m-1}}\left(\lambda+\kappa\|c_\rho\|\cos\theta\right)^{\frac{1}{m-1}}\sin^{\dm-1}\theta \cos\theta \, \d\theta,
\end{aligned}
\end{equation}
which has to be solved for $\lambda$ and $\|c_\rho\|$.

\begin{remark}
\label{rmk:linear-diff}
We note that for the free Onsager functional with linear diffusion, all critical points are fully supported on the sphere \cite{fatkullin2005critical, FrouvelleLiu2012}. In considering nonlinear diffusion, the set of admisible equilibria is richer, which leads to more sophisticated phase transitions, in particular in the case $m>2$.
\end{remark}


\section{Stability of the uniform distribution}
\label{sect:unif-distr}

In this section we will study the full nonlinear stability of the uniform distribution $\rhou$ given by \eqref{eqn:rho-uni}. Consider a general family of perturbations $\{ \rho^\epsilon \} \subset \calP_{ac}(\bbs^\dm) $ of the uniform distribution. In particular, we have 
\begin{equation}
\label{rho-eps-cond}
\rho^0 = \rhou, \qquad \int_{\bbs^\dm} \rhoeps(x) \dx  = 1, \quad \text{ for all } \epsilon.
\end{equation}
The energy of $\rho^\epsilon$ is given by (use \eqref{eqn:energy-cmform}):
\[
E[\rhoeps]= \frac{1}{m-1}\int_{\bbs^\dm}\rhoeps(x)^m \dx-\frac{\kappa}{2}\iint_{\bbs^{\dm}\times \bbs^{\dm}}\langle x, y\rangle \rhoeps(x)\rhoeps(y)\dS(x)\dS(y)+\frac{\kappa}{2}.
\]
The uniform distribution is a local minimizer provided 
\[
\frac{\d^2}{\d \epsilon^2} E[\rhoeps]_{\big|_{\epsilon = 0}}\geq 0,
\]
for all perturbations $\rho^\epsilon$ that satisfy \eqref{rho-eps-cond}.

Denote 
\[
\delta \rhoeps = \frac{\d }{\d \epsilon} \rhoeps, \qquad \text{ and } \qquad \delta^2 \rhoeps = \frac{\d^2 }{\d \epsilon^2} \rhoeps.
\]
Note that by the unit mass condition in \eqref{rho-eps-cond}, we have
\begin{equation}
\label{eqn:zero-int}
\int_{\bbs^\dm} \delta \rhoeps(x) \dx  = 0, \qquad \text{ and } \qquad \int_{\bbs^\dm} \delta^2 \rhoeps(x) \dx  = 0, \qquad \textrm{ for all } \epsilon.
\end{equation}

Now compute:
\begin{align*}
\frac{\d}{\d \epsilon} E[\rhoeps] &= \frac{1}{m-1}\int_{\bbs^\dm} m \rhoeps(x)^{m-1} \delta \rhoeps(x) \dx -\frac{\kappa}{2}\iint_{\bbs^\dm\times \bbs^\dm} \langle x,y \rangle \left( \delta \rhoeps(x)\rhoeps(y) + \rhoeps(x) \delta \rhoeps(y) \right) \dx \dy \\
&= \frac{1}{m-1}\int_{\bbs^\dm} m \rhoeps(x)^{m-1} \delta \rhoeps(x) \dx -\kappa \iint_{\bbs^\dm\times \bbs^\dm} \langle x,y \rangle \, \rhoeps(x) \delta \rhoeps(y) \dx \dy,
\end{align*}
and
\begin{equation}
\label{eqn:d2Eeps}
\begin{aligned}
\frac{\d^2}{\d \epsilon^2} E[\rhoeps] &= m \int_{\bbs^\dm} \rhoeps(x)^{m-2} \left( \delta \rhoeps(x)\right)^2 \dx + \frac{m}{m-1}\int_{\bbs^\dm} \rhoeps(x)^{m-1} \delta^2 \rhoeps(x) \dx \\[5pt]
& \quad -\kappa \iint_{\bbs^\dm\times \bbs^\dm} \langle x,y \rangle (\delta \rhoeps(x) \delta \rhoeps(y) + \rhoeps(x) \delta^2 \rhoeps(y) ) \dx \dy
\end{aligned}
\end{equation}

Evaluate \eqref{eqn:d2Eeps} at $\epsilon = 0$.  Using that $\rho^0(x)$ is constant and that $\delta^2 \rhoeps(x)$ integrates to $0$, we find that the second term in the r.h.s. vanishes:
\[
\int_{\bbs^\dm} \rho^0(x)^{m-1} \delta^2 \rho^0(x) \dx = \frac{1}{|\bbs^\dm |^{m-1}}  \int_{\bbs^\dm} \delta^2 \rho^0(x) \dx = 0.
\]

Also,
\[
\int_{\bbs^\dm} x \rho^0(x) \dx  = 0,
\]
and hence,
\begin{align*}
\iint_{\bbs^\dm\times \bbs^\dm} \langle x,y \rangle \rho^0(x) \delta^2 \rho^0(y) \dx \dy &=\bigg \langle \int_{\bbs^\dm} x \rho^0(x) \dx, \int_{\bbs^\dm} y \, \delta^2 \rho^0(x) \dy \bigg \rangle =0.
\end{align*}

Now drop the superscript $0$ and denote $\delta \rho = \delta \rhoeps_{\, | \epsilon = 0}$. We find from \eqref{eqn:d2Eeps}:
\begin{equation}
\label{eqn:d2Eeps0}
\begin{aligned}
\frac{\d^2}{\d \epsilon^2} E[\rhoeps]_{\big| {\epsilon = 0}}&= \frac{m}{|\bbs^\dm|^{m-2}} \int_{\bbs^\dm}  \left( \delta \rho(x) \right)^2 \dx -\kappa \left \| \int_{\bbs^\dm} x \delta \rho(x) \dx \right \|^2.
\end{aligned}
\end{equation}

For the uniform density to be a local minimum, the expression above has to be positive, i.e., 
\[
\kappa < \frac{m}{|\bbs^\dm|^{m-2}} \cdot \frac{\int_{\bbs^\dm} \left( \delta \rho(x) \right)^2 \dx}{\left | \int_{\bbs^\dm} x \delta \rho(x) \dx \right|^2},
\]
for all perturbations $\delta \rho$ that have zero mass.

Define the following functional on $\calA = \{ \psi \in L^2(\bbs^\dm): \int_{\bbs^\dm} \psi(x) \dx = 0, \psi \nequiv 0\}$:
\begin{equation}
\label{eqn:calF}
\calF [\psi] = \frac{\int_{\bbs^\dm} \left( \psi(x) \right)^2 \dx}{\left | \int_{\bbs^\dm} x \psi(x) \dx \right|^2}.
\end{equation}
The uniform distribution is a local minimum (and hence, stable) provided
\begin{equation}
\label{eqn:cond-kappa}
\kappa < \frac{m}{|\bbs^\dm|^{m-2}} \cdot \inf_{\psi \in \calA} \calF[\psi],
\end{equation}
and unstable otherwise. 

\begin{lemma}
\label{lem:minF}
Let $\mathcal{F}: \calA \to \bbr$ be given by \eqref{eqn:calF}. Then,
\[
\inf_{\psi \in \calA} \calF[\psi] =\frac{\dm+1}{\bbs^\dm}.
\]  
\end{lemma}
\begin{proof} 
We will present a direct proof, though this result can also be inferred from the known facts on the free energy with linear diffusion -- see Remark \ref{rmk:minF-linear}.

By Cauchy--Schwarz inequality, we see immediately that $\calF$ is bounded below by $1$. Also, $\calF$ is lower-semicontinuous and $\bbs^\dm$ is compact, from which we infer that $\calF$ admits a global minimizer in $\calA$.

Denote by $s_\psi$ the centre of mass of $\psi$:
\[
s_\psi = \int_{\bbs^\dm} x \psi(x) \dx.
\]
By calculating the first variation and setting it to zero, i.e., 
\[
\d \calF [\psi](\delta \psi) = 0,
\]
we find the following Euler-Lagrange equation for the critical points:
\[
\int_{\bbs^\dm} \left( \psi(x) |s_\psi|^2 - \left( \int_{\bbs^\dm} \psi^2(y) \dy \right) s_\psi \cdot x \right) \delta \psi(x) \dx =0,
\]
for all perturbations $\delta \psi \in \calA$. From here we find
\begin{equation}
\label{eqn:EL-Fcal}
\psi(x) |s_\psi|^2 - \left( \int_{\bbs^\dm} \psi^2(y) \dy \right) s_\psi \cdot x = \lambda,
\end{equation}
for a constant $\lambda$. By integrating this equation over $\bbs^\dm$ on both sides, and using that $\psi$ integrates to $0$, we find $\lambda =0$.

Consider a critical point $\bar{\psi}$ and denote $C:= \int_{\bbs^\dm} \bar{\psi}^2(y) \d y$. From \eqref{eqn:EL-Fcal} (where $\lambda = 0$), we can write
\[
\bar{\psi}(x) = \frac{C}{|s_{\bar{\psi}}|^2} \; s_{\bar{\psi}} \cdot x.
\]
By taking the square of this expression and integrating, we find
\[
C = \frac{|s_{\bar{\psi}}|^4}{\int_{\bbs^\dm} (s_{\bar{\psi}} \cdot x)^2  \dx},
\]
and hence,
\begin{equation}
\label{eqn:barpsi}
\bar{\psi}(x) = \frac{|s_{\bar{\psi}}|^4}{\int_{\bbs^\dm} (s_{\bar{\psi}} \cdot x)^2  \dx} \; s_{\bar{\psi}} \cdot x.
\end{equation}

All the critical points of $\calF$ are in the form above.  Evaluate $\calF$ at such $\bar{\psi}$ to find
\begin{align*}
\calF[\bar{\psi}] &= \frac{C}{|s_{\bar{\psi}}|^2} \\
& = \frac{|s_{\bar{\psi}}|^2}{\int_{\bbs^\dm} (s_{\bar{\psi}} \cdot x)^2  \dx}.
\end{align*}

Assume by symmetry that $s_{\bar{\psi}}$ is in the direction of the fixed unit vector $x_0 \in \bbs^\dm$, and use the hyperspherial coordinates used previously ($\theta$ denotes the angle between $x$ and $s_{\bar{\psi}}$). Then we get:
\[
\calF[\bar{\psi}] = \frac{1}{|\bbs^{\dm-1}| \int_0^\pi \cos^2 \theta  \sin^{\dm-1} \theta \d \theta}.
\]
The expression above can be simplified by using:
\begin{equation}
\label{eqn:cossin-rel}
\int_0^\pi \cos^2 \theta  \sin^{\dm-1} \theta \d \theta = \frac{1}{\dm +1} \int_0^\pi \sin^{\dm-1} \theta \d\theta,
\end{equation}
and 
\[
|\bbs^{\dm}| = |\bbs^{\dm-1}| \int_0^\pi \sin^{\dm-1} \theta\d\theta,
\]
to write
\begin{equation}
\label{eqn:Fmin}
\calF[\bar{\psi}] = \frac{\dm+1}{|\bbs^{\dm}|}.
\end{equation}

Note that $\calF$ is invariant under scalar multiplication $\psi \to \lambda \psi$, which is reflected in the fact that the value of $\calF[\bar{\psi}]$ does not depend on the magnitude of $s_{\bar{\psi}}$. Put it differently, up to symmetry and a scalar multiplication, the critical points in the form \eqref{eqn:barpsi} are unique. Since $\calF$ admits a global minimizer, $\bar{\psi}$ must necessarily achieve this mimimum value. 
\end{proof}

\begin{remark}
\label{rmk:minF-linear} Lemma \ref{lem:minF} also follows directly from the known results for linear diffusion. Indeed, consider the free energy with linear diffusion, which can be simplified and written as
\[
E_1[\rho] =\int_{\bbs^\dm}\rho(x)\log \rho(x)\dx-\frac{\kappa}{2}\left\|\int_{\bbs^\dm}x\rho(x)\dx\right\|^2+\frac{\kappa}{2}.
\]

By taking perturbations $\rho^\epsilon$ of the uniform distribution and performing a similar calculation as for $m \neq 1$, we calculate
\[
\frac{\d^2}{\d \epsilon^2} E_1[\rhoeps]_{\big| {\epsilon = 0}} = \frac{1}{|\bbs^\dm|^{-1}} \int_{\bbs^\dm}  \left( \delta \rho(x) \right)^2 \dx -\kappa \left \| \int_{\bbs^\dm} x \delta \rho(x) \dx \right \|^2,
\]
which is in fact \eqref{eqn:d2Eeps0} for $m=1$. Since the uniform distribution is a global energy minimizer for $\kappa \leq \dm +1$ \cite{fatkullin2005critical, FrouvelleLiu2012}, the second variation calculated above needs to be non-negative, for all $\delta \rho \in \calA$. Hence, $| \bbs^\dm | \mathcal{F}[\delta \rho]\geq \kappa$, for all $\delta \rho \in \calA$ and $\kappa \geq \dm+1$, from which we conclude
\[
\inf_{\psi \in \calA} \mathcal{F}[\psi] \geq \frac{\dm+1}{|\bbs^\dm|}.
\]
\end{remark}

The considerations above lead to the following result.
\begin{proposition}[Stability of the uniform distribution]
\label{prop:stab-unif}
The uniform distribution $\rhou$ is a stable critical point of the energy functional $E$ if and only if $\kappa < \kappa_1$, where
\begin{equation}
\label{eqn:kappa1}
\kappa_1:= \frac{m (\dm+1)}{|\bbs^\dm|^{m-1}}.
\end{equation}
\end{proposition}
\begin{proof}
The proof follows from the previous calculations, by using Lemma \ref{lem:minF} in \eqref{eqn:cond-kappa}. 
\end{proof}

\begin{remark}
\label{rmk:lin-consist}
We note that in the limit $m \searrow 1$, Proposition \ref{prop:stab-unif} recovers the stability result from the classical Onsager free energy with linear diffusion \cite{fatkullin2005critical,FrouvelleLiu2012}. Indeed, in the case of linear diffusion, the uniform distribution is stable for $\kappa<\dm +1$, and unstable otherwise. In the context of polymer orientation \cite{fatkullin2005critical,FrouvelleLiu2012}, this critical interaction strength indicates the phase transition from isotropic to nematic states.
\end{remark}



In the next sections, we investigate the existence and bifurcations of critical points of the energy. We will consider the following cases separately i) $1<m<2$, ii) $m=2$, and iii) $m>2$. The case $m=2$ is considered separately, as it constitutes a degenerate case. The range $m>2$ is very different qualitatively, and it is only considered for $\dm =2$.

\section{Case $1<m<2$: Critical values of $\kappa$ and energy minimizers}
\label{sect:bifurcations-mg1}


In this section we investigate the range $1<m<2$. We identify two critical values of $\kappa$, which we present in order. We then investigate the energy minimizers and also provide some numerical results.

\subsection{Bifurcation from the uniform distribution} 
\label{subsect:fs}
We first study a bifurcation that occurs at $\kappa = \kappa_1$ (see Proposition \ref{prop:stab-unif}). Consider the fully supported steady state \eqref{eqn:equil-fs}, with $\lambda$ and $\|c_\rho\|$ that solve \eqref{eqn:system-fs-mg1}. Recall that in this case, $\lambda\geq \kappa\|c_\rho\|$ holds.

For simpler notation, substitute 
\begin{equation}
\label{eqn:slambda}
  s=\|c_\rho\|, \qquad \lambda=-\kappa s\eta.
\end{equation}
To guarantee the condition on $\lambda$, we assume $\eta\leq -1$. Then, \eqref{eqn:system-fs-mg1} becomes
\begin{align*}
1&=\dm w_\dm\left(\frac{\kappa s (m-1)}{m}\right)^{\frac{1}{m-1}}\int_0^\pi (\cos\theta-\eta)^{\frac{1}{m-1}}\sin^{\dm-1}\theta\d\theta,\\
s&=\dm w_\dm\left(\frac{\kappa s (m-1)}{m}\right)^{\frac{1}{m-1}}\int_0^\pi (\cos\theta-\eta)^{\frac{1}{m-1}}\sin^{\dm-1}\theta\cos\theta\d\theta.
\end{align*}
From the above, we can express $s$ and $\kappa$ in terms of $\eta$ as follows:
\begin{subequations}
\label{eqn:skappa-eta}
\begin{align}
 s&=\frac{\int_0^\pi (\cos\theta-\eta)^{\frac{1}{m-1}}\sin^{\dm-1}\theta \cos\theta\d\theta}{\int_0^\pi (\cos\theta-\eta)^{\frac{1}{m-1}}\sin^{\dm-1}\theta\d\theta}, \label{eqn:s-eta} \\[2pt]
 \kappa^{-1}&=\displaystyle \frac{m-1}{m} (\dm w_\dm)^{m-1} \left(\int_0^\pi (\cos\theta-\eta)^{\frac{1}{m-1}}\sin^{\dm-1}\theta \cos\theta\d\theta\right)
\left(\int_0^\pi (\cos\theta-\eta)^{\frac{1}{m-1}}\sin^{\dm-1}\theta\d\theta\right)^{m-2}. \label{eqn:kappa-eta}
\end{align}
\end{subequations}

Define the following function for $\eta\leq -1$:
\begin{equation}
\label{eqn:H}
H(\eta)= \frac{m-1}{m} (\dm w_\dm)^{m-1} \left(\int_0^\pi (\cos\theta-\eta)^{\frac{1}{m-1}}\sin^{\dm-1}\theta \cos\theta\d\theta\right)
\left(\int_0^\pi (\cos\theta-\eta)^{\frac{1}{m-1}}\sin^{\dm-1}\theta\d\theta\right)^{m-2}.
\end{equation}
With this definition, \eqref{eqn:kappa-eta} can be written as 
\[
H(\eta) = \kappa^{-1}.
\]

\begin{lemma}
\label{lem:H-monotone}
The function $H(\eta)$ given by \eqref{eqn:H} is decreasing for $1<m<2$,  constant for $m=2$, and increasing for $m>2$.    
\end{lemma}

\begin{proof}
The proof is based on a technical and fairly involved result that uses associated Legendre functions of the first kind and the integral form of the Chebyshev inequality. This technical result is stated and proved in Appendix \ref{appendix:H-monotone} -- see Lemma \ref{sign-of-functionalH}. The reader should first check Appendix \ref{appendix:H-monotone} before continue reading. 

We first rewrite the first factor of $H(\eta)$ from \eqref{eqn:H} using integration by parts, to get
\begin{align*}
&\int_0^\pi (\cos\theta-\eta)^{\frac{1}{m-1}}\sin^{\dm-1}\theta \cos\theta\d\theta\\
=&\left[(\cos\theta-\eta)^{\frac{1}{m-1}}\times\frac{1}{\dm}\sin^{\dm}\theta\right]_0^\pi-\int_0^\pi\frac{1}{m-1}(\cos\theta-\eta)^{\frac{1}{m-1}-1}(-\sin\theta)\times\frac{1}{\dm}\sin^\dm\theta\d\theta\\
=&\frac{1}{\dm(m-1)}\int_0^\pi(\cos\theta-\eta)^{\frac{1}{m-1}-1}\sin^{\dm+1}\theta\d\theta.
\end{align*}
Therefore, we can express $H(\eta)$ as 
\[
H(\eta)=\frac{1}{m\dm}(\dm w_\dm)^{m-1} \left(\int_0^\pi (\cos\theta-\eta)^{\frac{1}{m-1}-1}\sin^{\dm+1}\theta\d\theta\right)
\left(\int_0^\pi (\cos\theta-\eta)^{\frac{1}{m-1}}\sin^{\dm-1}\theta\d\theta\right)^{m-2}.
\]

Then, for any $\eta<-1$, we have
\begin{equation}
\label{eqn:HpoverH}
\begin{aligned}
\frac{H'(\eta)}{H(\eta)} &=-\frac{2-m}{m-1}\left(\frac{\int_0^\pi (\cos\theta-\eta)^{\frac{1}{m-1}-2}\sin^{\dm+1}\theta \d\theta}{\int_0^\pi (\cos\theta-\eta)^{\frac{1}{m-1}-1}\sin^{\dm+1}\theta\d\theta}\right) \\
&\quad -\frac{m-2}{m-1}\left(\frac{\int_0^\pi (\cos\theta-\eta)^{\frac{1}{m-1}-1}\sin^{\dm-1}\theta\d\theta}{\int_0^\pi (\cos\theta-\eta)^{\frac{1}{m-1}}\sin^{\dm-1}\theta\d\theta}\right)\\
&=\left(\frac{m-2}{m-1}\right)\left(\frac{\int_0^\pi (\cos\theta-\eta)^{\frac{1}{m-1}-2}\sin^{\dm+1}\theta \d\theta}{\int_0^\pi (\cos\theta-\eta)^{\frac{1}{m-1}-1}\sin^{\dm+1}\theta\d\theta}-\frac{\int_0^\pi (\cos\theta-\eta)^{\frac{1}{m-1}-1}\sin^{\dm-1}\theta\d\theta}{\int_0^\pi (\cos\theta-\eta)^{\frac{1}{m-1}}\sin^{\dm-1}\theta\d\theta}\right).
\end{aligned}
\end{equation}
Using the notation \eqref{eqn:Hmd} in Appendix \ref{appendix:H-monotone}, we can further simplify it as
\[
\frac{H'(\eta)}{H(\eta)}=-\left(\frac{m-2}{m-1}\right)\frac{\mathcal{H}_{m, \dm}(-\eta)}{\left(\int_0^\pi (\cos\theta-\eta)^{\frac{1}{m-1}-1}\sin^{\dm+1}\theta\d\theta\right)\left(\int_0^\pi (\cos\theta-\eta)^{\frac{1}{m-1}}\sin^{\dm-1}\theta\d\theta\right)}.
\]
Finally, we apply Lemma \ref{sign-of-functionalH} in Appendix \ref{appendix:H-monotone} to get
\begin{align*}
&\mathrm{sgn}\left(\frac{H'(\eta)}{H(\eta)} \right) \\
&\quad =\mathrm{sgn}\left(
-\left(\frac{m-2}{m-1}\right)\frac{\mathcal{H}_{m, \dm}(-\eta)}{\left(\int_0^\pi (\cos\theta-\eta)^{\frac{1}{m-1}-1}\sin^{\dm+1}\theta\d\theta\right)\left(\int_0^\pi (\cos\theta-\eta)^{\frac{1}{m-1}}\sin^{\dm-1}\theta\d\theta\right)}
\right)\\
&\quad =\mathrm{sgn}\left(\left(\frac{m-2}{m-1}\right)\left(\frac{1}{m-1}+\frac{\dm-1}{2}\right)\right).
\end{align*}
Here, we used that both $\int_0^\pi (\cos\theta-\eta)^{\frac{1}{m-1}-1}\sin^{\dm+1}\theta\d\theta$ and $\int_0^\pi (\cos\theta-\eta)^{\frac{1}{m-1}}\sin^{\dm-1}\theta\d\theta$ are positive. From $m>1$, we also get $m-1>0$ and $\frac{1}{m-1}+\frac{\dm-1}{2}>0$. Therefore, also using that $H(\eta)$ is positive, we find
\[
\mathrm{sgn}\left(H'(\eta)\right)=\mathrm{sgn}(m-2),
\]
which proves the result.
\end{proof}


\begin{lemma}
\label{lem:Hlim}
The following holds:
\[
\Bigl( \lim_{\eta \to -\infty } H(\eta)\Bigr)^{-1} = \kappa_1,
\]
where $\kappa_1$ is the critical $\kappa$ identified in Proposition \ref{prop:stab-unif}. 
\end{lemma}
\begin{proof}
The proof is based on a direct calculation, using Newton's generalized binomial theorem; see Appendix \ref{appendix:Hlim}.   
\end{proof}

\begin{remark}
\label{rmk:kappac}
When $\eta\to-\infty$ and $s\to 0$ (along with $\lambda \to \lambda_{\text{uni}}$ defined in \eqref{eqn:lambda-uni}) one finds the uniform distribution.  Therefore, at $\kappa = \kappa_1$, a fully supported equilibrium $\rho_\kappa$ gets born from the uniform distribution $\rhou$. 
\end{remark}

Since the domain of $H(\eta)$ is $(-\infty,-1]$, and the first critical value $\kappa_1$ was derived from the behavior as $\eta \to -\infty$, it is natural to investigate the following second critical value of $\kappa$ defined as
\begin{equation}
\label{eqn:kappa2}
\kappa_{2}:= (H(-1))^{-1}.
\end{equation}
Note that since $H(\eta)$ is decreasing when $1<m<2$, we have $\kappa_1<\kappa_2$. Also, by Lemmas \ref{lem:H-monotone} and \ref{lem:Hlim}, for any $\kappa \in (\kappa_{1},\kappa_{2})$, there exists a unique $\eta_\kappa <-1$ such that $\kappa^{-1} = H(\eta_\kappa)$ -- see equation \eqref{eqn:kappa-eta} and Figure \ref{fig:HF}(a). With $\eta$ being known, one then finds the corresponding $s_\kappa$ from \eqref{eqn:s-eta}, and $\lambda_\kappa = -\kappa s_\kappa \eta_\kappa$ by \eqref{eqn:slambda}. All together, they determine uniquely a fully supported equilibrium in the form \eqref{eqn:equil-fs}. 

\begin{figure}[!htbp]
 \begin{center}
 \begin{tabular}{cc}
 \includegraphics[width=0.48\textwidth]{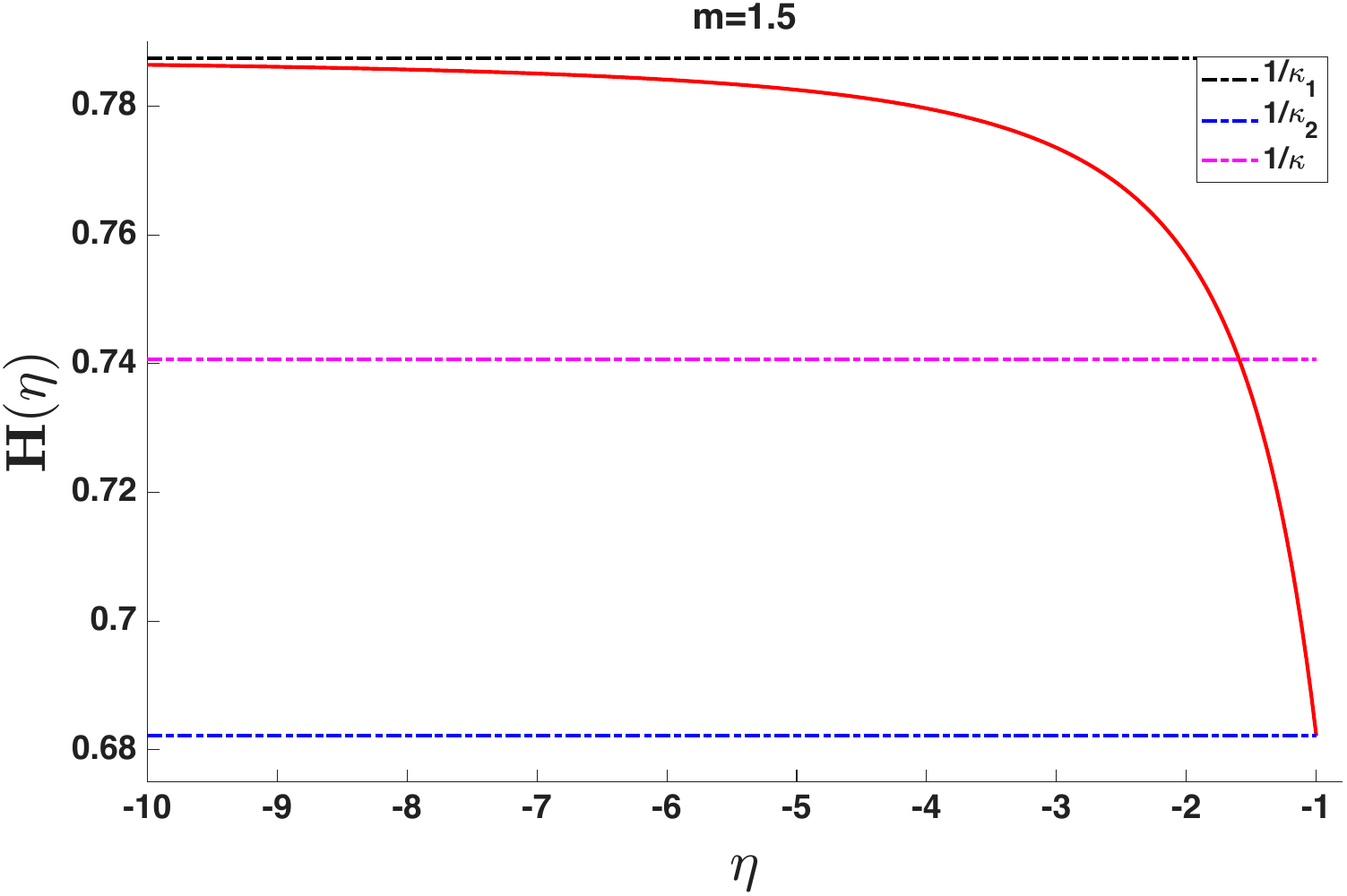} & 
 \includegraphics[width=0.48\textwidth]{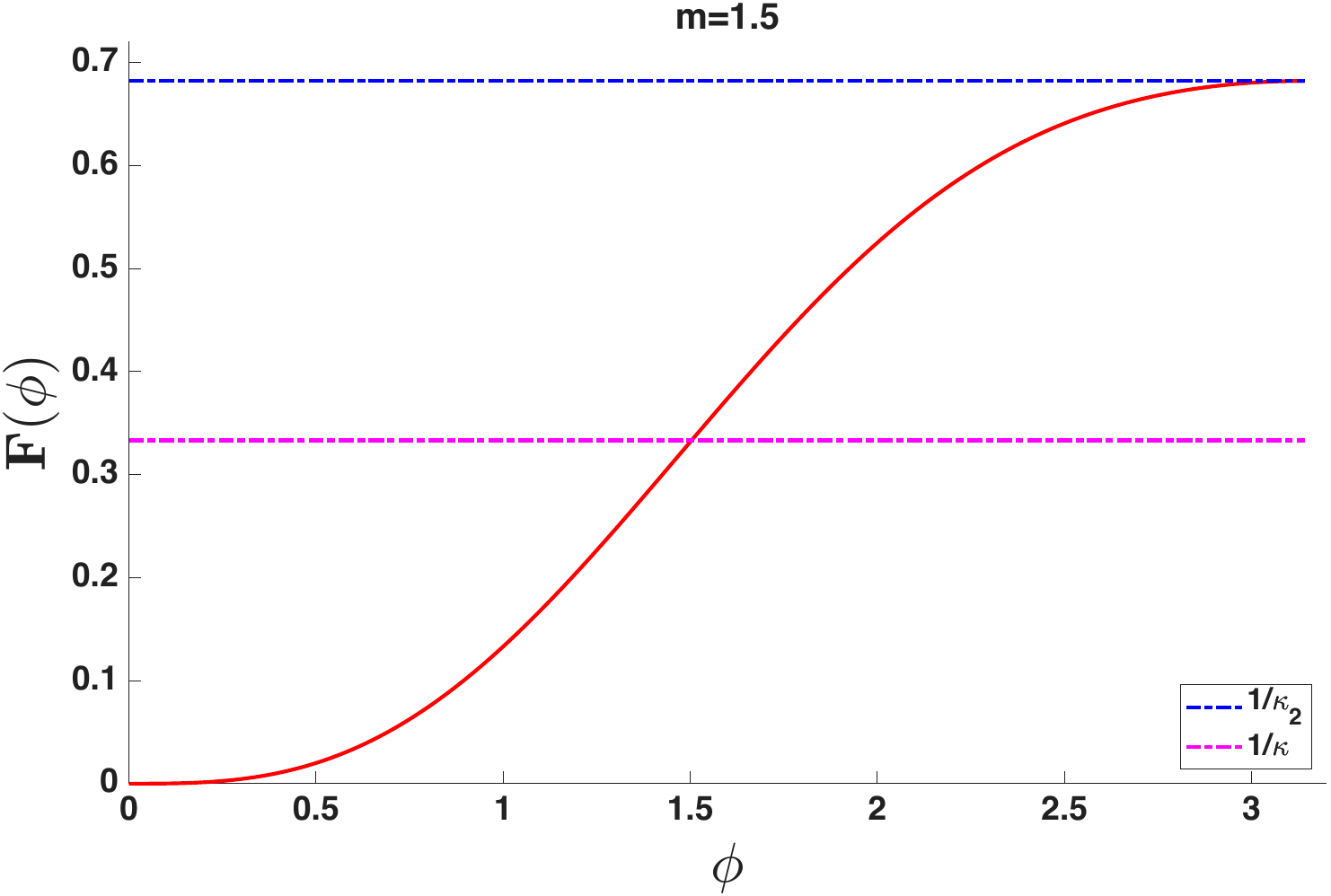} \\
 (a) & (b)
\end{tabular}
\caption{Case $1<m<2$. (a) Plot of function $H$ defined in \eqref{eqn:H}. For any $\kappa_1<\kappa<\kappa_2$, there exists a unique $\eta_\kappa <-1$ such that $\kappa^{-1} = H(\eta_\kappa)$ -- see equation \eqref{eqn:kappa-eta}. (b) Plot of function $F$ defined in \eqref{eqn:F}. For any $\kappa_2<\kappa<\infty$, there exists a unique $\phi_\kappa \in (0,\pi)$ that solves $\kappa^{-1} = F(\phi_\kappa)$ -- see equation \eqref{eqn:kappa-phi}. For both plots, $m=1.5$ and $\dm =2$.}
\label{fig:HF}
\end{center}
\end{figure}


\subsection{Second critical value $\kappa=\kappa_2$} 
\label{subsect:cs}
We now study what happens at the critical value $\kappa = \kappa_2$, with $\kappa_2$ defined in \eqref{eqn:kappa2}.  

Consider the equilibria with support strictly contained in $\bbs^\dm$ -- see \eqref{eqn:equil-cs}. 
We will continue using the notation $s=\|c_\rho\|$. Also recall the notation $\displaystyle \phi=\arccos(-\lambda/(\kappa\|c_\rho\|))$, so for strictly supported equilibria we have
\begin{equation}
\label{eqn:lambdaphi}  
   \lambda=-\kappa s \cos\phi.
\end{equation}
Then, we write system \eqref{eqn:system-cs-mg1} as
\begin{align*}
\displaystyle 1&=\dm w_\dm\left(\frac{\kappa s (m-1)}{m}\right)^{\frac{1}{m-1}}\int_0^\phi(\cos\theta-\cos\phi)^{\frac{1}{m-1}}\sin^{\dm-1}\theta \d\theta,\\
\displaystyle s&=\dm w_\dm\left(\frac{\kappa s (m-1)}{m}\right)^{\frac{1}{m-1}}\int_0^\phi(\cos\theta-\cos\phi)^{\frac{1}{m-1}}\sin^{\dm-1}\theta \cos\theta \d\theta,
\end{align*}
from which we can express $s$ and $\kappa$ in terms of $\phi$ as follows:
\begin{subequations}
\label{eqn:skappa-phi}
\begin{align}
 s&=\displaystyle\frac{\int_0^\phi (\cos\theta-\cos\phi)^{\frac{1}{m-1}}\sin^{\dm-1}\theta \cos\theta\d\theta}{\int_0^\phi (\cos\theta-\cos\phi)^{\frac{1}{m-1}}\sin^{\dm-1}\theta\d\theta}, \label{eqn:s-phi}\\[5pt]
 \kappa^{-1}&=\displaystyle \frac{m-1}{m} (\dm w_\dm)^{m-1} \left(\int_0^\phi (\cos\theta-\cos\phi)^{\frac{1}{m-1}}\sin^{\dm-1}\theta \cos\theta\d\theta\right)\left(\int_0^\phi (\cos\theta-\cos\phi)^{\frac{1}{m-1}}\sin^{\dm-1}\theta \d\theta\right)^{m-2}. \label{eqn:kappa-phi}
\end{align}
\end{subequations}

\begin{lemma}
\label{lem:F-monotone}
Let $1<m<2$ be fixed, and define the following function on $0<\phi<\pi$:
\begin{equation}
\label{eqn:F}
F(\phi)= \frac{m-1}{m} (\dm w_\dm)^{m-1} \left(\int_0^\phi (\cos\theta-\cos\phi)^{\frac{1}{m-1}}\sin^{\dm-1}\theta \cos\theta\d\theta\right)\left(\int_0^\phi (\cos\theta-\cos\phi)^{\frac{1}{m-1}}\sin^{\dm-1}\theta \d\theta\right)^{m-2}.
\end{equation}
Then, $F$ is an increasing function on $(0,\pi)$.
\end{lemma}

\begin{proof}
The proof is based on calculating $F'$ and assessing its sign; see Appendix \ref{appendix:F-monotone}.
\end{proof}

Equation \eqref{eqn:kappa-phi} can be writen as $F(\phi) = \kappa^{-1}$. First note that $F(\pi) = H(-1) = \kappa_2^{-1}$ (see \eqref{eqn:H} and \eqref{eqn:kappa2}). Also, $F(0)=0$ as we can write
\[
F(\phi)= \text{const} \times \frac{\int_0^\phi (\cos\theta-\cos\phi)^{\frac{1}{m-1}}\sin^{\dm-1}\theta \cos\theta\d\theta}{\int_0^\phi (\cos\theta-\cos\phi)^{\frac{1}{m-1}}\sin^{\dm-1}\theta \d\theta} \, \left(\int_0^\phi (\cos\theta-\cos\phi)^{\frac{1}{m-1}}\sin^{\dm-1}\theta \d\theta\right)^{m-1},
\]
where the quotient factor in the r.h.s is bounded in absolute value by $1$, while the other factor is $0$ at $\phi =0$. Therefore, since $F$ is strictly increasing, we infer that for any $\kappa \in(\kappa_2,\infty)$, there exists a unique $0<\phi_\kappa<\pi$ that solves $F(\phi_\kappa) = \kappa^{-1}$ -- see Figure \ref{fig:HF}(b), and hence it solves \eqref{eqn:kappa-phi}.  Corresponding to $\phi_\kappa$, we calculate $s_\kappa$ from \eqref{eqn:s-phi}, and then have $\lambda_\kappa = -\kappa s_\kappa \cos \phi_\kappa$ by \eqref{eqn:lambdaphi}. Together they determine a unique equilibrium in the form \eqref{eqn:equil-cs}. 

By combining these considerations with the finding from Section \ref{subsect:fs}, we arrive at the following result.

\begin{proposition}[Critical $\kappa$ and equilibria for $1<m<2$]
\label{prop:bif-mg1}
For any $1<m<2$ and $\dm\geq 1$, there exist two critical values of $\kappa$ ($\kappa_1$ and $\kappa_2$ defined by \eqref{eqn:kappa1} and \eqref{eqn:kappa2}, respectively) such that:

\noindent i) At $\kappa=\kappa_1$, an equilibrium fully supported on $\bbs^\dm$ bifurcates from the uniform distribution $\rhou$. For any $\kappa \in (\kappa_1,\kappa_2)$, there exists a unique equilibrium in the form \eqref{eqn:equil-fs}, given by
\begin{equation}
\label{eqn:rhok-fs}
\rho_\kappa(x)=\left(\frac{m-1}{m}\right)^{\frac{1}{m-1}}\left(\lambda_\kappa+\kappa s_\kappa \cos\theta_x\right)^{\frac{1}{m-1}},\qquad\forall x\in \bbs^\dm,
\end{equation}
where $s_\kappa$ and $\lambda_\kappa$ are uniquely determined by $\kappa$ -- see \eqref{eqn:slambda} and \eqref{eqn:skappa-eta}. 
\smallskip

\noindent ii) At $\kappa=\kappa_2$, the fully supported equilibria from part i) turn into equilibria with support strictly contained in $\bbs^\dm$.  For any $\kappa \in (\kappa_2,\infty)$, there exits a unique equilibrium in the form \eqref{eqn:equil-cs}, given by
\begin{equation}
\label{eqn:rhok-cs}
\rho_\kappa(x)=\begin{cases}
\left(\frac{m-1}{m}\right)^{\frac{1}{m-1}}\left(\lambda_\kappa+\kappa s_\kappa\cos\theta_x\right)^{\frac{1}{m-1}},\qquad&\text{ if }0\leq\theta_x\leq \arccos\left(-\frac{\lambda_\kappa}{\kappa s_\kappa}\right),\\[5pt]
0,\qquad&\text{ otherwise},
\end{cases}
\end{equation}
with $s_\kappa$ and $\lambda_\kappa$ uniquely determined by $\kappa$ -- see \eqref{eqn:lambdaphi} and \eqref{eqn:skappa-phi}.
\end{proposition}

\begin{remark}
\label{rmk:kappa2}
The second critical value $\kappa_2$ can be computed explicitly in terms of the Gamma function. Indeed. by direct calculations (see Appendix \ref{appendix:kappa2}), we can write $\kappa_2$ as
\begin{equation}
\label{eqn:kappa2-alt}
\kappa_2=\frac{m(1+\dm(m-1))}{(m-1)2^{1+(\dm-1)(m-1)}|\bbs^\dm|^{m-1}}\left(\frac{\Gamma\left(\frac{1}{2}\right) \Gamma\left(\frac{1}{m-1}+\dm\right)}{\Gamma\left(\frac{1}{m-1}+\frac{\dm}{2}\right)\Gamma\left(\frac{\dm+1}{2}\right)}\right)^{m-1}.
\end{equation}
By the log-convexity of the Gamma function, we have 
\[
\log\Gamma\left(\frac{1}{2}\right)+\log\Gamma\left(\frac{1}{m-1}+\dm\right)>\log\Gamma\left(\frac{1}{m-1}+\frac{\dm}{2}\right)+\log\Gamma\left(\frac{\dm+1}{2}\right),
\]
for any $\dm\geq1$ and $m>1$, and it implies $ \left(\frac{\Gamma\left(\frac{1}{2}\right)\Gamma\left(\frac{1}{m-1}+\dm\right)}{\Gamma\left(\frac{1}{m-1}+\frac{\dm}{2}\right)\Gamma\left(\frac{\dm+1}{2}\right)}\right)^{m-1} \geq 1$. In the linear diffusion limit we find that 
\[
\lim_{m\searrow 1} \kappa_2 \geq \lim_{m\searrow 1}\frac{m(1+\dm(m-1))}{(m-1)2^{1+(\dm-1)(m-1)}|\bbs^\dm|^{m-1}}=\infty.
\]
This is consistent with the results with linear diffusion \cite{fatkullin2005critical,FrouvelleLiu2012}, where equilibria have full support, and no second critical $\kappa$ exists.
\end{remark}


\subsection{Energy minimizers}
\label{subsect:energy-mg1}

Next, we will show that the equilibria $\rho_\kappa$ found in Proposition \ref{prop:bif-mg1} are more energetically favourable than the uniform distribution $\rhou$, for every $\kappa>\kappa_1$ fixed.  We caution the reader that for a fixed $\kappa$, the energy of the uniform distribution $E[\rhou]$ depends on $\kappa$ (see \eqref{eqn:energy-s}), though this dependence is not explicit in our notations. 

The formal result is given by the following theorem.
\begin{theorem}[Global energy minimizers for $1<m<2$]
\label{thm:gminimizers} 

Let $1<m<2$ and $\dm\geq 1$. The global minimizer of \eqref{energy-sphere} on $\calP_{ac}(\bbs^\dm)$ is 
\begin{enumerate}
\item the uniform distribution $\rhou$ when $\kappa<\kappa_1$, \\[-7pt]
\item the equilibrium $\rho_{\kappa}$ given by \eqref{eqn:rhok-fs} (for $\kappa_1< \kappa<\kappa_2$) or by \eqref{eqn:rhok-cs} (for $\kappa>\kappa_2$).
\end{enumerate}
Moreover, the energy difference $E[\rhou] - E[\rho_{\kappa}]$ increases with $\kappa$ for $\kappa>\kappa_1$. 
\end{theorem}

\begin{proof} When $\kappa<\kappa_1$, the only equilibrium is the uniform distribution and hence, it is the global minimizer. For $\kappa>\kappa_1$, we have two equilibria ($\rhou$ and $\rho_\kappa$) and we will compare directly their energies. We will consider separately the two ranges $\kappa_1<\kappa<\kappa_2$ and $\kappa>\kappa_2$, with equilibria given by \eqref{eqn:rhok-fs} and \eqref{eqn:rhok-cs}, respectively.
\smallskip

{\em Case $\kappa_1<\kappa<\kappa_2$.}  In this case the equlibria $\rho_\kappa$ from \eqref{eqn:rhok-fs} are fully supported in $\bbs^\dm$. We will express $E[\rhou] - E[\rho_{\kappa}]$ in terms of the variable $\eta$. By \eqref{eqn:energy-s} we have
\begin{equation}
\label{eqn:E-diff}
    E[\rhou]-E[\rho_{\kappa}] = 
    \frac{1}{m-1}\int_{\bbs^\dm} \rhou^m(x)  \dS(x)  - \frac{1}{m-1}\int_{\bbs^\dm}\rho_{\kappa}^m(x) \dS(x) + \frac{\kappa}{2} \|c_{\rho_\kappa}\|^2.
\end{equation}
Since $\frac{1}{m-1}\int_{\bbs^\dm} \rhou^m(x) \dS$ is a constant,  consider only 
\[
\frac{\kappa}{2} \|c_{\rho_\kappa}\|^2 - \frac{1}{m-1}\int_{\bbs^\dm}\rho_{\kappa}(x)^m \dS (x).
\]

Recall the notations $s_\kappa = \| c_{\rho_\kappa}\|$ and $\lambda_\kappa = - \kappa s_\kappa \eta_\kappa $. Using \eqref{eqn:s-eta}, \eqref{eqn:kappa-eta} and \eqref{eqn:rhok-fs}, compute
\begin{equation}
\label{eqn:energy-eta}
\begin{aligned}
\frac{\kappa}{2} \|c_{\rho_\kappa}\|^2 - \frac{1}{m-1}\int_{\bbs^\dm} \rho_{\kappa}^m(x) \dS(x) &= \frac{m}{2(m-1)(d\omega_d)^{m-1}} \frac{\int_0^\pi(\cos\theta-\eta_\kappa)^{\frac{1}{m-1}}\sin^{d-1}\theta\cos\theta\d\theta}{(\int_0^\pi(\cos\theta-\eta_\kappa)^{\frac{1}{m-1}}\sin^{d-1}\theta\d\theta)^{m}}\\[3pt]
&\quad-\frac{1}{(m-1) (\dm w_\dm)^{m-1}}\frac{\int_0^\pi(\cos\theta-\eta_\kappa)^{\frac{m}{m-1}}\sin^{d-1}\theta\d\theta}{(\int_0^\pi(\cos\theta-\eta_\kappa)^{\frac{1}{m-1}}\sin^{d-1}\theta\d\theta)^m}\\[3pt]
&=g_1(\eta_\kappa)g_2(\eta_\kappa),
\end{aligned}
\end{equation}
where we denoted
\begin{align}
    g_1(\eta) &= m\int_0^\pi(\cos\theta-\eta)^{\frac{1}{m-1}}\sin^{d-1}\theta\cos\theta\d\theta -2 \int_0^\pi(\cos\theta-\eta)^{\frac{m}{m-1}}\sin^{d-1}\theta\d\theta,\\[5pt]
    g_2(\eta) &= \frac{1}{2(m-1)(d\omega_d)^{m-1}\left(\int_0^\pi(\cos\theta-\eta)^{\frac{1}{m-1}}\sin^{d-1}\theta\d\theta\right)^m}.
\end{align}

Calculate
\begin{equation}
\label{eqn:g1p}
\begin{aligned}
    g_1'(\eta) &=  -\frac{m}{m-1}\int_0^\pi(\cos\theta-\eta)^{\frac{1}{m-1}-1}\sin^{d-1}\theta\cos\theta\d\theta + 2\frac{m}{m-1}\int_0^\pi(\cos\theta-\eta)^{\frac{1}{m-1}}\sin^{d-1}\theta\d\theta\\
    &= -\frac{m}{m-1}\int_0^\pi(\cos\theta-\eta)^{\frac{1}{m-1}-1}\sin^{d-1}\theta((\cos\theta-\eta)+\eta)\d\theta + 2\frac{m}{m-1}\int_0^\pi(\cos\theta-\eta)^{\frac{1}{m-1}}\sin^{d-1}\theta\d\theta\\
    &= -\frac{\eta m}{m-1}\int_0^\pi(\cos\theta-\eta)^{\frac{1}{m-1}-1}\sin^{d-1}\theta\d\theta + \frac{m}{m-1}\int_0^\pi(\cos\theta-\eta)^{\frac{1}{m-1}-1}\sin^{d-1}\theta(\cos\theta-\eta)\d\theta\\
    &= -\frac{2\eta m}{m-1}\int_0^\pi(\cos\theta-\eta)^{\frac{1}{m-1}-1}\sin^{d-1}\theta\d\theta + \frac{m}{m-1}\int_0^\pi(\cos\theta-\eta)^{\frac{1}{m-1}-1}\sin^{d-1}\theta\cos\theta\d\theta
\end{aligned}
\end{equation}
and
\begin{equation}
\label{eqn:g2p}
g_2'(\eta) = \frac{m}{2(m-1)^2(d\omega_d)^{m-1}}  \frac{\int_0^\pi(\cos\theta-\eta)^{\frac{1}{m-1}-1}\sin^{d-1}\theta\d\theta}{\left(\int_0^\pi(\cos\theta-\eta)^{\frac{1}{m-1}}\sin^{d-1}\theta\d\theta\right)^{m+1}}.
\end{equation}
Combining \eqref{eqn:g1p} and \eqref{eqn:g2p} we then find
\begin{align}
    (g_1 (\eta) g_2(\eta))' &= g_1'(\eta)g_2(\eta)+g_1(\eta)g_2'(\eta)\\[3pt]
    &= \frac{\frac{m}{m-1}\int_0^\pi(\cos\theta-\eta)^{\frac{1}{m-1}-1}\sin^{d-1}\theta\d\theta}{2(m-1)(d\omega_d)^{m-1}\left(\int_0^\pi(\cos\theta-\eta)^{\frac{1}{m-1}}\sin^{d-1}\theta\d\theta\right)^m}\times\\[3pt] 
    & \quad \biggl( -2\eta + \frac{\int_0^\pi(\cos\theta-\eta)^{\frac{1}{m-1}-1}\sin^{d-1}\theta\cos\theta\d\theta}{\int_0^\pi(\cos\theta-\eta)^{\frac{1}{m-1}-1}\sin^{d-1}\theta\d\theta} \\[3pt] 
    & \qquad + \frac{m\int_0^\pi(\cos\theta-\eta)^{\frac{1}{m-1}}\sin^{d-1}\theta\cos\theta\d\theta - 2\int_0^\pi(\cos\theta-\eta)^{\frac{m}{m-1}}\sin^{d-1}\theta\d\theta}{\int_0^\pi(\cos\theta-\eta)^{\frac{1}{m-1}}\sin^{d-1}\theta\d\theta} \biggr).
\end{align}

Now use
\begin{align*}
\int_0^\pi(\cos\theta-\eta)^{\frac{m}{m-1}}\sin^{d-1}\theta\d\theta &=\int_0^\pi(\cos\theta-\eta)^{\frac{1}{m-1}}(\cos\theta-\eta)\sin^{d-1}\theta\d\theta \\
&=\int_0^\pi(\cos\theta-\eta)^{\frac{1}{m-1}}\sin^{d-1}\cos\theta \d\theta - \eta\int_0^\pi(\cos\theta-\eta)^{\frac{1}{m-1}}\sin^{d-1}\theta\d\theta,
\end{align*}
to write
\begin{equation}
\label{eqn:g1g2p-prelim}
\begin{aligned}
 (g_1 (\eta) g_2(\eta))'  & = \frac{\frac{m}{m-1}\int_0^\pi(\cos\theta-\eta)^{\frac{1}{m-1}-1}\sin^{d-1}\theta\d\theta}{2(m-1)(d\omega_d)^{m-1}\left(\int_0^\pi(\cos\theta-\eta)^{\frac{1}{m-1}}\sin^{d-1}\theta\d\theta\right)^m}\times\\[2pt]
  &\hspace{-2.0cm}\left( -2\eta + \frac{\int_0^\pi(\cos\theta-\eta)^{\frac{1}{m-1}-1}\sin^{d-1}\theta\cos\theta\d\theta}{\int_0^\pi(\cos\theta-\eta)^{\frac{1}{m-1}-1}\sin^{d-1}\theta\d\theta} + \frac{(m-2)\int_0^\pi(\cos\theta-\eta)^{\frac{1}{m-1}}\sin^{d-1}\theta\cos\theta\d\theta}{\int_0^\pi(\cos\theta-\eta)^{\frac{1}{m-1}}\sin^{d-1}\theta\d\theta} +2\eta\right)\\
  & = \frac{\frac{m}{m-1}\int_0^\pi(\cos\theta-\eta)^{\frac{1}{m-1}-1}\sin^{d-1}\theta\d\theta}{2(m-1)(d\omega_d)^{m-1}\left(\int_0^\pi(\cos\theta-\eta)^{\frac{1}{m-1}}\sin^{d-1}\theta\d\theta\right)^m}\times\\
  &\hspace{0.0cm} \left(\frac{\int_0^\pi(\cos\theta-\eta)^{\frac{1}{m-1}-1}\sin^{d-1}\theta\cos\theta\d\theta}{\int_0^\pi(\cos\theta-\eta)^{\frac{1}{m-1}-1}\sin^{d-1}\theta\d\theta} + \frac{(m-2)\int_0^\pi(\cos\theta-\eta)^{\frac{1}{m-1}}\sin^{d-1}\theta\cos\theta\d\theta}{\int_0^\pi(\cos\theta-\eta)^{\frac{1}{m-1}}\sin^{d-1}\theta\d\theta} \right).
\end{aligned}
\end{equation}

By a direct calculation, using \eqref{eqn:H}, one can compute
\begin{equation}
\label{eqn:HpoverH-2}
 \frac{H'(\eta)}{H(\eta)} = -\frac{1}{m-1}  \left( \frac{\int_0^\pi (\cos\theta-\eta)^{\frac{1}{m-1}-1}\sin^{\dm-1}\theta\cos\theta \d\theta}{\int_0^\pi (\cos\theta-\eta)^{\frac{1}{m-1}}\sin^{\dm-1}\theta \cos\theta \d\theta}   + (m-2)\frac{\int_0^\pi (\cos\theta-\eta)^{\frac{1}{m-1}-1}\sin^{\dm-1}\theta \d\theta}{\int_0^\pi (\cos\theta-\eta)^{\frac{1}{m-1}}\sin^{\dm-1}\theta \d\theta} \right).
\end{equation}
Then, by combining \eqref{eqn:g1g2p-prelim} and \eqref{eqn:HpoverH-2}, we write
\begin{equation}
\label{eqn:g1g2p}
(g_1 (\eta) g_2(\eta))' = h(\eta) \left(-(m-1) \frac{H'(\eta)}{H(\eta)} \right),
\end{equation}
where we used the notation 
\begin{equation}
\label{eqn:h}
\begin{aligned}
h(\eta): &= \frac{\frac{m}{m-1}\int_0^\pi(\cos\theta-\eta)^{\frac{1}{m-1}-1}\sin^{d-1}\theta\d\theta}{2(m-1)(d\omega_d)^{m-1}\left(\int_0^\pi(\cos\theta-\eta)^{\frac{1}{m-1}}\sin^{d-1}\theta\d\theta\right)^{m}} \times \frac{\int_0^\pi(\cos\theta-\eta)^{\frac{1}{m-1}}\sin^{d-1}\theta\cos\theta\d\theta}{\int_0^\pi(\cos\theta-\eta)^{\frac{1}{m-1}-1}\sin^{d-1}\theta\d\theta} \\[3pt]
&= \frac{m}{2(m-1)^2 (d\omega_d)^{m-1}} \frac{\int_0^\pi(\cos\theta-\eta)^{\frac{1}{m-1}}\sin^{d-1}\theta\cos\theta\d\theta}{\left(\int_0^\pi(\cos\theta-\eta)^{\frac{1}{m-1}}\sin^{d-1}\theta\d\theta\right)^{m}}.
\end{aligned}
\end{equation}

We will show that $h(\eta)\geq 0$ -- see below. Also, $H(\eta) \geq0$, and hence, $(g_1 (\eta) g_2(\eta))'$ and $H'(\eta)$ have opposite signs, for all $m>1$ and $-\infty < \eta <-1$. By Lemma \ref{lem:H-monotone}, $H(\eta)$ is strictly decreasing for $1<m<2$, and therefore $g_1 (\eta) g_2(\eta)$ is strictly increasing in $\eta$.  Finally, by \eqref{eqn:E-diff} and \eqref{eqn:energy-eta}, we conclude that for any $1<m<2$, the energy difference $E[\rhou] - E[\rho_{\kappa}]$ increases with $\eta_\kappa$ (and hence with $\kappa$), as $\kappa$ ranges in $\kappa_1 \leq \kappa \leq \kappa_2$.

It remains to show that $h(\eta)\geq 0$, which amounts to showing that 
\[\int_0^\pi(\cos\theta-\eta)^{\frac{1}{m-1}}\sin^{d-1}\theta\cos\theta\d\theta\geq0.
\]
Indeed, for any $0\leq\tilde{\theta}\leq\frac{\pi}{2}$, we have: $\left(\cos\left(\frac{\pi}{2}-\tilde{\theta}\right) -\eta\right)^{\frac{1}{m-1}}\geq\left(\cos\left(\frac{\pi}{2}+\tilde{\theta}\right)-\eta\right)^{\frac{1}{m-1}}$ (by the monotonicity of cosine and $m>1$), $ \sin\left(\frac{\pi}{2}-\tilde{\theta}\right)=\sin\left(\frac{\pi}{2}+\tilde{\theta}\right)$ and $ 0\leq\cos\left(\frac{\pi}{2}-\tilde{\theta}\right)=-\cos\left(\frac{\pi}{2}+\tilde{\theta}\right)$. From these relations, we infer 
\[
\int_0^\frac{\pi}{2}(\cos\theta-\eta)^{\frac{1}{m-1}}\sin^{d-1}\theta\cos\theta\d\theta\geq-\int_\frac{\pi}{2}^\pi(\cos\theta-\eta)^{\frac{1}{m-1}}\sin^{d-1}\theta\cos\theta\d\theta,
\]
where in the integrals we change variables $\theta = \frac{\pi}{2} - \tilde{\theta}$ and $\theta = \frac{\pi}{2} +\tilde{\theta}$ respectively. Hence, we find
\begin{align}
    & \int_0^\pi(\cos\theta-\eta)^{\frac{1}{m-1}}\sin^{d-1}\theta\cos\theta\d\theta \\
    & \qquad = \int_0^\frac{\pi}{2}(\cos\theta-\eta)^{\frac{1}{m-1}}\sin^{d-1}\theta\cos\theta\d\theta + \int_\frac{\pi}{2}^\pi(\cos\theta-\eta)^{\frac{1}{m-1}}\sin^{d-1}\theta\cos\theta\d\theta\\
    & \qquad \geq 0.
\end{align}

\smallskip

{\em Case $\kappa>\kappa_2$.} The calculations are very similar to the fully supported case, and we will only summarize them here. 

We estimate again $E[\rhou] - E[\rho_{\kappa}]$ using \eqref{eqn:E-diff}. In this case the equilibria $\rho_\kappa$ from \eqref{eqn:rhok-cs} are supported on a geodesic disk of radius $\phi_\kappa$. Recall the notations $s_\kappa = \| c_{\rho_\kappa}\|$ and $\lambda_\kappa = - \kappa s_\kappa \cos \phi_\kappa $. Using \eqref{eqn:s-phi}, \eqref{eqn:kappa-phi} and \eqref{eqn:rhok-cs}, we compute (similar to \eqref{eqn:energy-eta}):
\begin{equation}
\label{eqn:energy-phi}
\begin{aligned}
\frac{\kappa}{2} \|c_{\rho_\kappa}\|^2 - \frac{1}{m-1}\int_{\bbs^\dm} \rho_{\kappa}^m(x) \dS(x) &= \frac{m}{2(m-1)(d\omega_d)^{m-1}} \frac{\int_0^{\phi_\kappa}(\cos\theta-\cos \phi_\kappa)^{\frac{1}{m-1}}\sin^{d-1}\theta\cos\theta\d\theta}{\left(\int_0^{\phi_\kappa}(\cos\theta-\cos \phi_\kappa)^{\frac{1}{m-1}}\sin^{d-1}\theta\d\theta \right)^{m}}\\[3pt]
&\quad-\frac{1}{(m-1) (\dm w_\dm)^{m-1}}\frac{\int_0^{\phi_\kappa}(\cos\theta-\cos \phi_\kappa)^{\frac{m}{m-1}}\sin^{d-1}\theta\d\theta}{\left(\int_0^{\phi_\kappa}(\cos\theta-\cos \phi_\kappa)^{\frac{1}{m-1}}\sin^{d-1}\theta\d\theta\right)^m}\\[3pt]
&=\bar{g}_1(\phi_\kappa)\bar{g}_2(\phi_\kappa),
\end{aligned}
\end{equation}
where 
\begin{equation}
\label{eqn:g1g2bar}
\begin{aligned}
    \bar{g}_1(\phi) &= m\int_0^{\phi}(\cos\theta-\cos \phi)^{\frac{1}{m-1}}\sin^{d-1}\theta\cos\theta\d\theta -2 \int_0^{\phi}(\cos\theta-\cos \phi)^{\frac{m}{m-1}}\sin^{d-1}\theta\d\theta,\\[5pt]
    \bar{g}_2(\phi) &= \frac{1}{2(m-1)(d\omega_d)^{m-1}\left(\int_0^{\phi} (\cos\theta-\cos \phi)^{\frac{1}{m-1}}\sin^{d-1}\theta\d\theta\right)^m}.
\end{aligned}
\end{equation}
In computing $(\bar{g}_1(\phi) \bar{g}_2(\phi))'$ we notice that the same calculations as in \eqref{eqn:g1p} and \eqref{eqn:g2p} appear, with $\eta$ being replaced by $\cos \phi$ and an additional $-\sin \phi$ that comes from the derivative of $\cos \phi$. We find 
\begin{equation}
\label{eqn:g1bg2bp}
\begin{aligned}
 (\bar{g}_1 (\phi) \bar{g}_2(\phi))'  &= \frac{- \sin \phi \, \frac{m}{m-1}\int_0^\phi(\cos\theta-\cos \phi)^{\frac{1}{m-1}-1}\sin^{d-1}\theta\d\theta}{2(m-1)(d\omega_d)^{m-1}\left(\int_0^\phi(\cos\theta-\cos \phi)^{\frac{1}{m-1}}\sin^{d-1}\theta\d\theta\right)^m}\times\\[2pt]
  &\hspace{-1.5cm}\left(\frac{\int_0^\phi(\cos\theta-\cos \phi)^{\frac{1}{m-1}-1}\sin^{d-1}\theta\cos\theta\d\theta}{\int_0^\phi(\cos\theta-\cos \phi)^{\frac{1}{m-1}-1}\sin^{d-1}\theta\d\theta} + \frac{(m-2)\int_0^\phi(\cos\theta-\cos \phi)^{\frac{1}{m-1}}\sin^{d-1}\theta\cos\theta\d\theta}{\int_0^\phi(\cos\theta-\cos \phi)^{\frac{1}{m-1}}\sin^{d-1}\theta\d\theta} \right)\\
   &= \bar{h}(\phi) \left(-(m-1) \frac{F'(\phi)}{F(\phi)} \right),
\end{aligned}
\end{equation}
where we used \eqref{eqn:dlnF} and the notation
\begin{equation}
\label{eqn:barh}
\begin{aligned}
\bar{h}(\phi):&= \frac{\frac{m}{m-1}\int_0^\phi(\cos\theta-\cos \phi)^{\frac{1}{m-1}-1}\sin^{d-1}\theta\d\theta}{2(m-1)(d\omega_d)^{m-1}\left(\int_0^\phi(\cos\theta-\cos \phi)^{\frac{1}{m-1}}\sin^{d-1}\theta\d\theta\right)^{m}} \times \frac{\int_0^\phi(\cos\theta-\cos \phi)^{\frac{1}{m-1}}\sin^{d-1}\theta\cos\theta\d\theta}{\int_0^\phi(\cos\theta-\cos \phi)^{\frac{1}{m-1}-1}\sin^{d-1}\theta\d\theta} \\[3pt]
&= \frac{m}{2(m-1)^2(d\omega_d)^{m-1}} \frac{\int_0^\phi(\cos\theta-\cos \phi)^{\frac{1}{m-1}}\sin^{d-1}\theta\cos\theta\d\theta}{\left(\int_0^\phi(\cos\theta-\cos \phi)^{\frac{1}{m-1}}\sin^{d-1}\theta\d\theta\right)^{m}}.
\end{aligned}
\end{equation}

We will show that $\bar{h}(\phi)\geq 0$ -- see below. As $F(\phi) \geq0$, we then find that $(\bar{g}_1 (\phi) \bar{g}_2(\phi))'$ and $F'(\phi)$ have opposite signs, for all $m>1$ and $0 < \phi <\pi$. By Lemma \ref{lem:F-monotone}, $F(\phi)$ is strictly increasing in $\phi$ for $1<m<2$, and therefore $\bar{g}_1 (\eta) \bar{g}_2(\phi)$ is strictly decreasing in $\phi$.  By \eqref{eqn:E-diff} and \eqref{eqn:energy-phi}, we conclude that for any $1<m<2$, the energy difference $E[\rhou] - E[\rho_{\kappa}]$ decreases with $\phi_\kappa$. As $\kappa$ takes values from $\kappa_2$ to $\infty$, $\phi_\kappa$ ranges from $\pi$ to $0$ (hence $\phi_\kappa$ decreases with $\kappa$). We conclude that $E[\rhou] - E[\rho_{\kappa}]$ increases with $\kappa$, for $\kappa >\kappa_2$.

It remains to show that $\bar{h}(\phi)\geq 0$, which amounts to showing that 
\begin{equation}
\label{eqn:int-g0}
\int_0^\phi(\cos\theta-\cos \phi)^{\frac{1}{m-1}}\sin^{d-1}\theta\cos\theta\d\theta\geq0.
\end{equation}
If $\phi\leq\frac{\pi}{2}$, then the inequality is immediate, as the integrand is non-negative. For $\frac{\pi}{2}< \phi < \pi$, we note that for any $0<\tilde{\theta}< \phi - \frac{\pi}{2}$, we have
\begin{multline}
 \left(\cos\left( \frac{\pi}{2}-\tilde{\theta}\right)-\cos\phi\right)^{\frac{1}{m-1}}\sin^{d-1}\left(\frac{\pi}{2}-\tilde{\theta}\right)\cos\left(\frac{\pi}{2}-\tilde{\theta}\right) \geq \\
   -\left(\cos\left(\frac{\pi}{2}+\tilde{\theta}\right)-\cos\phi\right)^{\frac{1}{m-1}}\sin^{d-1}\left(\frac{\pi}{2}+\tilde{\theta}\right)\cos\left(\frac{\pi}{2}+\tilde{\theta}\right),
\end{multline}
as it follows from the monotonicity of cosine, and $\cos\left(\frac{\pi}{2}+\tilde{\theta}\right) = -\cos\left(\frac{\pi}{2}-\tilde{\theta}\right)$, $\sin^{d-1}\left(\frac{\pi}{2}+\tilde{\theta}\right)=\sin^{d-1}\left(\frac{\pi}{2}-\tilde{\theta}\right)$. Therefore, by denoting $\phi=\frac{\pi}{2}+\phi_1$, we find
\begin{align}
    & \int_0^\phi(\cos\theta-\cos\phi)^{\frac{1}{m-1}}\sin^{d-1}\theta\cos\theta\d\theta 
    = \int_0^{\frac{\pi}{2}-\phi_1}(\cos\theta-\cos\phi)^{\frac{1}{m-1}}\sin^{d-1}\theta\cos\theta\d\theta \\ 
    & \qquad + \int_{\frac{\pi}{2}-\phi_1}^{\frac{\pi}{2}}(\cos\theta-\cos\phi)^{\frac{1}{m-1}}\sin^{d-1}\theta\cos\theta\d\theta + \int_{\frac{\pi}{2}}^{\frac{\pi}{2}+\phi_1}(\cos\theta-\cos\phi)^{\frac{1}{m-1}}\sin^{d-1}\theta\cos\theta\d\theta \\
    & \qquad \geq \int_0^{\frac{\pi}{2}-\phi_1}(\cos\theta-\cos\phi)^{\frac{1}{m-1}}\sin^{d-1}\theta\cos\theta\d\theta \\
    & \qquad \geq 0.
\end{align}

{\em Conclusion.} To conclude, we have shown that the energy difference $E[\rhou] - E[\rho_{\kappa}]$ increases with $\kappa$, as $\kappa$ ranges in $\kappa_1 \leq \kappa \leq \kappa_2$ (where $\rho_\kappa$ has full support in $\bbs^\dm$), as well as in $\kappa > \kappa_2$ (where $\rho_\kappa$ has strict support in $\bbs^\dm$). Since at $\kappa = \kappa_1$, $\rho_\kappa$ coincides with $\rhou$, the energy difference is zero there. We conclude that
\[
E[\rhou]>E[\rho_\kappa], \quad \text{ for all } \kappa>\kappa_1.
\]
\end{proof}

\begin{remark}
\label{rmk:calc-extend}
We note that calculations in the proof of Theorem \ref{thm:gminimizers} which lead to \eqref{eqn:g1g2p} (for equilibria of full support) and \eqref{eqn:g1bg2bp} (for equilibria of strict support), apply to any $m>1$. Hence, for easier reference, we summarize the conclusions inferred from \eqref{eqn:g1g2p} and \eqref{eqn:g1bg2bp}: a) $\displaystyle (g_1 (\eta) g_2(\eta))'$ and $H'(\eta)$ have opposite signs, for all $m>1$ and $-\infty < \eta <-1$, and b) $\displaystyle (\bar{g}_1 (\phi) \bar{g}_2(\phi))'$ and $F'(\phi)$ have opposite signs, for all $m>1$ and $0 < \phi <\pi$. These results will be used in Sections \ref{sect:m2} and \ref{sect:bifurcations-mg2} for the case $m \geq 2$.
\end{remark}

\subsection{Numerical illustrations}
\label{subsect:numerics-mg1}
We illustrate Proposition \ref{prop:bif-mg1} with some numerical results. In the numerics we have used $m=1.5$ and $\dm=2$. Figure \ref{fig:m15-se}(a) illustrates the bifurcation at $\kappa_1$ and also indicates the second critical value $\kappa_2$. For this choice of $m$ and $\dm$ we have $\kappa_1 \approx 1.2694$ and $\kappa_2 \approx 1.4658$. In blue we plot the uniform distribution, which is stable for $\kappa<\kappa_1$, and unstable otherwise. In red we plot $s_\kappa = \|\rho_\kappa\|$ for the equilibria \eqref{eqn:rhok-fs} ($\kappa_1< \kappa <\kappa_2$) and \eqref{eqn:rhok-cs} ($\kappa >\kappa_2$). The transition at $\kappa_2$ is indicated by a black diamond. Also, Theorem \ref{thm:gminimizers} is illustrated in Figure \ref{fig:m15-se}(b); note how the energy difference $E[\rhou] - E[\rho_{\kappa}]$ increases with $\kappa$.

In Figure \ref{fig:m15-rhok}(a) we show the size of the support $\phi_\kappa$ of equilibria \eqref{eqn:rhok-cs}, and in Figure \ref{fig:m15-rhok}(b) we plot the equilibria at several values of $\kappa$. Note that at $\kappa=\kappa_1$, the equilibria is simply the uniform distribution (i.e., constant on $\bbs^\dm$). As $\kappa$ increases, $\rho_k$ becomes more and more concentrated at $\theta=0$. However, as seen in plot (a), the size of the support $\phi_\kappa$ has a rather slow convergence to $0$.

\begin{figure}[htbp]
 \begin{center}
 \begin{tabular}{cc}
 \includegraphics[width=0.48\textwidth]{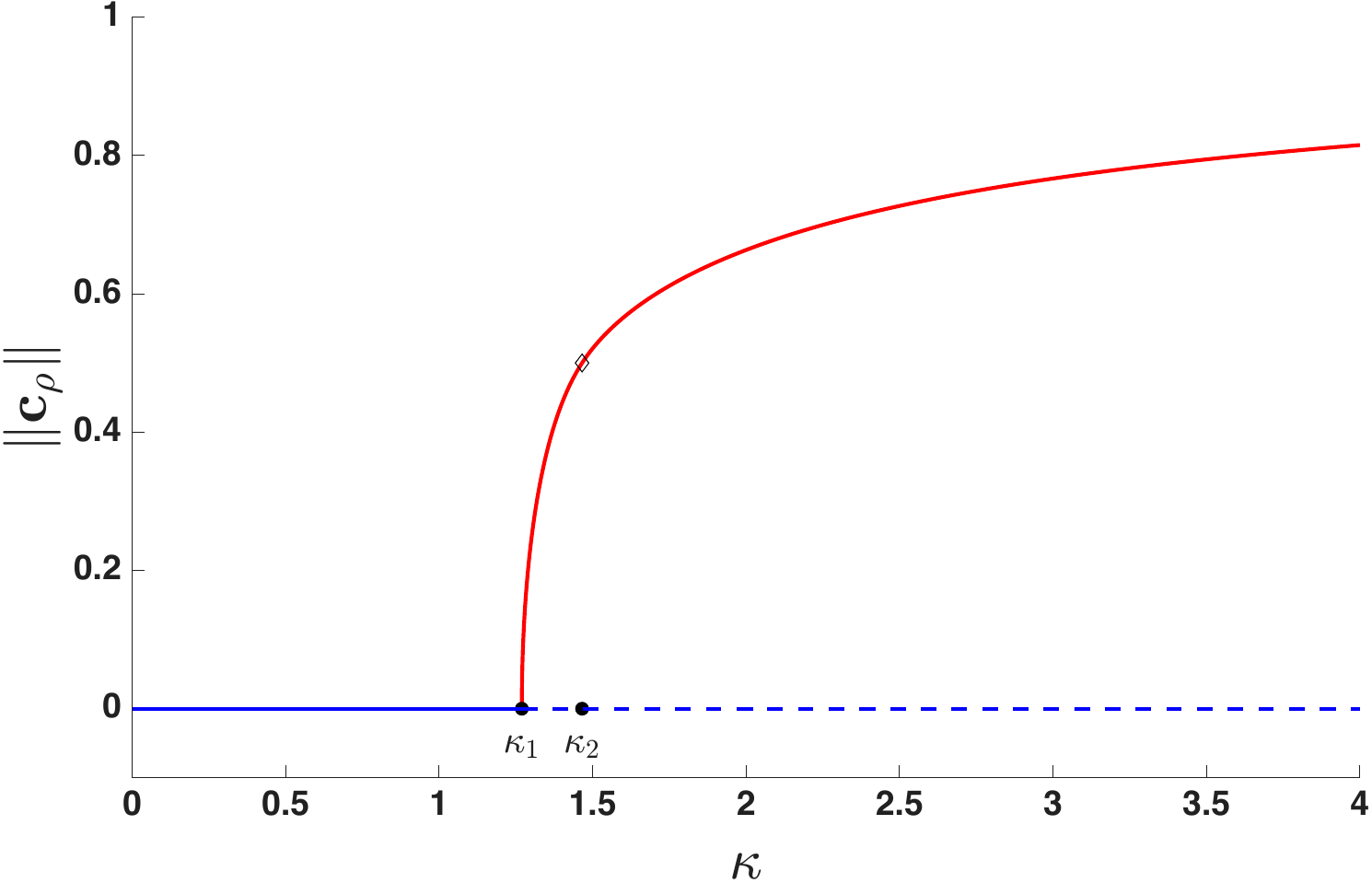} & 
 \includegraphics[width=0.48\textwidth]{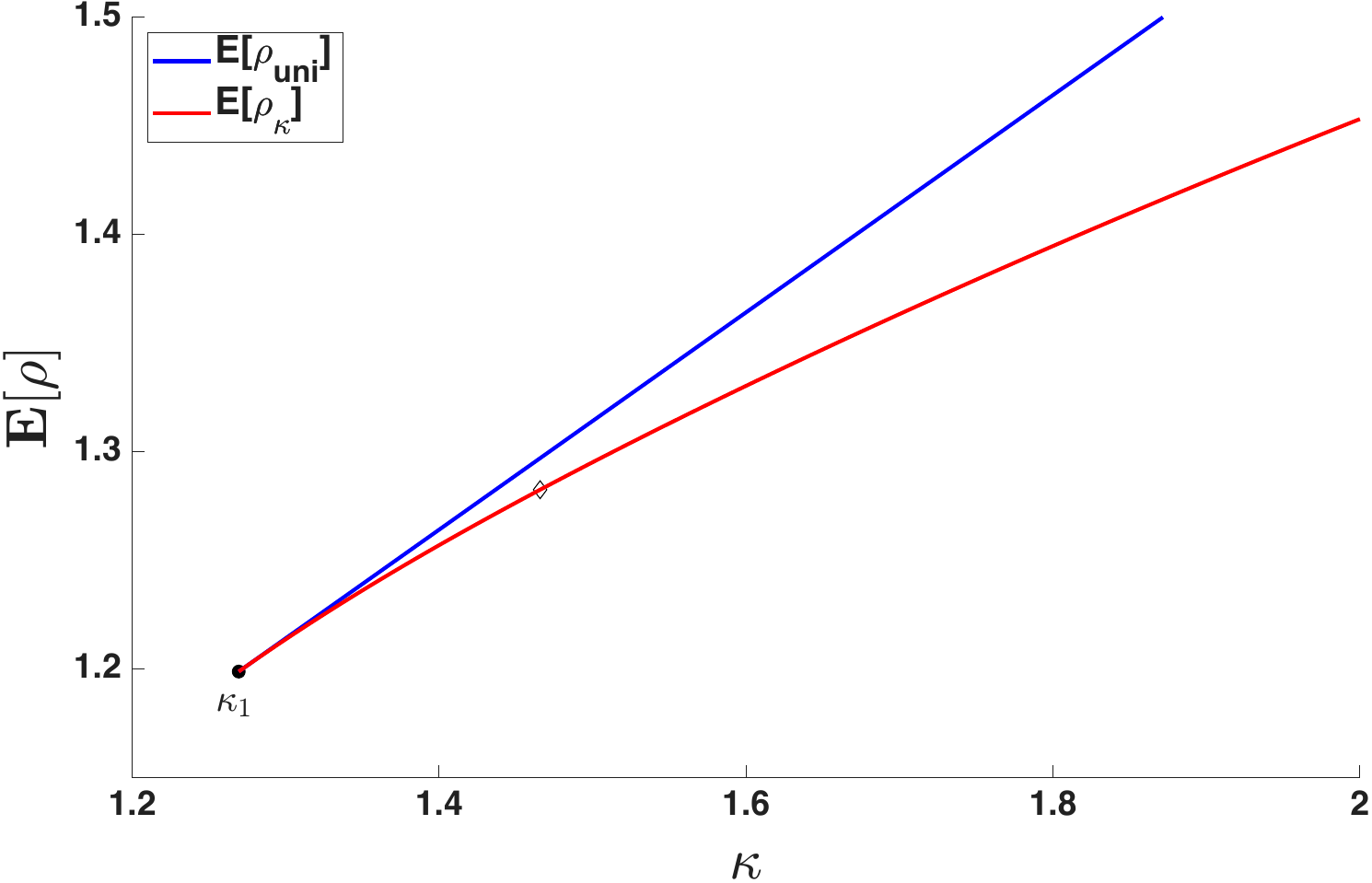} \\
 (a) & (b)
\end{tabular}
\caption{Case $1<m<2$. (a) Plot of the norm of the centre of mass of $\rhou$ (blue) and $\rho_\kappa$ (red) -- see Proposition \ref{prop:bif-mg1}. At $\kappa=\kappa_1$, a fully supported equilibrium $\rho_\kappa$ in the form \eqref{eqn:rhok-fs} emerges from the uniform distribution. At $\kappa=\kappa_2$, $\rho_\kappa$ changes from being fully supported to having support strictly contained on $\bbs^\dm$ -- see \eqref{eqn:rhok-cs}; this transition is indicated by a black diamond. (b) Plot of the energies of $\rhou$ (blue) and $\rho_\kappa$ (red) for $\kappa>\kappa_1$. The global energy minimizer is $\rho_\kappa$ -- see Theorem \ref{thm:gminimizers}  The numerical simulations correspond to $m=1.5$, $\dm=2$.}
\label{fig:m15-se}
\end{center}
\end{figure}

\begin{figure}[htbp]
 \begin{center}
 \begin{tabular}{cc}
\includegraphics[width=0.48\textwidth]{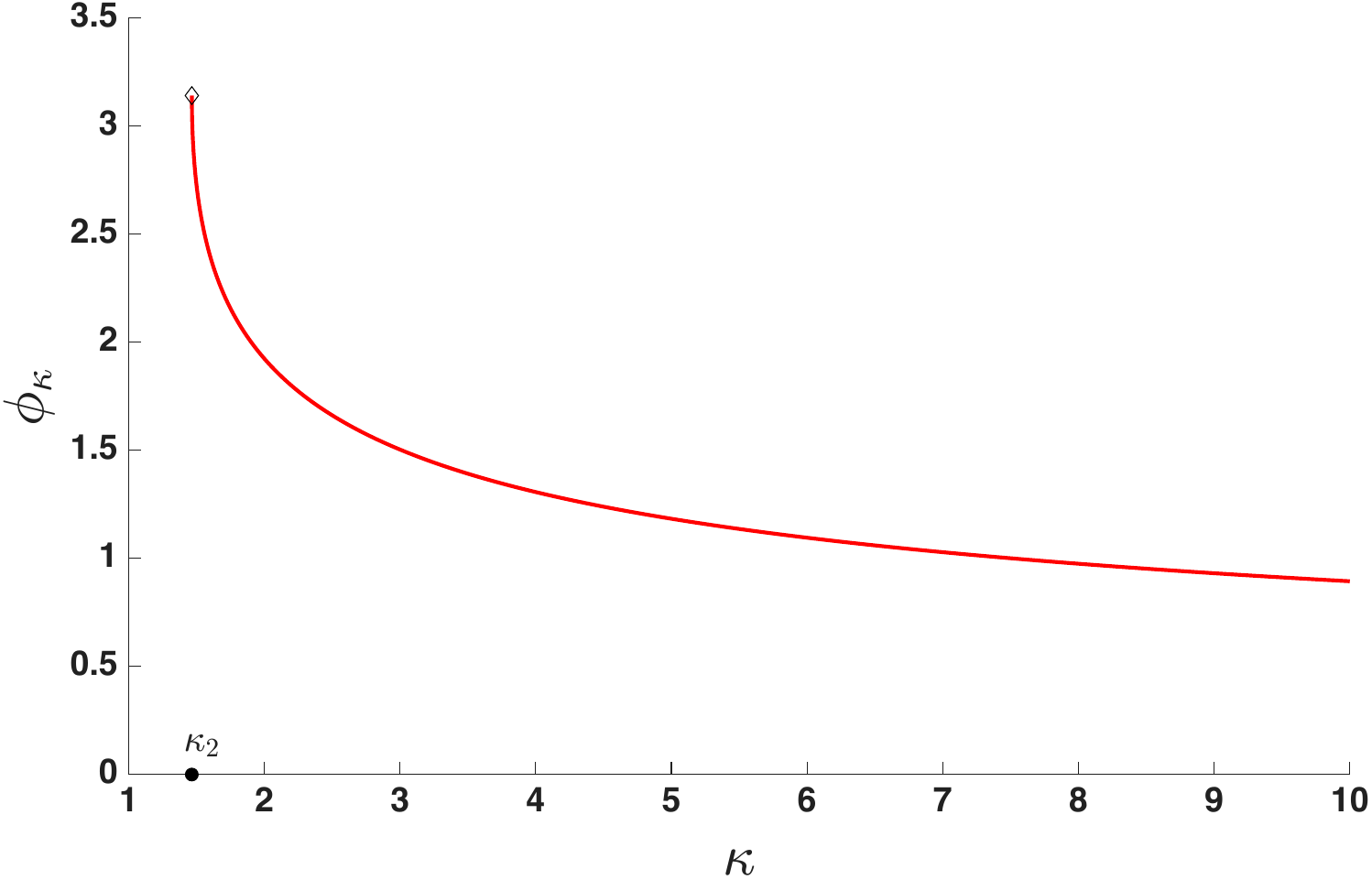} &
\includegraphics[width=0.48\textwidth]{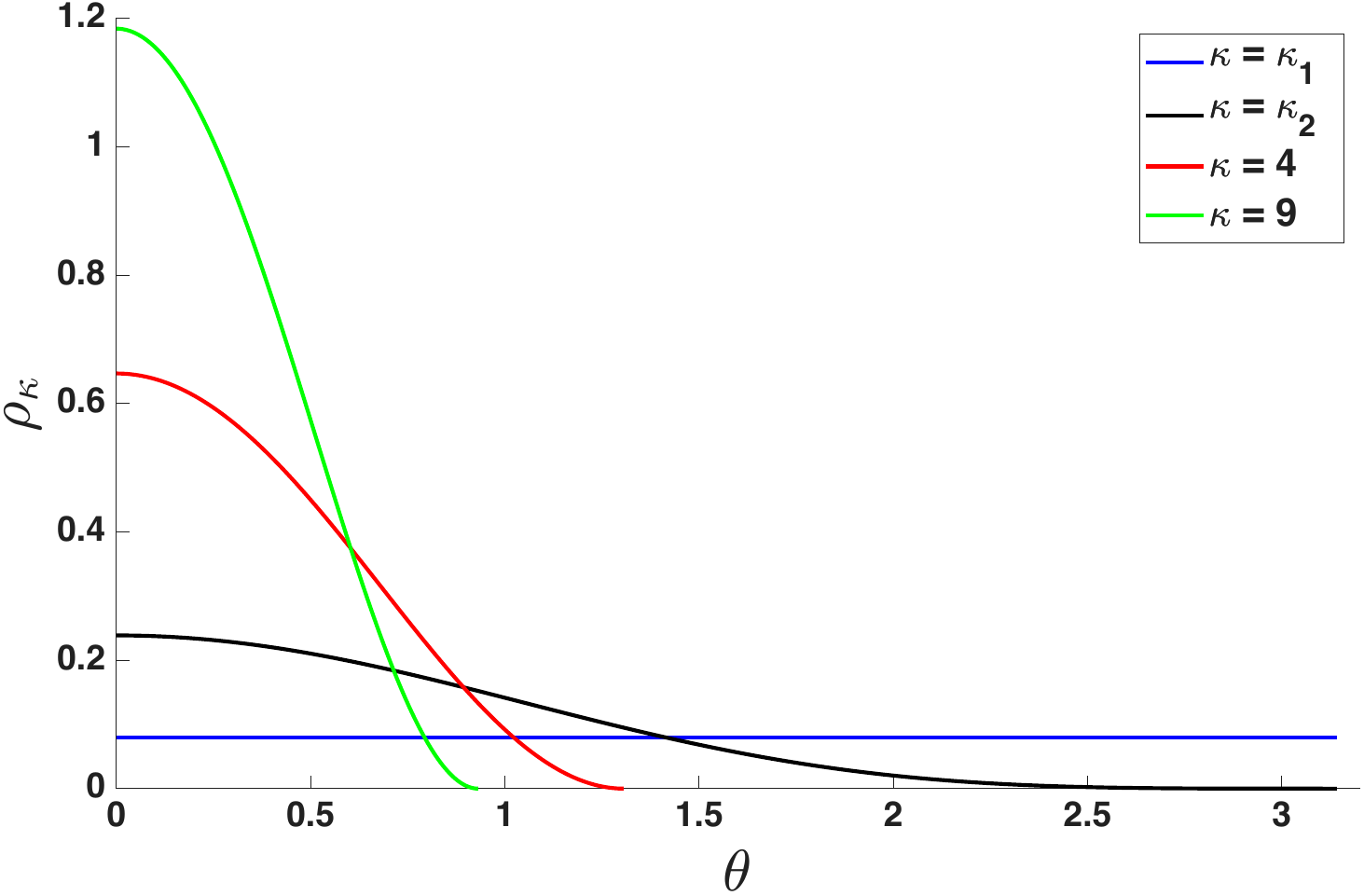} \\
\end{tabular}
\caption{ Case $1<m<2$. (a) Plot of $\phi_\kappa$, the size of the support of equilibria \eqref{eqn:rhok-cs}. (b) Plot of equilibria $\rho_\kappa$ (see \eqref{eqn:rhok-fs} and \eqref{eqn:rhok-cs}) for various values of $\kappa$. Note that at $\kappa=\kappa_1$, the equilibrium is the uniform distribution. The equilibria $\rho_k$ become more and more concentrated at $\theta=0$ as $\kappa$ increases. We used $m=1.5$ and $\dm = 2$.}
\label{fig:m15-rhok}
\end{center}
\end{figure}


\section{Case $m=2$}
\label{sect:m2}
We consider separately the case $m=2$, which turns out to be degenerate in a sense that we will explain. 

\subsection{Equilibria and critical $\kappa$}
\label{subsect:equil-m2}
Consider first equilibria that are fully supported on $\bbs^\dm$ -- see case a) in Section \ref{sect:cp}, in particular equations \eqref{eqn:equil-fs} and \eqref{eqn:system-fs-mg1}. Use again the notation $s=\|c_\rho\|$. Equations \eqref{eqn:system-fs-mg1} become
\begin{equation*}
\begin{aligned}
1 & =\frac{\dm w(\dm)}{2}\int_0^\pi (\lambda+\kappa s\cos\theta)\sin^{\dm-1}\theta\d\theta,\\
s & =\frac{\dm w(\dm)}{2}\int_0^\pi (\lambda+\kappa s\cos\theta)\sin^{\dm-1}\theta\cos\theta\d\theta.
\end{aligned}
\end{equation*}
However, some of the integrals above are zero, so the system simplifies to 
\begin{equation}
\label{eqn:system-fs-m2}
\begin{aligned}
 1&=\lambda\left(\frac{\dm w(\dm)}{2}\int_0^\pi\sin^{\dm-1}\theta\d\theta\right),\\[5pt]
 s&=\kappa s\left(\frac{\dm w(\dm)}{2}\int_0^\pi \sin^{\dm-1}\theta\cos^2\theta\d\theta\right).
 \end{aligned}
\end{equation}

From the first equation in \eqref{eqn:system-fs-m2}, one can determine $\lambda$:
\begin{equation}
\label{eqn:lambda-m2}
\lambda = \left(\frac{\dm w(\dm)}{2}\int_0^\pi \sin^{\dm-1}\theta \d\theta\right)^{-1} = \frac{2}{|\bbs^\dm|}.
\end{equation}
Using \eqref{eqn:cossin-rel}, the second equation in \eqref{eqn:system-fs-m2} can also be written as 
\begin{equation}
\label{eqn:s-simple}
s=\kappa s \, \frac{|\bbs^\dm|}{2 (\dm+1)}.
\end{equation}
For $m=2$, the uniform distribution loses stability (see \eqref{eqn:kappa1}) at 
\begin{equation}
\label{eqn:kappa1-m2}
\kappa_1 =  \frac{2 (\dm+1)}{|\bbs^\dm|}.
\end{equation}
Then, the second equation in \eqref{eqn:system-fs-m2} (or equivalently,  \eqref{eqn:s-simple}) is satisfied trivially for $\kappa=\kappa_1$, and it has only the solution $s=0$ (corresponding to the uniform distribution) for $\kappa \neq \kappa_1$.

Recall that for fully supported equilibria, $\lambda \geq \kappa s$. Hence, for $\kappa=\kappa_1$, using \eqref{eqn:lambda-m2} we get that  $s$ can be any value that satisfies:
\[
0\leq s \leq \frac{\lambda}{\kappa_1}
=\frac{1}{\dm+1}.
\]
Therefore, for $\kappa=\kappa_1$ there is a continuum family of fully supported equilibria in the form \eqref{eqn:equil-fs}, where $\lambda$ is given by \eqref{eqn:lambda-m2} and $s$ can take any value in the range $0 \leq s \leq \frac{1}{\dm+1}$.

Look now into equilibria with support strictly contained in $\bbs^\dm$ -- see case b) in Section \ref{sect:cp}, equations \eqref{eqn:equil-cs} and \eqref{eqn:system-cs-mg1}. In this case, $-\kappa s<\lambda < \kappa s$. Equations \eqref{eqn:system-cs-mg1} become
\begin{equation}
\label{eqn:system-cs-m2}
\begin{aligned}
1&=\frac{\dm w(\dm)}{2}\int_0^\phi (\lambda+\kappa s\cos\theta)\sin^{\dm-1}\theta\d\theta,\\[5pt]
s&=\frac{\dm w(\dm)}{2}\int_0^\phi (\lambda+\kappa s\cos\theta)\sin^{\dm-1}\theta\cos\theta\d\theta.
\end{aligned}
\end{equation}

By substituting $\lambda=-\kappa s\cos\phi$ in \eqref{eqn:system-cs-m2}, we get
\begin{equation}
\label{two-2}
\begin{aligned}
1 & =\frac{\dm w(\dm)}{2}\; \kappa s\left(-\cos\phi\int_0^\phi \sin^{\dm-1}\theta \d\theta+\int_0^\phi \cos\theta\sin^{\dm-1}\theta\d\theta\right),\\
s &=\frac{\dm w(\dm)}{2}\; \kappa s\left(-\cos\phi\int_0^\phi \sin^{\dm-1}\theta\cos\theta\d\theta+\int_0^\phi \cos^2\theta\sin^{\dm-1}\theta\d\theta\right).
\end{aligned}
\end{equation}

By noting that
\begin{align}\label{X-1}
\int_0^\phi \cos^2\theta \sin^{\dm-1}\d\theta-\cos\phi\int_0^\phi \sin^{\dm-1}\theta\cos\theta \d\theta=\frac{1}{d}\int_0^\phi \sin^{\dm+1}\theta\d\theta,
\end{align}
and canceling $s$, the second equation in \eqref{two-2} reduces to
\[
1=\frac{w(\dm)\kappa}{2}\int_0^\phi \sin^{\dm+1}\theta \d\theta,
\]
or equivalently,
\begin{equation}
\label{eqn:kappa-phi-m2}
\kappa^{-1}=\frac{w(\dm)}{2}\int_0^\phi\sin^{\dm+1}\theta\d\theta.
\end{equation}
This is in fact equation \eqref{eqn:kappa-phi} for $m=2$. Note that 
\[
\frac{w(\dm)}{2}\int_0^\pi \sin^{\dm+1}\theta\d\theta=\frac{ w(\dm)}{2} \frac{\dm}{\dm+1}\int_0^\pi \sin^{\dm-1}\theta \d\theta=\kappa_1^{-1}.
\]
Then, since $\phi$ ranges from 0 to $\pi$, if $\kappa^{-1}>\kappa_1^{-1}$ (or $\kappa<\kappa_1$), there is no $\phi$ that satisfies \eqref{eqn:kappa-phi-m2}. On the other hand, for any $\kappa > \kappa_1$, there is a unique $\phi\in[0, \pi]$; this follows from the intermediate value property and the property of increasing functions. 

With $\phi$ determined uniquely in terms of $\kappa$, we can find $s$ from the first equation in \eqref{two-2}:
\begin{equation}
\label{eqn:s-phi-m2}
s=\frac{2}{\dm w(\dm)\kappa }\left(-\cos\phi\int_0^\phi \sin^{\dm-1}\theta\d\theta+\int_0^\phi\cos\theta\sin^{\dm-1}\theta\d\theta\right)^{-1},
\end{equation}
and then set $\lambda=-\kappa s\cos\phi$.

\begin{remark}
\label{rmk:kappa1-m2}
For $m=2$, the function $H(\eta)$ from \eqref{eqn:H} is constant in $\eta$, as the second term in the r.h.s. is $1$, and the first term does not depend on $\eta$, since
\[
\int_0^\pi (\cos\theta-\eta) \sin^{\dm-1}\theta \cos\theta\d\theta =  \int_0^\pi \cos^2 \theta \sin^{\dm-1} \theta \d \theta.
\]
Hence, $\kappa_1 = \kappa_2$ in this case; note that one can also check that $\kappa_1=\kappa_2$ by substituting $m=2$ into \eqref{eqn:kappa1} and \eqref{eqn:kappa2-alt}.
\end{remark}



The considerations above can be collected in the following proposition. 
\begin{proposition}[Critical $\kappa$ and equilibria for $m=2$]
\label{prop:bif-m2}
For $m=2$ and $\dm \geq 1$, there exists only one critical value $\kappa_1$ (see \eqref{eqn:kappa1} for $m=2$) such that:

\noindent i) The uniform distribution $\rhou$ is the only equilibrium for $\kappa<\kappa_1$. At $\kappa=\kappa_1$ , there exists a family of fully supported equilibria in the form \eqref{eqn:equil-fs} given by 
\begin{equation}
\label{eqn:rhok-fs-m2}
\rho_{\kappa_1}(x; s)=\frac{1}{2} (\lambda +\kappa_1 s \cos\theta_x),\qquad\forall x\in \bbs^\dm,
\end{equation}
where $\lambda$ is given by \eqref{eqn:lambda-m2} and $s$ can take any value in the range $0 \leq s \leq \frac{1}{\dm+1}$. 
\smallskip

\noindent ii) For any $\kappa>\kappa_1$ there exists a unique equilibrium with support strictly contained in $\bbs^\dm$. This equilibrium is given by (see \eqref{eqn:equil-cs}):
\begin{equation}
\label{eqn:rhok-cs-m2}
\rho_\kappa(x)=\begin{cases}
\frac{1}{2} \left(\lambda_\kappa+\kappa s_\kappa\cos\theta_x\right),\qquad&\text{ if }0\leq\theta_x\leq \arccos\left(-\frac{\lambda_\kappa}{\kappa s_\kappa}\right),\\[5pt]
0,\qquad&\text{ otherwise},
\end{cases}
\end{equation}
with $s_\kappa$ and $\lambda_\kappa$ uniquely determined by $\kappa$ -- see \eqref{eqn:kappa-phi-m2} and \eqref{eqn:s-phi-m2}.
\end{proposition}

\begin{remark}
\label{rmk:bif-m2}   
The bifurcation for $m=2$ is degenerate in the sense that the two critical values $\kappa_1$ and $\kappa_2$ identified for $1<m<2$ (Section \ref{sect:bifurcations-mg1}), are in fact the same here. 
We also point out that in the limit $\kappa \searrow \kappa_1$, $\rho_\kappa(\cdot)$ from \eqref{eqn:rhok-cs-m2} does not approach the uniform distribution, but it approaches instead $\rho_{\kappa_1}(\cdot;\frac{1}{\dm+1})$ from \eqref{eqn:rhok-fs-m2} -- see also Figure \ref{fig:m2-srho}.


\end{remark}

We present some numerical simulations in Figure \ref{fig:m2-srho}. Figure \ref{fig:m2-srho}(a) illustrates the degenerate bifurcation that occurs at $\kappa_1 \approx 0.4775$ (in the numerics we used $\dm=2$). The vertical segment at $\kappa=\kappa_1$, with $1\leq s \leq 1/3$, corresponds to the family of equilibria \eqref{eqn:rhok-fs-m2}. At $\kappa=\kappa_1$ and $s = 1/(\dm+1)$, there is a transition (indicated by a black diamond) to equilibria of the form \eqref{eqn:rhok-cs-m2}. Several equilibria in the form \eqref{eqn:rhok-fs-m2}-\eqref{eqn:rhok-cs-m2} are shown in Figure \ref{fig:m2-srho}(b).

\begin{figure}[htbp]
 \begin{center}
 \begin{tabular}{cc}
 \includegraphics[width=0.48\textwidth]{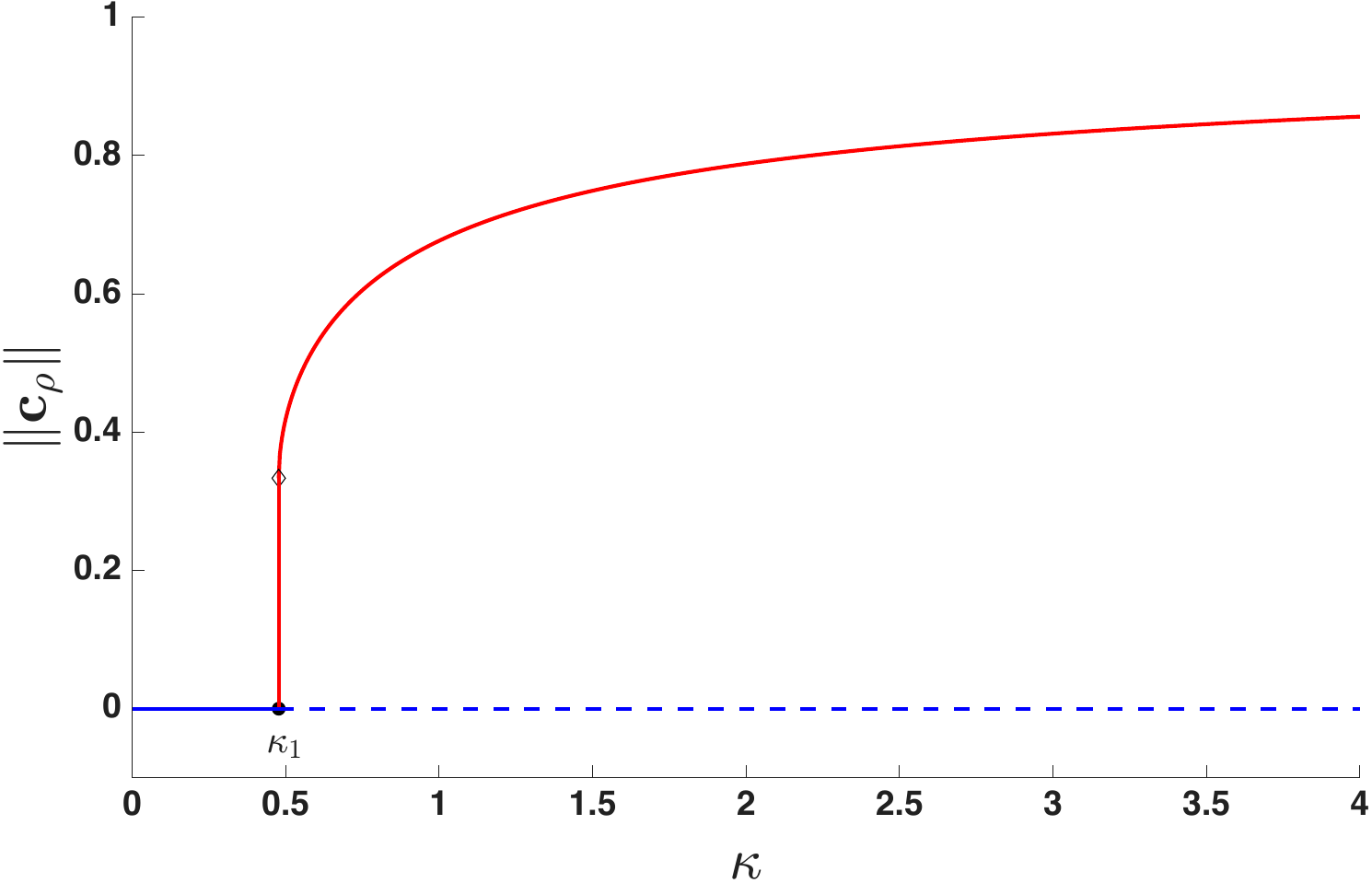} & 
 \includegraphics[width=0.48\textwidth]{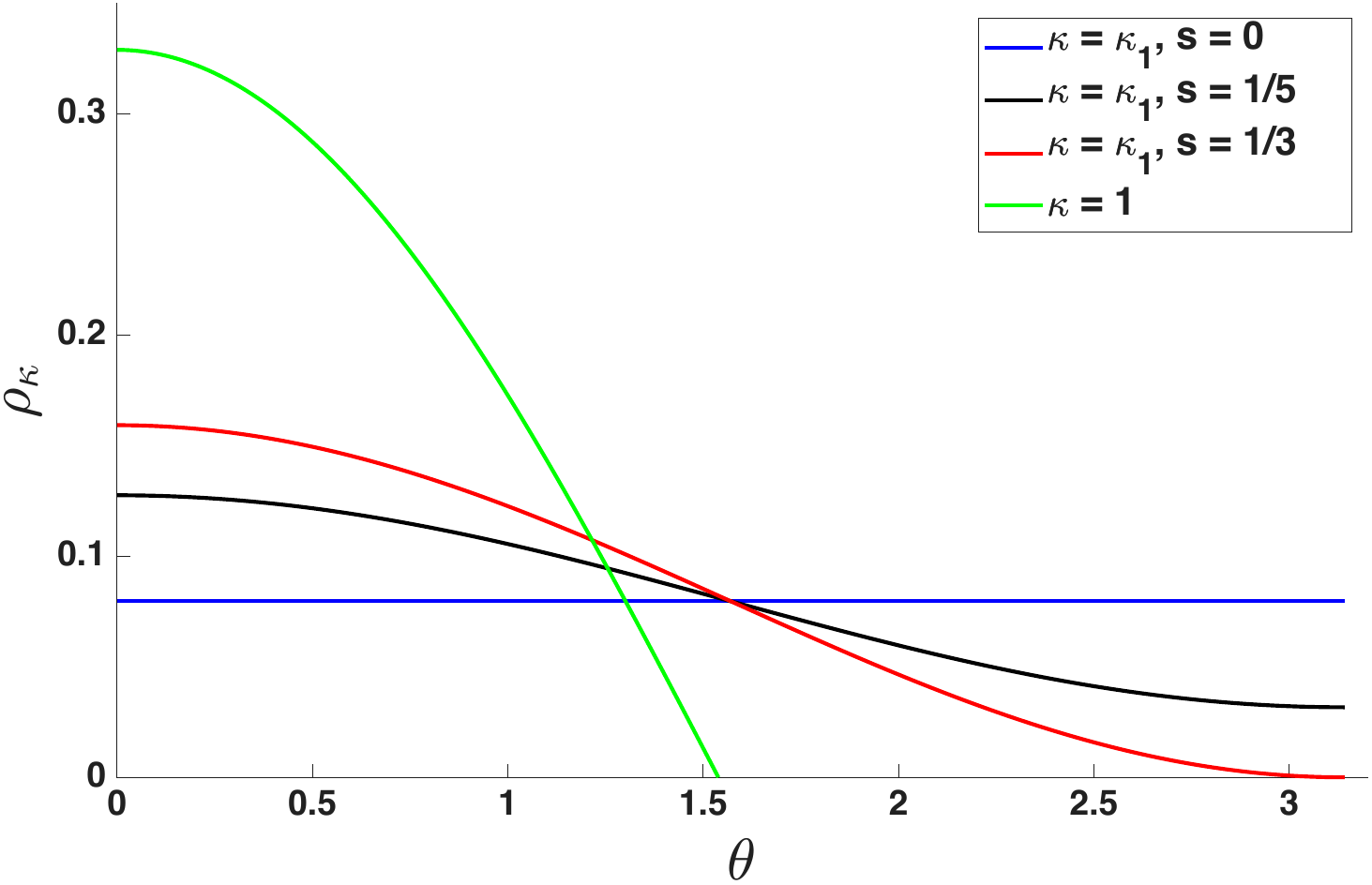} \\
 (a) & (b)
\end{tabular}
\caption{Case $m=2$. (a) Plot of the norm of the centre of mass of $\rhou$ (blue) and $\rho_\kappa$ (red) -- see Proposition \ref{prop:bif-m2}. At $\kappa=\kappa_1$, there exists a family of fully supported equilibria in the form \eqref{eqn:rhok-fs-m2}, parameterized by 
$s \in [ 0, \frac{1}{\dm+1}]$. At every $\kappa>\kappa_1$, there is a unique equilibrium in the form \eqref{eqn:rhok-cs-m2}, which has support strictly contained in $\bbs^\dm$. The transition at $\kappa=\kappa_1$, $s = 1/(\dm+1)$ is indicated by a black diamond. (b) Plot of equilibria \eqref{eqn:rhok-fs-m2} and \eqref{eqn:rhok-cs-m2} for various values of $\kappa$ and $s$. For $\kappa=\kappa_1$, we show three equilibria of the form \eqref{eqn:rhok-fs-m2}, where $s=0$ corresponds to the uniform distribution and $s=1/3$ is the transition value shown by black diamond. The numerical simulations correspond to $\dm=2$.}
\label{fig:m2-srho}
\end{center}
\end{figure}


\subsection{Energy minimizers}
\label{subsect:gmin-m2}
The global energy minimizers for the case $m=2$ are presented in the following theorem.

\begin{theorem}[Global energy minimizers for $m=2$]
\label{thm:gminimizers-m2} 

Let $m=2$ and $\dm\geq 1$. The global minimizer of \eqref{energy-sphere} on $\calP_{ac}(\bbs^\dm)$ is 
\begin{enumerate}
\item the uniform distribution $\rhou$ when $\kappa<\kappa_1$, \\[-7pt]
\item any equilibrium of the form \eqref{eqn:rhok-fs-m2} with $s \in \left[0, \frac{1}{\dm+1} \right]$, when $\kappa= \kappa_1$, \\[-7pt]
\item the equilibrium $\rho_{\kappa}$ given by \eqref{eqn:rhok-cs-m2} when $\kappa>\kappa_1$.
\end{enumerate}
Moreover, the energy difference $E[\rhou] - E[\rho_{\kappa}]$ increases with $\kappa$ for $\kappa>\kappa_1$. 
\end{theorem}
\begin{proof} 
For $\kappa<\kappa_1$, $\rhou$ is the only equilibrium, and hence the global minimizer. We will next discuss the cases $\kappa = \kappa_1$ and $\kappa > \kappa_1$.
\smallskip

{\em Case $\kappa = \kappa_1$.} Consider a generic equilibrium in the form \eqref{eqn:rhok-fs-m2}, corresponding to $\kappa=\kappa_1$ and some $ \displaystyle 0\leq s\leq \frac{1}{\dm+1}$ (see the vertical segment in Figure \ref{fig:m2-srho}(a)), and compute its energy.  

The entropy can be computed as follows:
\begin{align*}
\int_{\bbs^\dm}\rho_{\kappa_1}(x;s)^2\d S(x)&=\frac{\dm w_\dm}{4}\int_0^\pi (\lambda+\kappa_1 s \cos\theta)^2\sin^{\dm-1}\theta\d\theta\\[3pt]
&=\frac{\dm w_\dm}{4}\int_0^\pi \left(\lambda^2\sin^{\dm-1}\theta+\kappa_1^2 s^2\cos^2\theta\sin^{\dm-1}\theta \right)\d\theta\\[3pt]
&=\frac{\dm w_\dm}{4} \left( \frac{\kappa_1^2}{(\dm+1)^2} + \kappa_1^2 s^2 \frac{1}{\dm+1} \right) \int_0^\pi \sin^{\dm-1}\theta \d\theta,
\end{align*}
where for the last equality we used \eqref{eqn:cossin-rel} and that $\lambda= \frac{\kappa_1}{\dm+1}$ (see \eqref{eqn:lambda-m2} and \eqref{eqn:kappa1-m2}).
Also, using the expression of $\kappa_1$ from \eqref{eqn:kappa1-m2}, and the fact that $|\bbs^\dm| = \dm w_\dm \int_0^\pi \sin^{\dm-1}\theta \d\theta$, we can write the above as
\[
\int_{\bbs^\dm}\rho_{\kappa_1}(x;s)^2\d S(x) = \frac{\kappa_1}{2(\dm+1)} + \frac{\kappa_1 s^2}{2}.
\]


The interaction energy (see \eqref{eqn:energy-s}) is given by
\[
-\frac{\kappa_1 s^2}{2}+\frac{\kappa_1}{2}.
\]
By combining the two components, we find the total energy
\[
E[\rho_{\kappa_1}(\cdot,s)]=\frac{\kappa_1}{2(\dm+1)}+\frac{\kappa_1}{2}=\frac{\kappa_1(\dm+2)}{2(\dm+1)}.
\]
Note that the energy does not depend on $s$. It means that when $\kappa=\kappa_1$, each equilibrium in the form \eqref{eqn:rhok-fs-m2} (with $0\leq s\leq \frac{1}{\dm+1}$) is a global energy minimizer. 
\smallskip

{\em Case $\kappa > \kappa_1$.} One can show that $E[\rhou]>E[\rho_\kappa]$ in the same way as this was shown for Theorem \eqref{thm:gminimizers} (see Remark \ref{rmk:calc-extend}) in the case $\kappa>\kappa_2$; recall that $m=2$ is a degenerate case where $\kappa_1 = \kappa_2$. Indeed, following the same calculations as in the proof of Theorem \eqref{thm:gminimizers} (case $\kappa>\kappa_2$) using $m=2$  -- see \eqref{eqn:energy-phi} and the calculations that follow, the desired conclusion results from the monotonicity of the function $F(\phi)$. 

Though not covered by Lemma \ref{lem:F-monotone}, the monotonicity of the function $F(\phi)$ for $m=2$ is immediate. Indeed, for $m=2$ \eqref{eqn:F} reduces to
\begin{equation}
\label{eqn:Fm2}
F(\phi)= \frac{1}{2} \dm w_\dm \int_0^\phi (\cos\theta-\cos\phi) \sin^{\dm-1}\theta \cos\theta\d\theta, 
\end{equation}
which further simplifies to (see calculation leading to \eqref{eqn:kappa-phi-m2}):
\begin{equation}
F(\phi) = \frac{w(\dm)}{2}\int_0^\phi\sin^{\dm+1}\theta\d\theta.
\end{equation}

To conclude, $F(\phi)$ is strictly increasing with $\phi$, and as in the proof of Theorem \eqref{thm:gminimizers} we infer that the energy difference $E[\rhou] - E[\rho_{\kappa}]$ decreases with $\phi_\kappa$. As $\kappa$ increases from $\kappa_1$ to $\infty$, $\phi_\kappa$ decreases from $\pi$ to $0$. Hence, $\phi_\kappa$ decreases with $\kappa$, and we conclude that $E[\rhou] - E[\rho_{\kappa}]$ increases with $\kappa$, for $\kappa >\kappa_1$.
\end{proof}


\section{Case $m>2$ and $\dm=2$}
\label{sect:bifurcations-mg2}

In this section we consider general exponent $m>2$, but restrict to the case $\dm=2$ only. From the point of view of applications to rod-like polymer orientation \cite{fatkullin2005critical}, we argue that this is the most relevant case to consider. The limitation to $\dm =2$ is due to the difficulty in assessing the monotonicity of the function $F(\phi)$ (see \eqref{eqn:F}) for general dimension $\dm$. As shown below, $F(\phi)$ is no longer monotonic for $\dm=2$, and numerical checks indicate that it is no longer monotonic for general $\dm$ as well.

\subsection{Phase transitions and existence of equilibria}
\label{subsect:mg2:existence}

Start with equilibria of full support. Following Section \ref{subsect:fs} (see \eqref{eqn:skappa-eta} and \eqref{eqn:H}), the problem reduces to solving $H(\eta) = \kappa^{-1}$. Note that when $m>2$, the function $H(\eta)$ is increasing (see Lemma \ref{lem:H-monotone} and Figure \ref{fig:HF-mg2}(a)). Hence, by using the same notations for $\kappa_1$ and $\kappa_2$ (see Lemma \ref{lem:Hlim} and equation \eqref{eqn:kappa2}), in this case we have $\kappa_2<\kappa_1$. For any $\kappa_2<\kappa<\kappa_1$, there exists a unique solution of $H(\eta) = \kappa^{-1}$ in $(-\infty,-1)$ (i.e., of \eqref{eqn:kappa-eta}) -- see Figure \ref{fig:HF-mg2}(a) for an illustration. In particular, there exists a unique equilibrium with full support in $\bbs^\dm$ in the form \eqref{eqn:equil-fs}; in Proposition \ref{prop:bif-mg2} below (see \eqref{eqn:rhok-fs-mg2}) this equilibrium is referred to as $\rho_{\kappa,2}$. Therefore, at $\kappa=\kappa_1$ we find again a bifurcation from the uniform density $\rhou$, but in this case a fully supported equilibrium in the form \eqref{eqn:equil-fs} emerges as $\kappa$ {\em decreases} through $\kappa_1$. This is one of the major differences from the case $1<m<2$.

At $\kappa=\kappa_2$, fully supported equilibria transition to equilibria of strict support. Finding the support of the latter equilibria reduces to solving $F(\phi) = \kappa^{-1}$ (see \eqref{eqn:skappa-phi} and \eqref{eqn:F}). And here arises the second fundamental difference from the case $1<m<2$: $F(\phi)$ is no longer monotone when $m>2$. In $\dm=2$ the monotonicity of $F$ can be established easily, by a direct calculation, as given in the following lemma.

\begin{lemma}
\label{lem:F-monotone-mg2}
Let $m>2$, $\dm=2$, and the function $F(\phi)$ for $0<\phi<\pi$ given by \eqref{eqn:F}. Then, there is $\bar{\phi} \in (0,\pi)$ such that $F$ is increasing on $(0,\bar{\phi})$ and decreasing on $(\bar{\phi}, \pi)$. 
\end{lemma}
\begin{proof}
 When $\dm =2$, the integrals in \eqref{eqn:F} can be computed explicitly by substitution. We find 
\[
F(\phi) = \frac{(m-1)^2}{m(2m-1)} \left( \frac{m-1}{m}\right)^{m-1} (2 \pi)^{m-1} (1-\cos \phi)^m \left(\frac{1}{m-1} +1 + \cos \phi \right).
\]
Denote the constant
\[
C := \frac{(m-1)^2}{m(2m-1)} \left( \frac{m-1}{m}\right)^{m-1} (2 \pi)^{m-1}.
\]
Then, the derivative of $F$ is computed as
\[
F'(\phi) = C (m+1) \sin \phi (1-\cos \phi)^{m-1} \left( \frac{m^2-m+1}{m^2-1} + \cos \phi\right).
\]

Hence, the sign of $F'$ depends on the behaviour of the function 
\[
g(m) = \frac{m^2-m+1}{m^2-1}.
\]
For $1<m<2$, we have $g(m)>1$, so $F'(\phi)>0$ (cf., Lemma \ref{lem:F-monotone}). For $m>2$, we have $g(m)<1$, which makes $F'$ change sign from positive to negative at 
\[
\bar{\phi}:= \pi -\operatorname{acos}(g(m)).
\]
We conclude that $F$ is increasing on $(0,\bar{\phi})$, decreasing on $(\bar{\phi}, \pi)$, and at $\phi=\bar{\phi}$ it has a global maximum 
-- see Figure \ref{fig:HF-mg2}(b).
\end{proof}

\begin{figure}[!htbp]
 \begin{center}
 \begin{tabular}{cc}
 \includegraphics[width=0.48\textwidth]{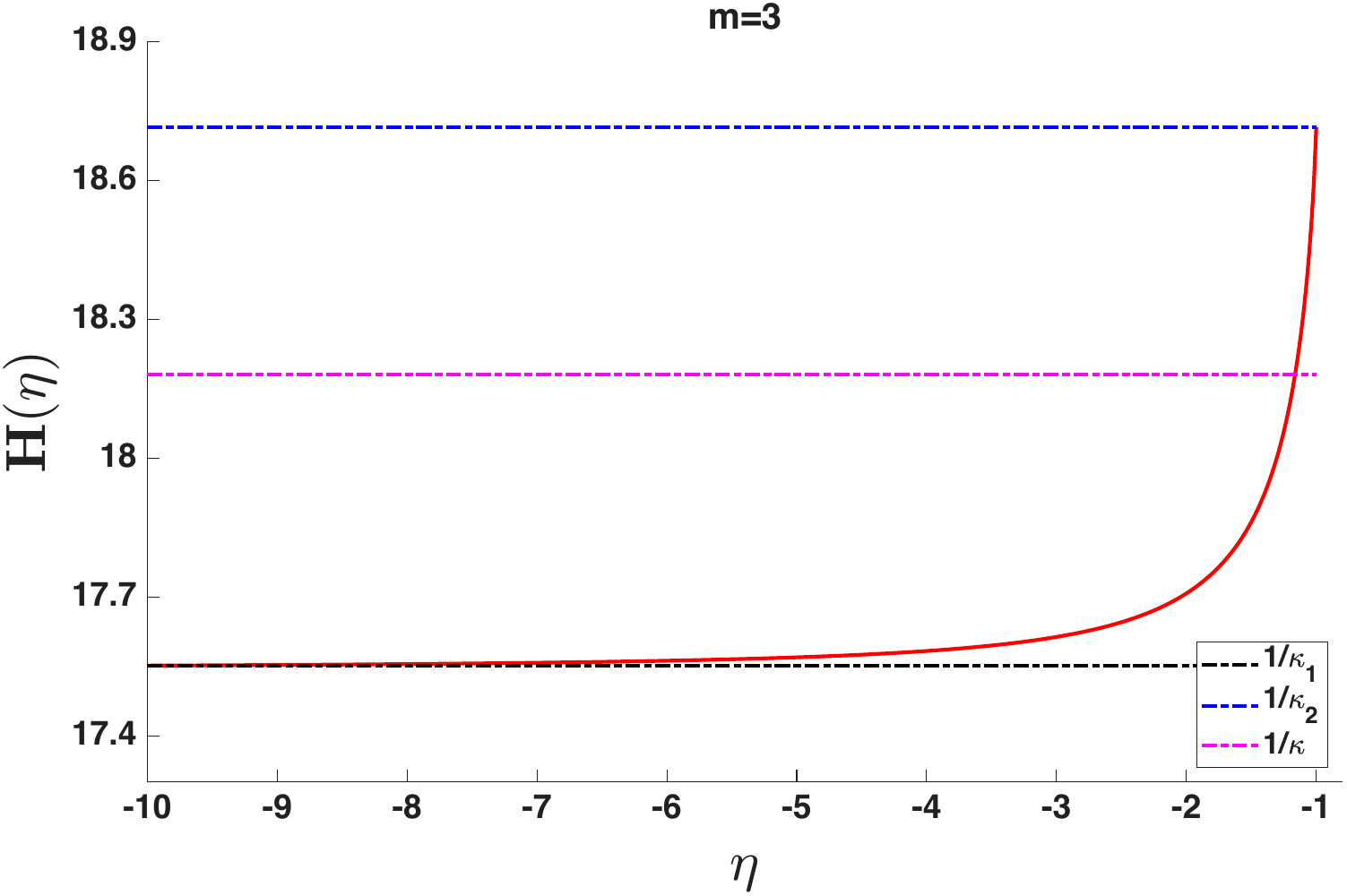} & 
 \includegraphics[width=0.48\textwidth]{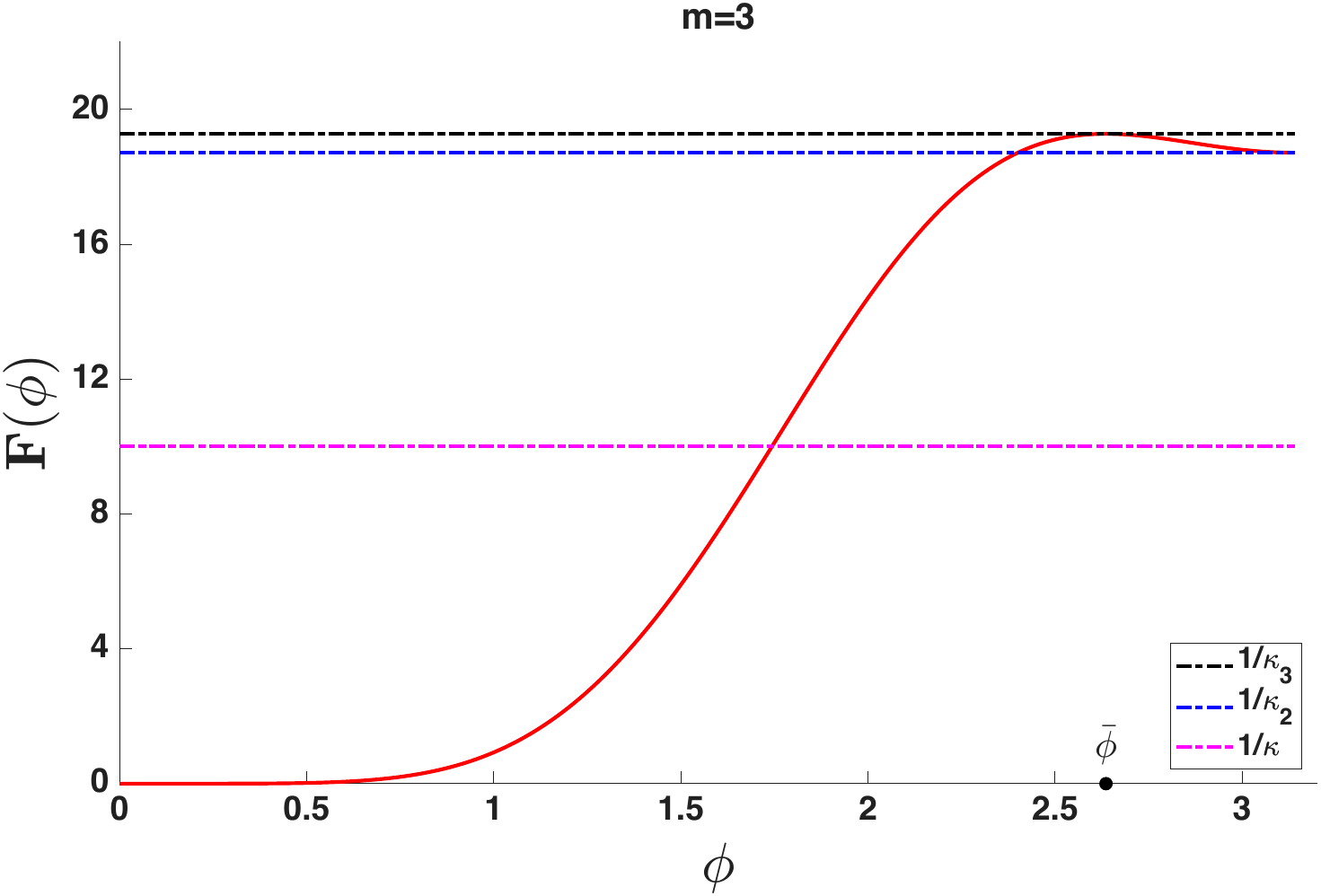} \\
 (a) & (b)
\end{tabular}
\caption{Case $\dm =2$, $m>2$. (a) Plot of function $H$ defined in \eqref{eqn:H}. For any $\kappa_2<\kappa<\kappa_1$, there exists a unique $\eta_\kappa <-1$ such that $\kappa^{-1} = H(\eta_\kappa)$ -- see equation \eqref{eqn:kappa-eta}. (b) Plot of function $F$ defined in \eqref{eqn:F}. The function changes monotonicity at $\bar{\phi}$, where it has a global maximum. For $\kappa_3<\kappa<\kappa_2$, there exist two solutions $\phi_{\kappa,1} \in (0,\bar{\phi})$ and $\phi_{\kappa,2}\in(\bar{\phi},\pi)$ of \eqref{eqn:kappa-phi}. At $\kappa=\kappa_3$, the two solutions $\phi_{\kappa,1}$ and $\phi_{\kappa,2}$ coincide. For $\kappa>\kappa_2$, there exists a unique solution $\phi_{\kappa,1} \in (0,\bar{\phi})$ of \eqref{eqn:kappa-phi}.  For both plots, $m=3$.}
\label{fig:HF-mg2}
\end{center}
\end{figure}

\begin{remark}
\label{rmk:kappa3}
Define a third critical value of $\kappa$ by 
\begin{equation}
\label{eqn:kappa3}
\kappa_3:= (F(\bar{\phi}))^{-1}.
\end{equation}
Then, at $\kappa=\kappa_3$, the equation $F(\phi) = \kappa^{-1}$ has one solution given by $\bar{\phi}$ -- see the black dash-dotted line in Figure \ref{fig:HF-mg2}(b). Also, as $\kappa_2 = 1/F(\pi) > \kappa_3$, for any $\kappa_3<\kappa<\kappa_2$, the equation $F(\phi) = \kappa^{-1}$ has two solutions: $\phi_{\kappa,1} \in (0, \bar{\phi})$ and $\phi_{\kappa,2} \in (\bar{\phi},\pi)$. Finally, for $\kappa >\kappa_2$, there exists only one solution $\phi_{\kappa,1} \in (0, \bar{\phi})$ of $F(\phi) = \kappa^{-1}$ -- see the magenta dash-dotted line in Figure \ref{fig:HF-mg2}(b). As $\kappa$ increases to $\infty$, $\phi_{\kappa,1}$ decreases to $0$.
\end{remark}

The considerations above can be put together in the following proposition. We also refer to Figures \ref{fig:m3-splot} and \ref{fig:m3-rho} for some numerical illustrations which support the theoretical findings. For numerical results we used $m=3$, $\dm =2$, for which $\kappa_1 \approx 0.0569$, $\kappa_2 \approx 0.0534$, $\kappa_3 \approx 0.0518$.

\begin{figure}[htbp]
 \begin{center}
 \begin{tabular}{cc}
 \includegraphics[width=0.48\textwidth]{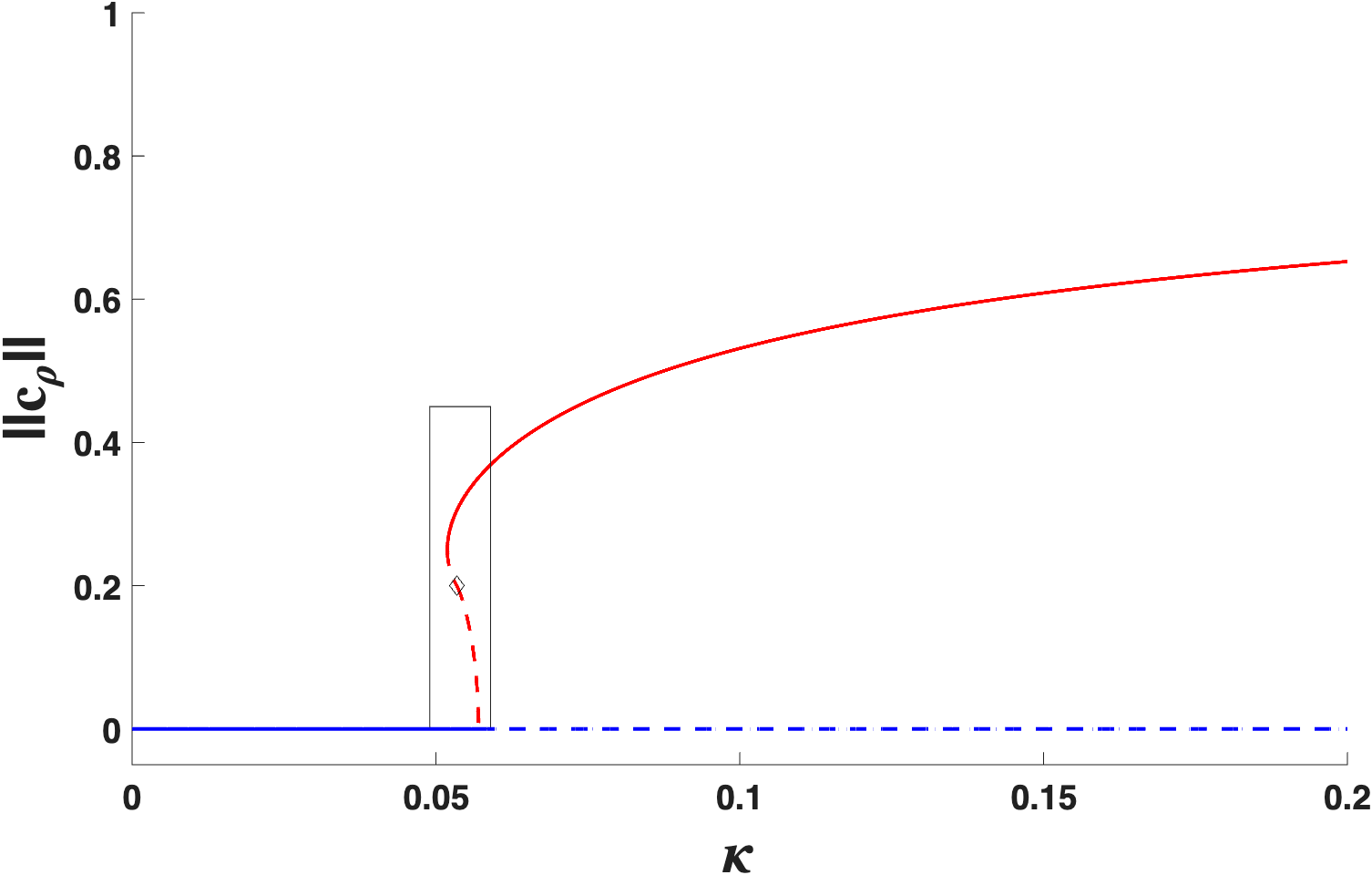} & 
 \includegraphics[width=0.48\textwidth]{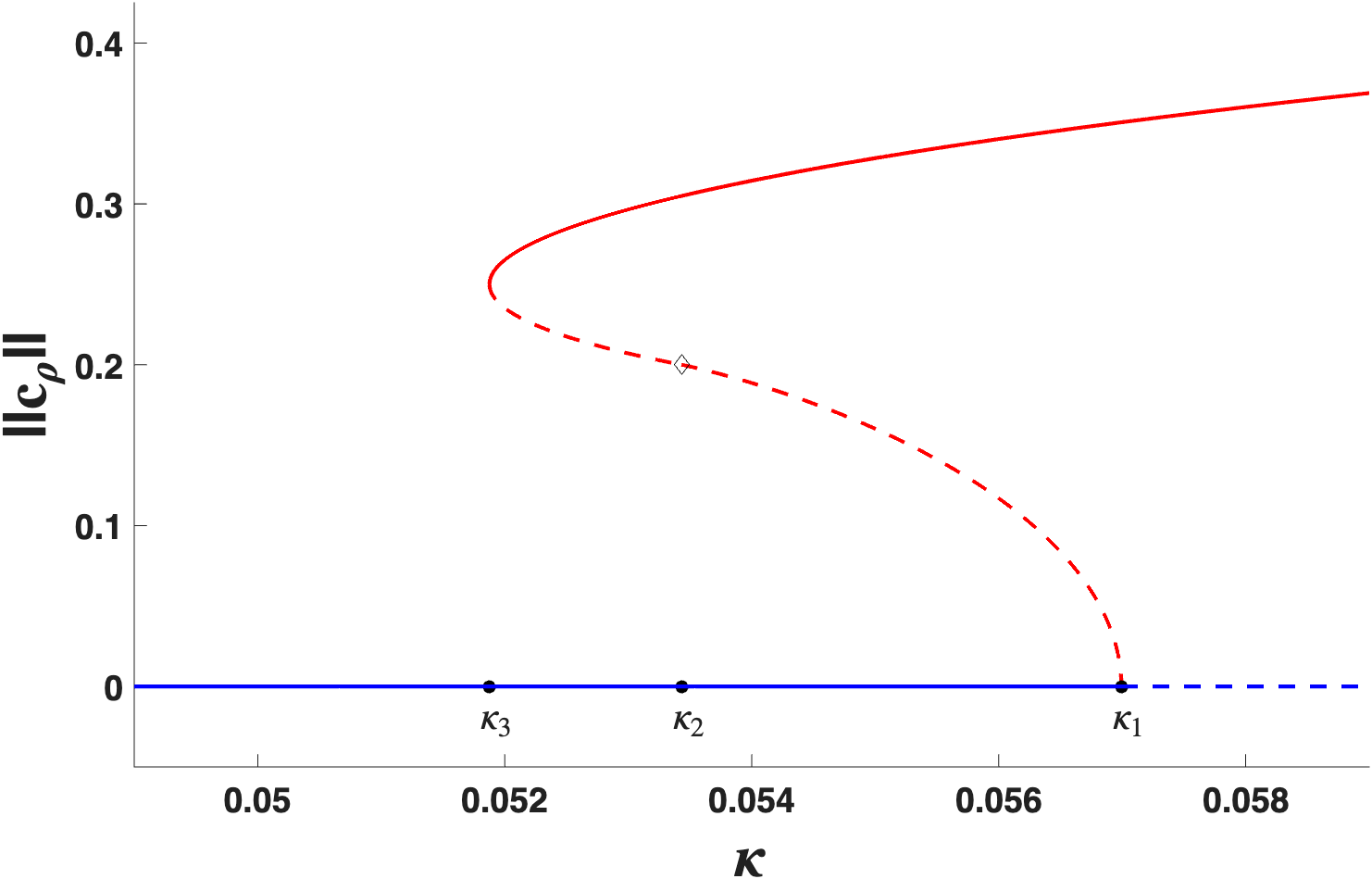} \\
 (a) $0<\kappa<0.2$ & (b) inset of rectangle from plot (a)
\end{tabular}
\caption{Case $\dm =2$, $m>2$. Plot of the norm of the centre of mass of  
$\rhou$ (blue), $\rho_{\kappa,1}$ (solid red) and $\rho_{\kappa,2}$ (dashed red) -- see Proposition \ref{prop:bif-mg2}. Plot(b) is the inset of the rectangle from plot (a). At $\kappa=\kappa_3$, a pair of equilibria with strict support in $\bbs^\dm$ gets born. The equilibrium $\rho_{\kappa,1}$ undergoes no further transitions and concentrates into a Dirac delta as $\kappa \to \infty$. On the other hand, $\rho_{\kappa,2}$ becomes fully supported at $\kappa = \kappa_2$ (transition indicated by a black diamond), and then merges with the uniform distribution at $\kappa=\kappa_1$. The numerical simulations correspond to $m=3$.}
\label{fig:m3-splot}
\end{center}
\end{figure}

\begin{figure}[htbp]
 \begin{center}
 \begin{tabular}{cc} 
 \includegraphics[width=0.48\textwidth]{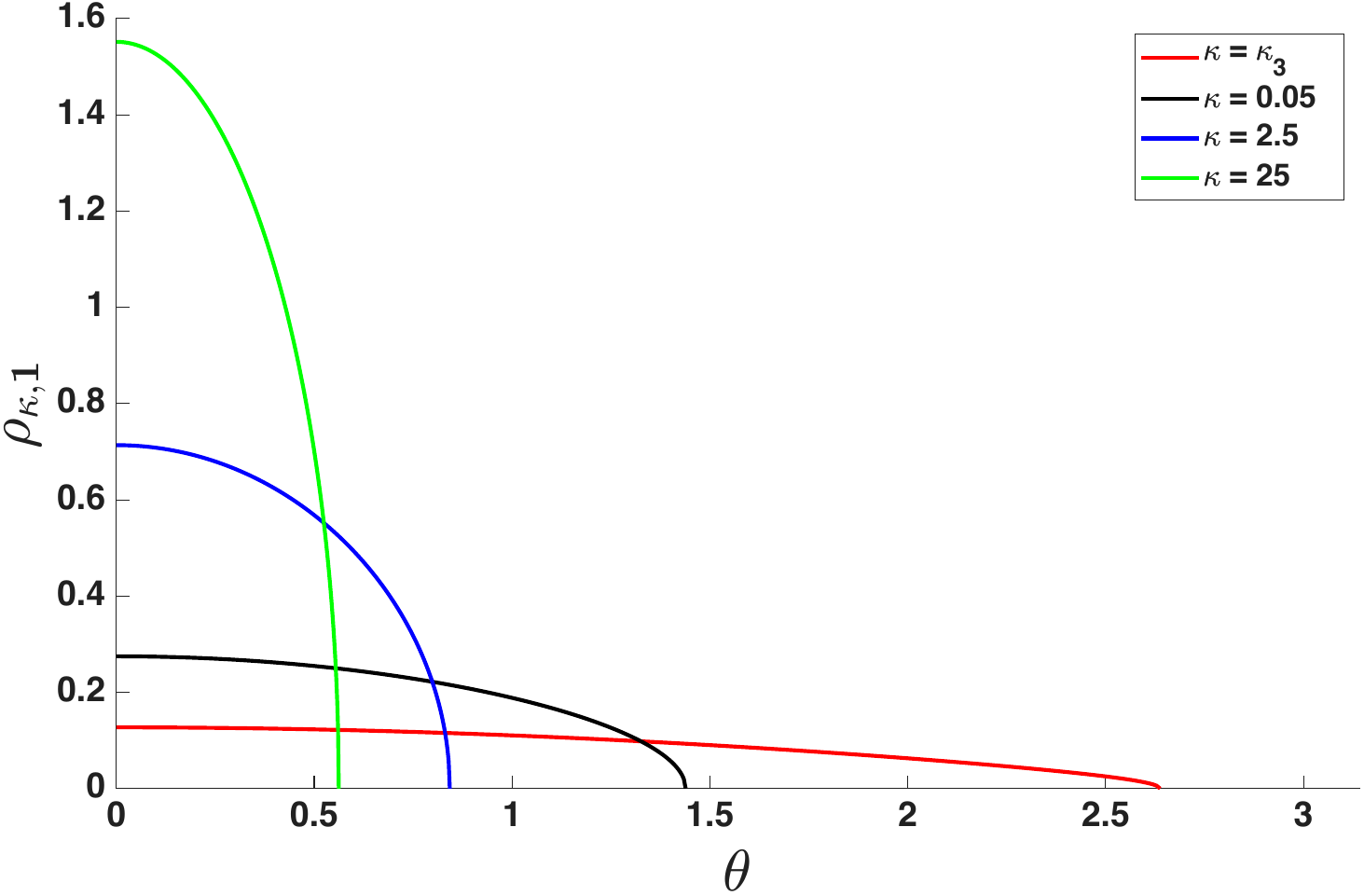} & 
 \includegraphics[width=0.48\textwidth]{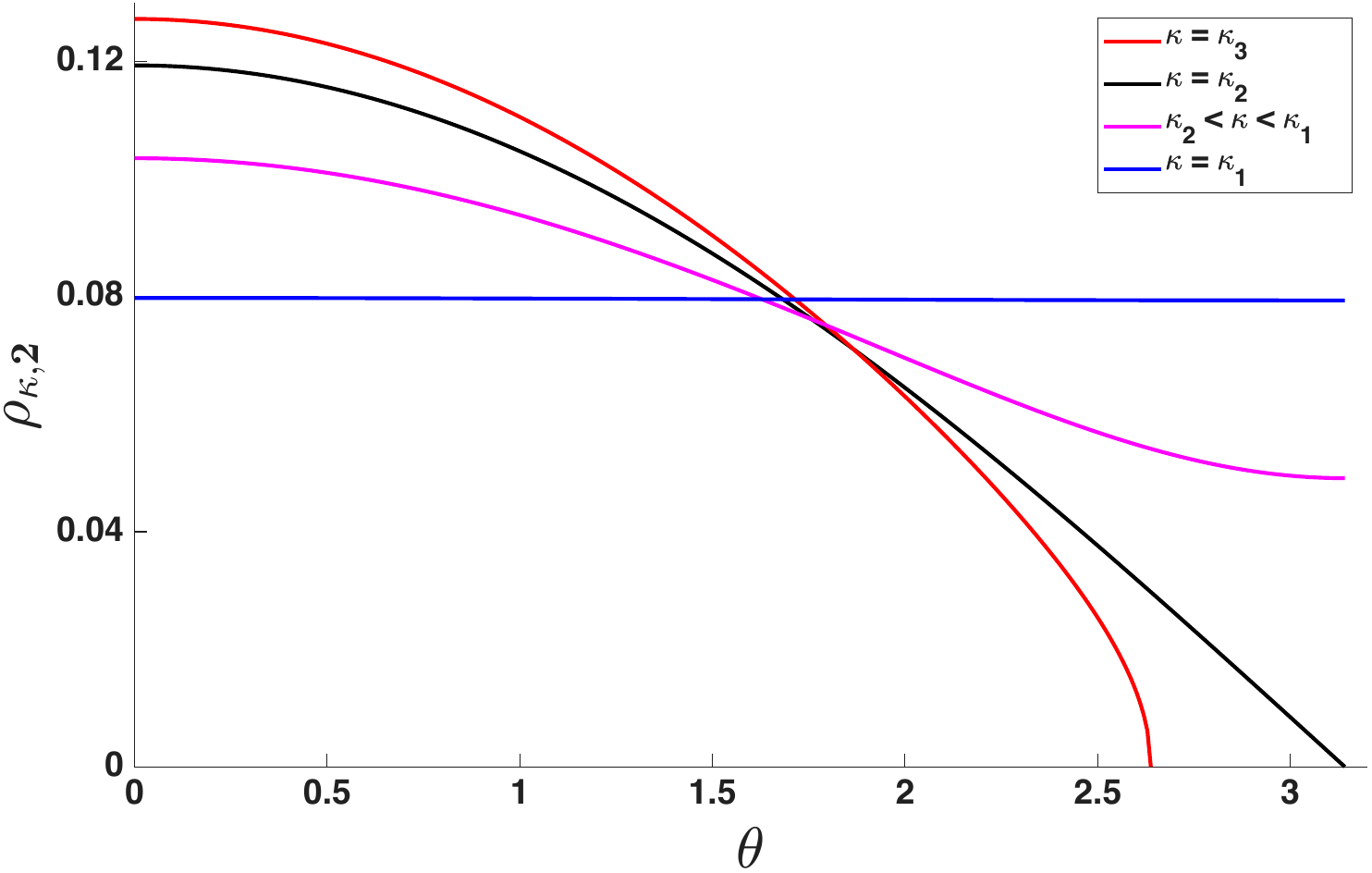} \\
 (a) Equilibrium $\rho_{\kappa, 1}$ & (b) Equilibrium $\rho_{\kappa, 2}$
\end{tabular}
\caption{Case $\dm =2$, $m>2$. Plot of equilibria $\rho_{\kappa,1}$ and $\rho_{\kappa,2}$ for various values of $\kappa$ (see \eqref{eqn:rhok-cs-mg2} and \eqref{eqn:rhok-fs-mg2}). Both equilibria form (and coincide) at $\kappa=\kappa_3$. The equilibrium $\rho_{\kappa,1}$ remains of strict support in $\bbs^\dm$ and concentrates at $\theta=0$ as $\kappa$ increases, while $\rho_{\kappa,2}$ transitions to a fully supported equilibrium at $\kappa=\kappa_2$, and at $\kappa=\kappa_1$ turns into the uniform distribution. We used $m=3$.}
\label{fig:m3-rho}
\end{center}
\end{figure}

\begin{proposition}[Critical $\kappa$ and equilibria for $m>2$, $\dm =2$]
\label{prop:bif-mg2}
For $\dm=2$ and any $m>2$, there exist three critical values of $\kappa$ ($\kappa_1>\kappa_2>\kappa_3$ defined by \eqref{eqn:kappa1}, \eqref{eqn:kappa2} and \eqref{eqn:kappa3}, respectively) such that:

\noindent i) At $\kappa=\kappa_3$, a pair of equilibria with support strictly contained in $\bbs^\dm$ gets born. For any $\kappa \in (\kappa_3,\kappa_2)$, there exist two solutions $0<\phi_{\kappa,1}<\phi_{\kappa,2}<\pi$ of \eqref{eqn:kappa-phi}; at $\kappa=\kappa_3$, the two solutions $\phi_{\kappa,1}$ and $\phi_{\kappa,2}$ coincide, with common value $\bar{\phi}$. Correspondingly, there exist two equilibria $\rho_{\kappa,1}$ and $\rho_{\kappa,2}$ of the form (see \eqref{eqn:equil-cs}):
\begin{equation}
\label{eqn:rhok-cs-mg2}
\rho_{\kappa,i}(x)=\begin{cases}
\left(\frac{m-1}{m}\right)^{\frac{1}{m-1}}\left(\lambda_{\kappa,i}+\kappa s_{\kappa,i}\cos\theta_x\right)^{\frac{1}{m-1}},\qquad&\text{ if }0\leq\theta_x\leq \phi_{\kappa,i},\\[5pt]
0,\qquad&\text{ otherwise},
\end{cases}
\end{equation}
with $s_{\kappa,i}$ and $\lambda_{\kappa,i}$ uniquely determined by $\phi_{\kappa,i}$ ($i=1,2$) -- see \eqref{eqn:lambdaphi} and \eqref{eqn:s-phi}.
\smallskip

\noindent ii) At $\kappa=\kappa_2$, the equilibrium $\rho_{\kappa,2}$ from part i) becomes fully supported on $\bbs^\dm$ (i.e., $\phi_{\kappa_2,2} = \pi$). For $\kappa_2<\kappa<\kappa_1$, $\rho_{\kappa,2}$ has the form (see  \eqref{eqn:equil-fs}):
\begin{equation}
\label{eqn:rhok-fs-mg2}
\rho_{\kappa,2}(x)=\left(\frac{m-1}{m}\right)^{\frac{1}{m-1}}\left(\lambda_{\kappa,2}+\kappa s_{\kappa,2} \cos\theta_x\right)^{\frac{1}{m-1}},\qquad\forall x\in \bbs^\dm,
\end{equation}
where $s_{\kappa,2}$ and $\lambda_{\kappa,2}$ are uniquely determined by $\kappa$ -- see \eqref{eqn:slambda} and \eqref{eqn:s-eta}. On the other hand, $\rho_{\kappa,1}$ undergoes no transition at $\kappa=\kappa_2$, and retains the form \eqref{eqn:rhok-cs-mg2}.
\smallskip

\noindent iii) At $\kappa=\kappa_1$, the fully supported equilibrium $\rho_{\kappa,2}$ from part ii) merges with the uniform distribution $\rhou$, and vanishes. And again,  $\rho_{\kappa,1}$ undergoes no transition at $\kappa=\kappa_1$; as $\kappa$ increases to infinity, $\phi_{\kappa,1}$ approaches $0$ and $\rho_{\kappa,1}$ concentrates into a Dirac delta.
\end{proposition}


\subsection{Energy minimizers}
\label{subsect:gmin-mg2}
It remains to discuss the energy minimizers. As noted in Remark \ref{rmk:calc-extend}, the calculations in the proof of Theorem \ref{thm:gminimizers} work for $m>2$ as well, with $\rho_\kappa$ replaced by either $\rho_{\kappa,1}$ or $\rho_{\kappa,2}$, depending on the context. In the subsequent discussion, we will be referring back to notations and calculations used in the proof of Theorem \ref{thm:gminimizers}.

For equilibria of full support $\rho_{\kappa,2}$ from \eqref{eqn:rhok-fs-mg2} ($\kappa_2<\kappa<\kappa_1$), the energy difference $E_[\rhou]-E[\rho_{\kappa,2}]$ depends on the bevaviour of the function $g_1(\eta) g_2(\eta)$. By Remark \ref{rmk:calc-extend}, since $H(\eta)$ is increasing, $g_1(\eta) g_2(\eta)$ is decreasing with $\eta$. Hence, as $\kappa$ increases from $\kappa_2$ to $\kappa_1$ (or equivalently, $\eta$ decreases from $-1$ to $-\infty$), $E[\rhou]-E[\rho_{\kappa,2}]$ increases from a negative value to zero.

For equlibria of strict support \eqref{eqn:rhok-cs-mg2}, the energy difference $E[\rhou]-E[\rho_{\kappa,i}]$ ($i=1,2$) depends on the bevaviour of the function $\bar{g}_1(\phi) \bar{g}_2(\phi)$. In turn, the latter depends on the monotonicity of $F(\phi)$, as given by Lemma \ref{lem:F-monotone-mg2}.  We have, for each equilibria: 

\noindent a) {\em $\rho_{\kappa,2}$ ($\phi$ increases from $\bar{\phi}$ to $\pi$).} Since $F(\phi)$ is decreasing, $\bar{g}_1(\phi) \bar{g}_2(\phi)$ increases on $(\bar{\phi},\pi)$ (cf., Remark \ref{rmk:calc-extend}). Hence, $E[\rhou]-E[\rho_{\kappa,2}]$ increases as $\phi$ increases from $\bar{\phi}$ to $\pi$ (or equivalently, $\kappa$ increases from $\kappa_3$ to $\kappa_2$). Combined with the previous observation (for the range $\kappa_2<\kappa<\kappa_1$), we find that as $\kappa$ increases from $\kappa_3$ to $\kappa_1$, $E[\rhou]-E[\rho_{\kappa,2}]$ increases from a negative value to zero. 
\smallskip

\noindent b) {\em $\rho_{\kappa,1}$ ($\phi$ decreases from $\bar{\phi}$ to $0$).}  Since $F(\phi)$ is increasing from $0$ to $\bar{\phi}$, $\bar{g}_1(\phi) \bar{g}_2(\phi)$ decreases on $(0,\bar{\phi})$ (cf., Remark \ref{rmk:calc-extend}). Hence, $E[\rhou]-E[\rho_{\kappa,1}]$ increases as $\phi$ decreases from $\bar{\phi}$ to $0$ (or equivalently, as $\kappa$ increases from $\kappa_3$ to $\infty$). Given that at $\kappa=\kappa_3$, $\rho_{\kappa,1}$ and $\rho_{\kappa,2}$ coincide, using part (a) we infer that $E[\rhou]-E[\rho_{\kappa,1}]<0$ at $\kappa=\kappa_3$. Therefore, $E[\rhou]-E[\rho_{\kappa,1}]$ increases from a negative value at $\kappa=\kappa_3$, but we have no more information yet on how the energies compare for a general $\kappa>\kappa_3$.



\begin{theorem}[Global energy minimizers for $m>2$, $d=2$]
\label{thm:gminimizers-mg2} 
Let $\dm=2$ and $m>2$. Then, there exists a fourth critical value of $\kappa$ (which we denote by $\kappa_c$), with $\kappa_3<\kappa_c<\kappa_1$, such that the global minimizer of \eqref{energy-sphere} on $\calP_{ac}(\bbs^\dm)$ is 
\begin{enumerate}
\item the uniform distribution $\rhou$ when $\kappa<\kappa_c$, \\[-7pt]
\item the equilibrium $\rho_{\kappa,1}$ given by \eqref{eqn:rhok-cs-mg2} when $\kappa > \kappa_c$.
\end{enumerate}
Moreover, the energy difference $E[\rhou] - E[\rho_{\kappa,1}]$ increases with $\kappa$ for $\kappa>\kappa_3$.
\end{theorem}
\begin{proof}
The considerations we made before listing the theorem imply that $\rhou$ is more energetically favourable than $\rho_{\kappa,2}$ for all $\kappa_3<\kappa<\kappa_1$. Also, $\rhou$ is more energetically favourable than $\rho_{\kappa,1}$ at $\kappa=\kappa_3$, and the difference $E[\rhou]-E[\rho_{\kappa,1}]$ increases for $\kappa>\kappa_3$ (but we have not yet quantified how). These behaviours can be observed in Figure \ref{fig:m3-energy}; note that for a better visualization we separated the energy plots in three parts, for different ranges of $\kappa$. To establish which equilibrium is the global minimizer we still have to compare $E[\rhou]$ and $E[\rho_{\kappa,1}]$, and $E[\rho_{\kappa,1}]$ and $E[\rho_{\kappa,2}]$. We will do this separately.
\smallskip

{\em Compare $E[\rhou]$ and $E[\rho_{\kappa,1}]$.} We use \eqref{eqn:E-diff} and \eqref{eqn:energy-phi} to write
\begin{equation}
    E[\rhou]-E[\rho_{\kappa,1}] = 
    \frac{1}{m-1}\int_{\bbs^\dm} \rhou^m(x)  \dS(x) + \bar{g}_1(\phi_{\kappa,1}) \bar{g}_2(\phi_{\kappa,1}).
\end{equation}

Also recall that $\phi_{\kappa,1}$ decreases to $0$ as $\kappa$ increases to $\infty$. For this reason, we will check the behaviour of $\bar{g}_1(\phi)\bar{g}_2(\phi)$ as $\phi$ tends to zero. 

Write $\bar{g}_1(\phi)$ from \eqref{eqn:g1g2bar} (with $\dm=2$) as
\begin{align}
    \bar{g}_1(\phi) &
    = (m-2)\int_0^{\phi}(\cos\theta-\cos \phi)^{\frac{1}{m-1}}\sin\theta\cos\theta\d\theta \\ 
    &\quad + 2\cos\phi \int_0^{\phi}(\cos\theta-\cos \phi)^{\frac{1}{m-1}}\sin \theta\d\theta.
\end{align}
Then, using $m>2$ and \eqref{eqn:int-g0} we get
\begin{align}
    & \bar{g}_1(\phi)\bar{g}_2(\phi) \\
    &\quad = \frac{(m-2)\int_0^{\phi}(\cos\theta-\cos \phi)^{\frac{1}{m-1}}\sin \theta\cos\theta\d\theta + 2\cos\phi \int_0^{\phi}(\cos\theta-\cos \phi)^{\frac{1}{m-1}}\sin \theta\d\theta}{2(m-1)(d\omega_d)^{m-1}\left(\int_0^\phi(\cos\theta-\cos \phi)^{\frac{1}{m-1}}\sin \theta\d\theta\right)^m}\\[2pt]
    &\quad \geq \frac{2\cos\phi \int_0^{\phi}(\cos\theta-\cos \phi)^{\frac{1}{m-1}}\sin \theta\d\theta}{2(m-1)(d\omega_d)^{m-1}\left(\int_0^\phi(\cos\theta-\cos\phi)^{\frac{1}{m-1}}\sin \theta\d\theta\right)^m}\\[2pt]
    &\quad = C \, \frac{\cos\phi}{\left(\int_0^\phi(\cos\theta-\cos\phi)^{\frac{1}{m-1}}\sin \theta\d\theta\right)^{m-1}},
\end{align}
where $C$ is a positive constant. 

As the r.h.s. above tends to $\infty$ as $\phi \to 0$, we infer that 
\[
\lim_{\phi \to 0} \bar{g}_1(\phi)\bar{g}_2(\phi) = \infty.
\]
The limit $\phi \to 0$ corresponds to $\kappa \to \infty$. Hence, $E[\rhou]-E[\rho_{\kappa,1}]$ is negative at $\kappa=\kappa_3$, increases on $(\kappa_3,\infty)$, and approaches $\infty$ as $\kappa \to \infty$. We conclude that there is a fourth critical value, which we denote by $k_c$ such that $E[\rhou]=E[\rho_{\kappa,1}]$ at $\kappa=\kappa_c$. Also, $E[\rhou]<E[\rho_{\kappa,1}]$ for $\kappa_3<\kappa<\kappa_c$, and $E[\rhou]>E[\rho_{\kappa,1}]$ for $\kappa>\kappa_c$ -- see Figure \ref{fig:m3-energy}, in particular plot (b) and its inset.
\smallskip

{\em Compare $E[\rho_{\kappa,1}]$ and $E[\rho_{\kappa,2}]$.} We will consider two cases: a) $\kappa_3<\kappa<\kappa_2$, and b) $\kappa_2<\kappa<\kappa_1$.

{\em Case a) $\kappa_3<\kappa<\kappa_2$.} We use \eqref{eqn:energy-s} and \eqref{eqn:energy-phi}, to write
\begin{equation}
\label{eqn:E12-diff}
E[\rho_{\kappa,2}]-E[\rho_{\kappa,1}] = \bar{g}_1(\phi_{\kappa,1}) \bar{g}_2(\phi_{\kappa,1}) - \bar{g}_1(\phi_{\kappa,2}) \bar{g}_2(\phi_{\kappa,2}).
\end{equation}
We will show that $E[\rho_{\kappa,2}]-E[\rho_{\kappa,1}]$ increases as $\kappa$ increases from $\kappa_3$ to $\kappa_2$. 

Compute the derivative of $\bar{g}_1(\phi_{\kappa,i})\bar{g}_2(\phi_{\kappa,i})$ with respect to $\kappa$, for $i=1,2$. By chain rule,
\begin{equation}
\label{eqn:g1g2p-chain}
    \frac{\d} {\ d\kappa}(\bar{g}_1(\phi_{\kappa,i})\bar{g}_2(\phi_{\kappa,i})) = \frac{\d (\bar{g}_1(\phi_{\kappa,i})\bar{g}_2(\phi_{\kappa,i})) }{\d \phi_{\kappa,i}} \, \frac{\d \phi_{\kappa,i}} {\d\kappa}.
\end{equation}
Now differentiate
\[
    F(\phi_{\kappa,i})^{-1}=\kappa,
\]
 with respect to $\kappa$, to get
 \[
-\frac{F'(\phi_{\kappa,i})}{F^2(\phi_{\kappa,i})}\frac{\d\phi_{\kappa,i}}{\d\kappa} =1.
 \]
From the above we get
\begin{equation}
\label{eqn:dphi_ki}
  \frac{\d\phi_{\kappa,i}}{\d\kappa} = -\frac{F^2(\phi_{\kappa,i})}{F'(\phi_{\kappa,i})}.
\end{equation}

By \eqref{eqn:g1bg2bp}, we also have
\begin{equation}
\label{eqn:dg1g2_phiki}
\frac{\d (\bar{g}_1(\phi_{\kappa,i})\bar{g}_2(\phi_{\kappa,i})) }{\d \phi_{\kappa,i}} = -(m-1) \bar{h}(\phi_{\kappa,i}) \frac{F'(\phi_{\kappa,i})}{F(\phi_{\kappa,i})},
\end{equation}
with $\bar{h}$ given by \eqref{eqn:barh}. Then, by combining \eqref{eqn:g1g2p-chain}, \eqref{eqn:dphi_ki} and \eqref{eqn:dg1g2_phiki}, we find
\begin{equation}
\label{eqn:dg1bg2b_dk}
\begin{aligned}
     \frac{\d} {\ d\kappa}(\bar{g}_1(\phi_{\kappa,i})\bar{g}_2(\phi_{\kappa,i})) &= (m-1) \bar{h}(\phi_{\kappa,i})  F(\phi_{\kappa,i})\\
     &= (m-1) \bar{h}(\phi_{\kappa,i})\kappa^{-1}.
\end{aligned}
\end{equation}

The equation above holds for $i=1,2$. Now, from \eqref{eqn:E12-diff} and \eqref{eqn:dg1bg2b_dk} we get
\begin{equation}
\label{eqn:dEdiff_dk}
\frac{\d}{\d \kappa} \left( E[\rho_{\kappa,2}]-E[\rho_{\kappa,1}] \right) = \kappa^{-1} (m-1) \left( \bar{h}(\phi_{\kappa,1}) - \bar{h}(\phi_{\kappa,2}) \right).
\end{equation}
Note that at each $\kappa$ fixed, we have $\phi_{\kappa,1} < \phi_{\kappa,2}$. If we show that $\bar{h}(\phi)$ is decreasing in $\phi$, then together with \eqref{eqn:dEdiff_dk}, we infer that $E[\rho_{\kappa,2}]-E[\rho_{\kappa,1}]$ increases with $\kappa$.


We will find that $\bar{h}'(\phi)<0$ by a direct calculation. By \eqref{eqn:barh} with $\dm=2$, write
\begin{equation}
\bar{h}(\phi) = \bar{C} \;  \frac{\int_0^\phi(\cos\theta-\cos \phi)^{\frac{1}{m-1}}\sin \theta\cos\theta\d\theta}{\left(\int_0^\phi(\cos\theta-\cos \phi)^{\frac{1}{m-1}}\sin \theta\d\theta\right)^{m}},
\end{equation}
where $\bar{C}>0$ is a constant. Both the numerator and the denominator above can be computed exactly. Indeed, by direct integration we find 
\[
\int_0^\phi(\cos\theta-\cos \phi)^{\frac{1}{m-1}}\sin \theta\cos\theta\d\theta = \frac{m-1}{2m-1} (1-\cos \phi)^{\frac{2m-1}{m-1}} + \frac{m-1}{m} \cos \phi (1-\cos \phi)^{\frac{m}{m-1}}, 
\]
and 
\[
\int_0^\phi(\cos\theta-\cos \phi)^{\frac{1}{m-1}}\sin \theta\d\theta = \frac{m-1}{m} (1-\cos \phi)^{\frac{m}{m-1}}, 
\]
and hence, by redenoting the constant $\bar{C}$ several times to include new constant terms, we have
\begin{align*}
\bar{h}(\phi) &= \bar{C} \;  \frac{\frac{1}{2m-1} (1-\cos \phi)^{\frac{2m-1}{m-1}} + \frac{1}{m} \cos \phi (1-\cos \phi)^{\frac{m}{m-1}}}{(1-\cos \phi)^{\frac{m^2}{m-1}}} \\
&=  \bar{C} \; \left( \frac{1}{2m-1} (1-\cos \phi)^{1-m} + \frac{1}{m} \cos \phi (1-\cos \phi)^{-m} \right)  \\[3pt]
&= \bar{C} (1-\cos \phi)^{-m} (m+(m-1) \cos \phi).
\end{align*}

Then, compute
\begin{align*}
{\bar{h}}' (\phi) &= -\bar{C} \sin \phi (1-\cos \phi)^{-m} \left( \frac{m(m+(m-1) \cos \phi)}{1-\cos \phi} + m-1\right) \\[3pt]
&= -\bar{C} \sin \phi (1-\cos \phi)^{-m-1} \left(m^2 + m-1 + (m-1)^2 \cos \phi \right)
\end{align*}
Finally, we have 
\begin{align*}
m^2 + m-1 + (m-1)^2 \cos \phi &\geq m^2 + m-1 - (m-1)^2 \\
& = 3m-2,
\end{align*}
and hence in the regime $m>2$, ${\bar{h}}' (\phi)<0$.
\smallskip

{\em Case b) $\kappa_2<\kappa<\kappa_1$.} 
In this range of $\kappa$, $\rho_{\kappa,2}$ is fully supported on $\bbs^\dm$. We use \eqref{eqn:energy-s}, \eqref{eqn:energy-eta} and \eqref{eqn:energy-phi} to write
\begin{equation}
\label{eqn:E12-diff-fs}
E[\rho_{\kappa,2}]-E[\rho_{\kappa,1}] = \bar{g}_1(\phi_{\kappa,1}) \bar{g}_2(\phi_{\kappa,1}) - {g}_1(\eta_\kappa) {g}_2(\eta_\kappa).
\end{equation}
We will show again that $E[\rho_{\kappa,2}]-E[\rho_{\kappa,1}]$ increases as $\kappa$ increases from $\kappa_2$ to $\kappa_1$.

From computations similar to those leading to \eqref{eqn:dphi_ki} and \eqref{eqn:dg1bg2b_dk}, we get
\[
\frac{\d\eta_\kappa}{\d\kappa}= -\frac{H^2(\eta_\kappa)}{H'(\eta_\kappa)},
\]
and by chain rule,
\[
\frac{\d }{\d\kappa} \left( g_1(\eta_\kappa)g_2(\eta_\kappa) \right) = (m-1)h(\eta_\kappa)\kappa^{-1},
\]
with $h$ given by \eqref{eqn:h}. Then, together with \eqref{eqn:dg1bg2b_dk}, we write \eqref{eqn:E12-diff-fs} as
\begin{equation}
\label{eqn:dEdiff_dk-fs}
    \frac{\d}{\d \kappa} \left( E[\rho_{\kappa,2}]-E[\rho_{\kappa,1}] \right) = \kappa^{-1} (m-1) \left( \bar{h}(\phi_{\kappa,1}) - h(\eta_\kappa) \right).
\end{equation}

By \eqref{eqn:h} with $\dm =2$, we have
\begin{equation}
h(\eta) = C \, \frac{\int_0^\pi(\cos\theta-\eta)^{\frac{1}{m-1}}\sin \theta\cos\theta\d\theta}{\left(\int_0^\pi(\cos\theta-\eta)^{\frac{1}{m-1}}\sin \theta\d\theta\right)^{m}},
\end{equation}
for $C>0$ constant. Inspect first the numerator

\[
f(\eta):=\int_0^\pi(\cos\theta-\eta)^{\frac{1}{m-1}}\sin \theta\cos\theta\d\theta.
\]
Compute the derivative:
\begin{align}
    f'(\eta) &= -\frac{1}{m-1}\int_0^\pi(\cos\theta-\eta)^{\frac{1}{m-1}-1}\sin \theta\cos\theta\d\theta.
\end{align}
We will show that $\int_0^\pi(\cos\theta-\eta)^{\frac{1}{m-1}-1}\sin \theta\cos\theta\d\theta\leq0$, and hence $f'(\eta) \geq 0$. Note first that as $m>2$, $ \frac{1}{m-1}-1 < 0$.
Then, for any $0 \leq  \tilde{\theta} \leq  \frac{\pi}{2}$, we have: $\left(\cos\left(\frac{\pi}{2}-\tilde{\theta}\right) -\eta\right)^{\frac{1}{m-1}-1} \leq  \left(\cos\left(\frac{\pi}{2}+\tilde{\theta}\right)-\eta\right)^{\frac{1}{m-1}-1}$ (by the monotonicity of cosine and $\frac{1}{m-1}-1 < 0$), $ \sin\left(\frac{\pi}{2}-\tilde{\theta}\right)=\sin\left(\frac{\pi}{2}+\tilde{\theta}\right)$ and $ 0\leq\cos\left(\frac{\pi}{2}-\tilde{\theta}\right)=-\cos\left(\frac{\pi}{2}+\tilde{\theta}\right)$. From these relations, we infer 
\[
\int_0^\frac{\pi}{2}(\cos\theta-\eta)^{\frac{1}{m-1}-1}\sin \theta\cos\theta\d\theta\leq-\int_\frac{\pi}{2}^\pi(\cos\theta-\eta)^{\frac{1}{m-1}-1}\sin \theta\cos\theta\d\theta,
\]
where in the integrals we change variables $\theta = \frac{\pi}{2} - \tilde{\theta}$ and $\theta = \frac{\pi}{2} +\tilde{\theta}$ respectively. Hence, we find
\begin{align}
    & \int_0^\pi(\cos\theta-\eta)^{\frac{1}{m-1}-1}\sin \theta\cos\theta\d\theta \\
    & \qquad = \int_0^\frac{\pi}{2}(\cos\theta-\eta)^{\frac{1}{m-1}-1}\sin \theta\cos\theta\d\theta + \int_\frac{\pi}{2}^\pi(\cos\theta-\eta)^{\frac{1}{m-1}-1}\sin \theta\cos\theta\d\theta\\
    & \qquad \leq 0.
\end{align}

It is also immediate that $\bar{f}(\eta):=\frac{1}{\left(\int_0^\pi(\cos\theta-\eta)^{\frac{1}{m-1}}\sin \theta\d\theta\right)^{m}}$ is an increasing function of $\eta$, for $-\infty<\eta<-1$. Therefore, as $h(\eta)$ is a product of two increasing functions, it is also increasing. Recall that $\eta_\kappa$ decreases from $-1$ to $-\infty$ as $\kappa$ increases from $\kappa_2$ to $\kappa_1$. So overall, $h(\eta_\kappa)$ decreases as $\kappa$ increases from $\kappa_2$ to $\kappa_1$, and in particular we have
\[
h(-1) \geq h(\eta_\kappa), \qquad \text{ for all } \kappa_2<\kappa<\kappa_1.
\]

On the other hand, by the monotonicity of $\bar{h}$ (established as in case a)), we have 
\[
\bar{h}(\phi_{\kappa,1}) > \bar{h}(\phi_{{\kappa_2},1}), \qquad \text{ for all } \kappa_2<\kappa<\kappa_1,
\]
and case a), at $\kappa = \kappa_2$ it holds that
\[
\bar{h}(\phi_{{\kappa_2},1}) > \bar{h}(\phi_{{\kappa_2},2}).
\]
Finally, observe that $\phi_{\kappa_2,2}=\pi$ and
$\bar{h}(\phi_{\kappa_2,2})= h(-1)$, and by combining the inequalities above we find
\[
\bar{h}(\phi_{\kappa,1}) > h(\eta_\kappa), \qquad \text{ for all } \kappa_2<\kappa<\kappa_1.
\]
From this result and \eqref{eqn:dEdiff_dk-fs}, we infer that $ E[\rho_{\kappa,2}]-E[\rho_{\kappa,1}]$ increases with $\kappa$ on $(\kappa_2,\kappa_1)$.

The comparison of energies above provides the global minimum as in the statement of the theorem. Finally, using the fact that $\rhou$ and $\rho_{\kappa,2}$ coincide at $\kappa=\kappa_1$, we infer that $E[\rhou]>E[\rho_{\kappa,1}]$ at $\kappa=\kappa_1$. Hence, the fourth critical value $\kappa_c$ (where $E[\rhou]-E[\rho_{\kappa,1}]$ changes sign) is in the range $\kappa_3<\kappa_c<\kappa_1$. 
\end{proof}

The numerical illustration in Figure \ref{fig:m3-energy} shows separately three ranges of $\kappa$. Various facts that appear in the proof of Theorem \ref{thm:gminimizers-mg2} can be observed in the figure; e.g., the monotonicty with respect to $\kappa$ of $E[\rhou]-E[\rho_{\kappa,1}]$, $E[\rhou]-E[\rho_{\kappa,2}]$ and $E[\rho_{\kappa,2}]-E[\rho_{\kappa,1}]$. In Figure \ref{thm:gminimizers-mg2}(b) we indicate by a triangle the point where $E[\rhou]=E[\rho_{\kappa,1}]$, i.e., at the fourth critical value $\kappa_c$.


\begin{figure}[htbp]
 \begin{center}
 \begin{tabular}{cc}
 \includegraphics[width=0.5\textwidth]{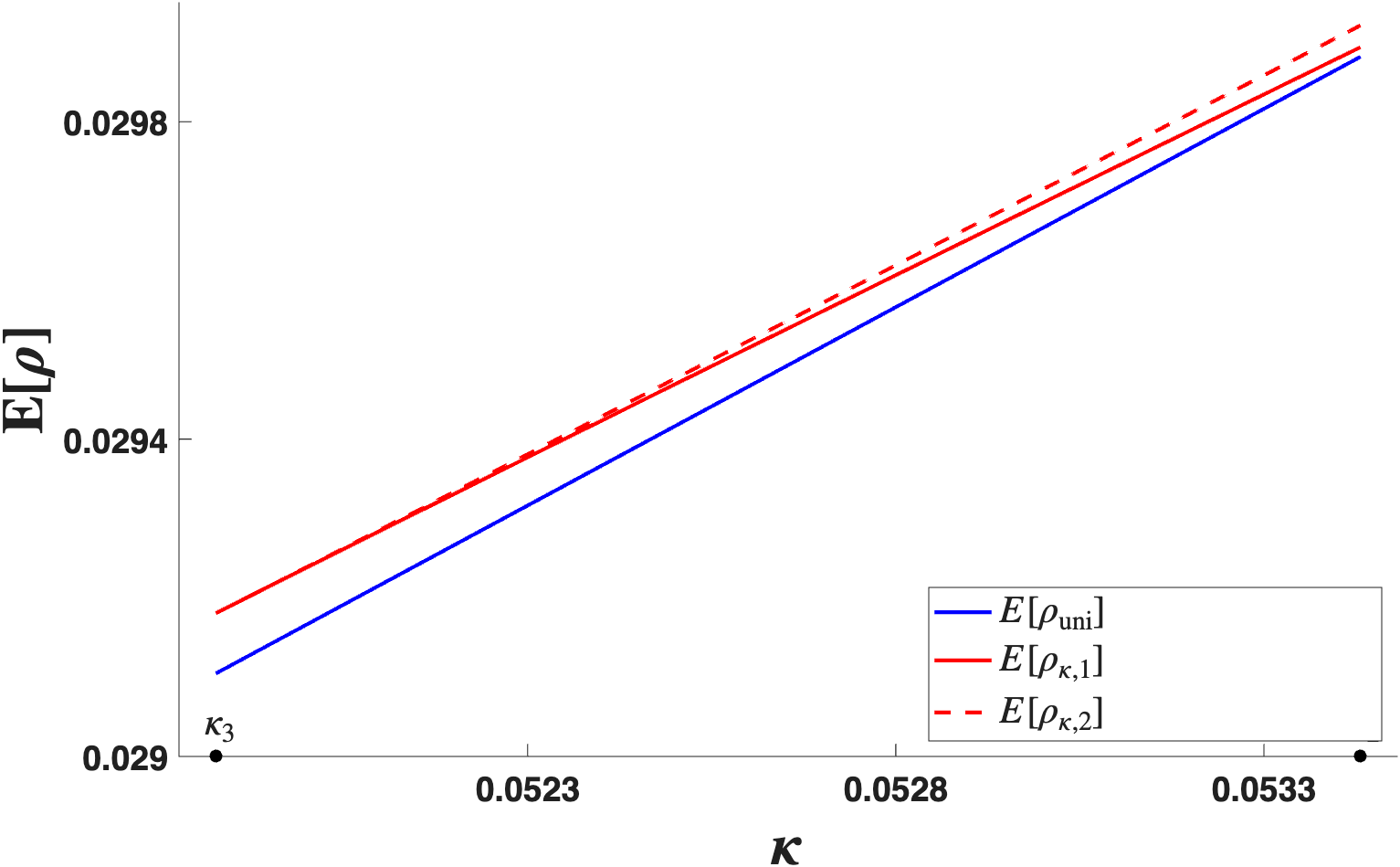} & \\
 (a) $\kappa_3<\kappa<\kappa_2$ & \\[10pt]
\includegraphics[width=0.5\textwidth]{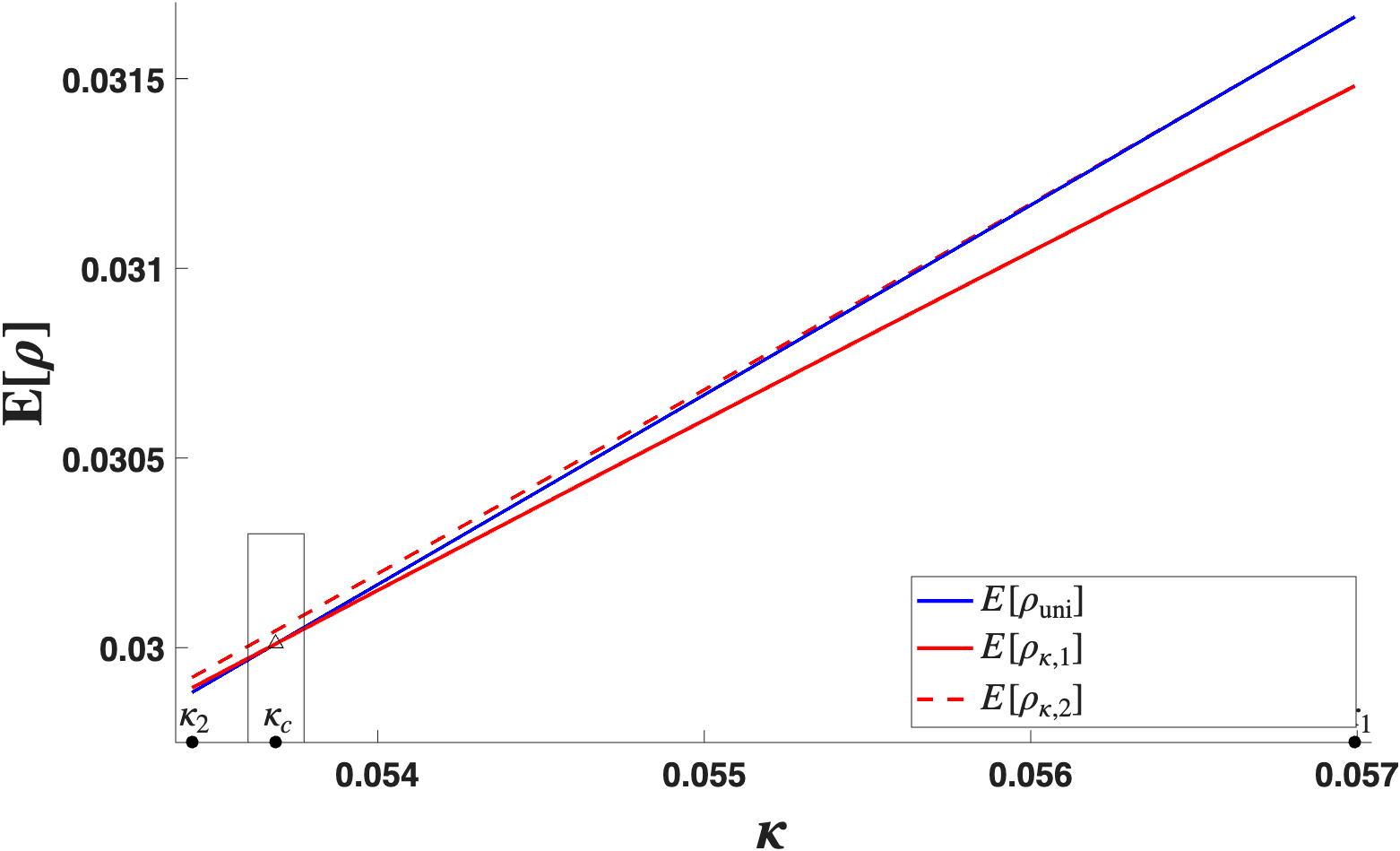} &  \includegraphics[width=0.45\textwidth]{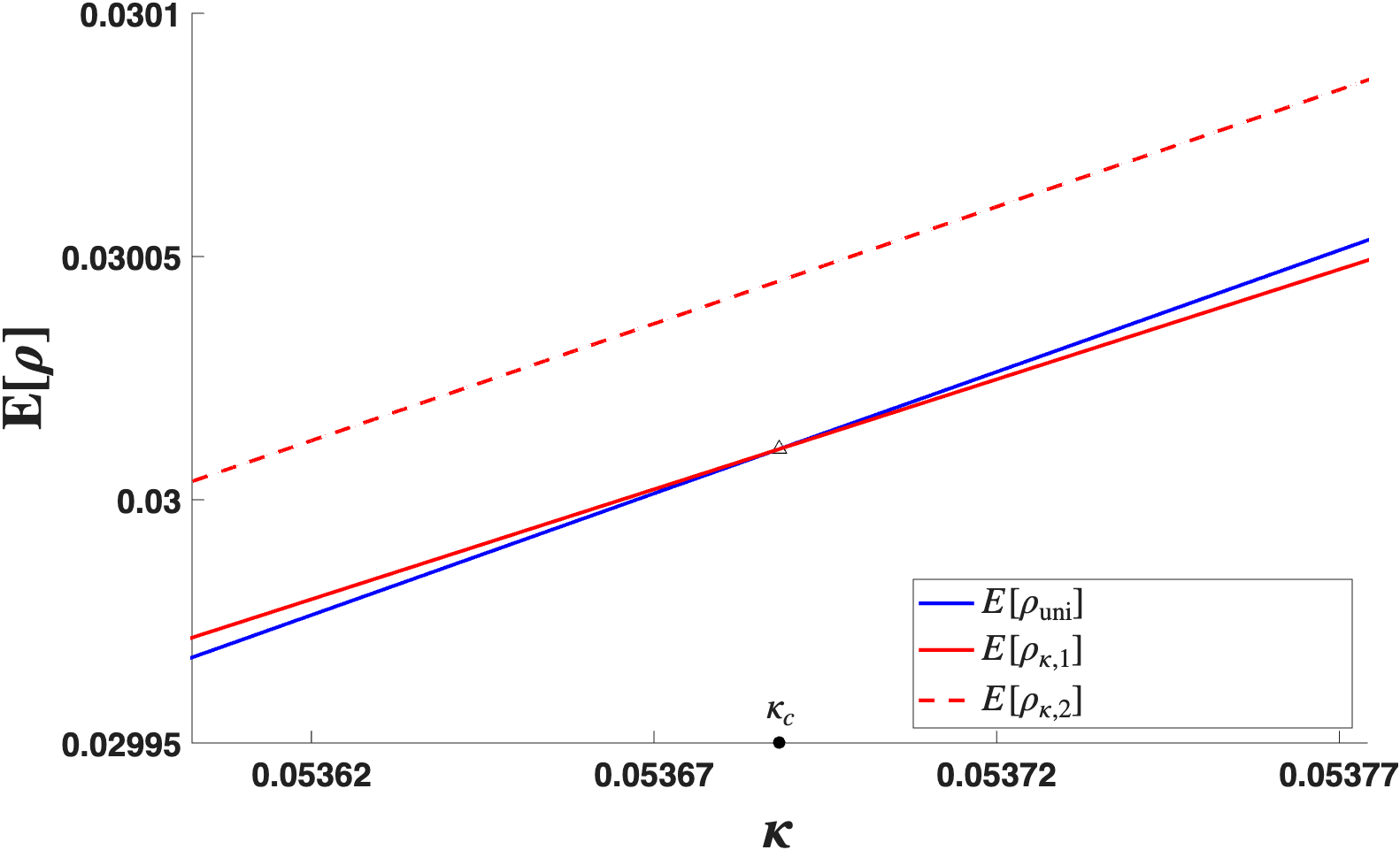} \\
 (b) $\kappa_2<\kappa<\kappa_1$ & (b)'  inset of the rectangle from plot (b) \\[10pt]
 \includegraphics[width=0.5\textwidth]{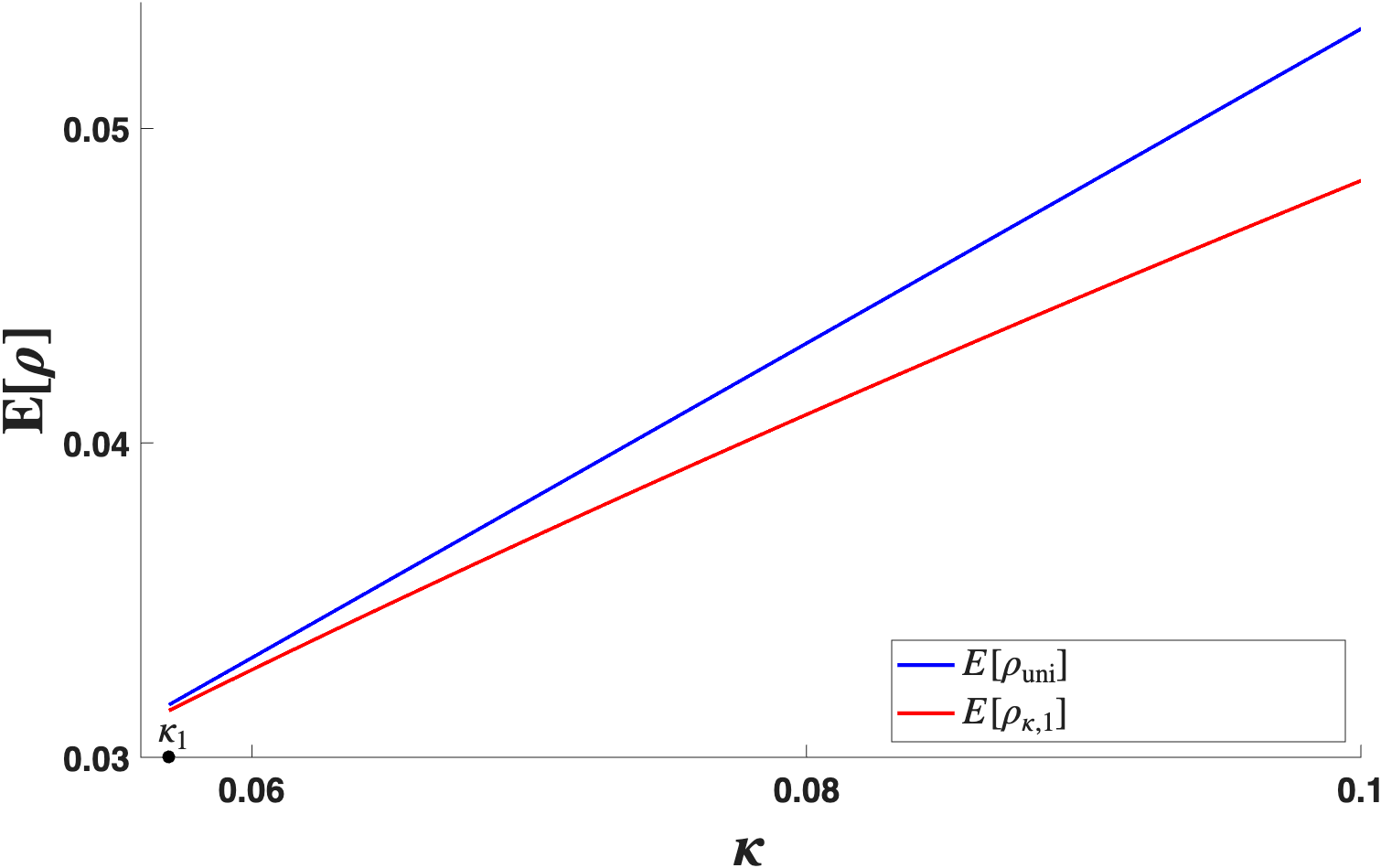} & \\
 (c) $\kappa> \kappa_1$ & 
\end{tabular}
\caption{Case $\dm =2$, $m>2$. Energies of the various equilibria for different ranges of $\kappa$. (a) As $\kappa$ increases from $\kappa_3$ to $\kappa_2$, the energy difference $E[\rhou]-E[\rho_{\kappa,1}]$ increases, starting from a negative value. Also, the energy difference $E[\rho_{\kappa,2}]-E[\rho_{\kappa,1}]$ increases from $0$ to a positive value. (b) The energy differences $E[\rhou]-E[\rho_{\kappa,1}]$ and $E[\rho_{\kappa,2}]-E[\rho_{\kappa,1}]$ continue to increase as $\kappa$ ranges from $\kappa_2$ to $\kappa_1$. There is a fourth critical value of $\kappa$, $\kappa_c \in (\kappa_3,\kappa_1)$, such that $E[\rhou]-E[\rho_{\kappa,1}]$ changes sign from negative to positive there -- this transition, indicated by a triangle, can be better observed in the inset from plot (b)'. (c) The energy difference $E[\rhou]-E[\rho_{\kappa,1}]$ increases indefinitely as $\kappa \to \infty$. The global minimizers are: $\rhou$ for $\kappa_3<\kappa<\kappa_c$, and $\rho_{\kappa,1}$ for $\kappa>\kappa_c$. The simulations are for $m=3$.}
\label{fig:m3-energy}
\end{center}
\end{figure}


\appendix
\section{Supporting result for Lemma \ref{lem:H-monotone}}
\label{appendix:H-monotone}
\begin{lemma}\label{sign-of-functionalH}
For any given real number $m>1$ and natural number $\dm$, define the following function on $\xi>1$:
\begin{equation}
\label{eqn:Hmd}    
\begin{aligned}
\mathcal{H}_{m, \dm}(\xi)=&\left(\int_0^\pi (\xi+\cos\theta)^{\frac{1}{m-1}-1}\sin^{\dm-1}\theta \d\theta\right)\left(
\int_0^\pi (\xi+\cos\theta)^{\frac{1}{m-1}-1}\sin^{\dm+1}\theta\d\theta
\right)\\[3pt]
&\quad - \left(\int_0^\pi (\xi+\cos\theta)^{\frac{1}{m-1}-2}\sin^{\dm+1}\theta\d\theta\right)\left(\int_0^\pi (\xi+\cos\theta)^{\frac{1}{m-1}}\sin^{\dm-1}\theta \d\theta\right).
\end{aligned}
\end{equation}
Then, we distinguish the following cases:
\begin{enumerate}
\item if $\frac{1}{m-1}+\frac{\dm-1}{2}<0$ then $\mathcal{H}_{m, \dm}(\xi)>0$ for all $\xi>1$, \\[-5pt]
\item if $\frac{1}{m-1}+\frac{\dm-1}{2}=0$ then $\mathcal{H}_{m, \dm}(\xi)=0$ for all $\xi>1$, and\\[-5pt]
\item if $\frac{1}{m-1}+\frac{\dm-1}{2}>0$ then $\mathcal{H}_{m, \dm}(\xi)<0$ for all $\xi>1$.
\end{enumerate}
In other words, we have
\[
\mathrm{sgn}(\mathcal{H}_{m,\dm}(\xi))=-\mathrm{sgn}\left(\frac{1}{m-1}+\frac{\dm-1}{2}\right), \qquad \text{ for all }\xi>1.
\]
\end{lemma}

\begin{proof}
 From Section 3.666, equation (2) of \cite{gradshteyn2014table}, we have:
\[
\int_0^\pi (\cosh \beta+\sinh\beta \cos \theta)^{\mu+\nu}\sin^{-2\nu}\theta\d \theta=\frac{\sqrt{\pi}}{2^\nu}\sinh^{\nu}\beta~\Gamma\left(\frac{1}{2}-\nu\right)P_\mu^\nu(\cosh \beta)\quad\forall\mathrm{Re}(\nu)<\frac{1}{2},
\]
where $P_\mu^\nu$ denotes the associated Legendre function of the first kind. By factoring out $\sinh \beta$ from the l.h.s., the equation can be written as
\begin{equation}
\label{eqn:int-Legendre}
\int_0^\pi (\coth \beta+ \cos \theta)^{\mu+\nu}\sin^{-2\nu}\theta\d \theta=\frac{\sqrt{\pi}}{2^\nu}\sinh^{-\mu}\beta~\Gamma\left(\frac{1}{2}-\nu\right)P_\mu^\nu(\cosh \beta).
\end{equation}

Denote $\mu_0=\frac{1}{m-1}-2+\frac{\dm+1}{2}$, $\nu_0=-\frac{\dm+1}{2}$ and substitute $\xi=\coth\beta$ in \eqref{eqn:Hmd}, and use \eqref{eqn:int-Legendre} to simplify each integral in $\mathcal{H}_{m, \dm}(\xi)$ as follows:
\begin{align*}
\begin{cases}
\displaystyle\int_0^\pi (\xi+\cos\theta)^{\frac{1}{m-1}-1}\sin^{\dm-1}\theta \d\theta&= \displaystyle\frac{\sqrt{\pi}}{2^{\nu_0+1}}\sinh^{-\mu_0}\beta~\Gamma\left(\frac{1}{2}-(\nu_0+1)\right)P_{\mu_0}^{\nu_0+1}(\cosh \beta),
\vspace{0.2cm}
\\
\displaystyle\int_0^\pi (\xi+\cos\theta)^{\frac{1}{m-1}-1}\sin^{\dm+1}\theta\d\theta&=\displaystyle\frac{\sqrt{\pi}}{2^{\nu_0}}\sinh^{-(\mu_0+1)}\beta~\Gamma\left(\frac{1}{2}-\nu_0\right)P_{\mu_0+1}^{\nu_0}(\cosh \beta),
\vspace{0.2cm}
\\
\displaystyle\int_0^\pi (\xi+\cos\theta)^{\frac{1}{m-1}-2}\sin^{\dm+1}\theta\d\theta&=\displaystyle\frac{\sqrt{\pi}}{2^{\nu_0}}\sinh^{-\mu_0}\beta~\Gamma\left(\frac{1}{2}-\nu_0\right)P_{\mu_0}^{\nu_0}(\cosh \beta),
\vspace{0.2cm}
\\
\displaystyle\int_0^\pi (\xi+\cos\theta)^{\frac{1}{m-1}}\sin^{\dm-1}\theta \d\theta&=\displaystyle\frac{\sqrt{\pi}}{2^{\nu_0+1}}\sinh^{-(\mu_0+1)}\beta~\Gamma\left(\frac{1}{2}-(\nu_0+1)\right)P_{\mu_0+1}^{\nu_0+1}(\cosh \beta).
\vspace{0.2cm}
\end{cases}
\end{align*}

Hence, we can write $\mathcal{H}_{m, \dm}(\xi)$ as
\begin{align*}
\mathcal{H}_{m, \dm}(\xi)&=\frac{\pi}{2^{2\nu_0+1}}\sinh^{-(2\mu_0+1)}\beta \; \Gamma\left(\frac{1}{2}-\nu_0\right)\Gamma\left(-\frac{1}{2}-\nu_0\right)\\
&\qquad\times\left(P_{\mu_0}^{\nu_0+1}(\cosh \beta) P_{\mu_0+1}^{\nu_0}(\cosh \beta) -P_{\mu_0}^{\nu_0}(\cosh \beta)P_{\mu_0+1}^{\nu_0+1}(\cosh \beta)\right).
\end{align*}

Here, the range of $\beta$ is $(0,\infty)$ since the range of $\xi$ is $(1, \infty)$. Equation (1) in Section 8.715 of \cite{gradshteyn2014table} provides the following integral representation of the associated Legendre function of the first kind:
\[
P^\nu_\mu(\cosh\beta)=\frac{\sqrt{2}\sinh^\nu\beta}{\sqrt{\pi}\Gamma\left(\frac{1}{2}-\nu\right)}\int_0^\beta\frac{\cosh\left(\mu+\frac{1}{2}\right)t }{\left(\cosh \beta-\cosh t\right)^{\nu+\frac{1}{2}}}\d t.
\]

By using the integral expression above, we can rewrite $\mathcal{H}_{m, \dm}(\xi)$ as
\begin{align*}
\mathcal{H}_{m, \dm}(\xi)&=\frac{\pi}{2^{2\nu_0+1}}\sinh^{-(2\mu_0+1)}\beta \; \Gamma\left(\frac{1}{2}-\nu_0\right)\Gamma\left(-\frac{1}{2}-\nu_0\right)\times\frac{2\sinh^{2\nu_0+1}\beta}{\pi\Gamma\left(\frac{1}{2}-\nu_0\right)\Gamma\left(-\frac{1}{2}-\nu_0\right)}\\[3pt]
&\quad \times\bigg(
\int_0^\beta \frac{\cosh\left(\mu_0+\frac{1}{2}\right)t}{\left(\cosh \beta-\cosh t\right)^{\nu_0+\frac{3}{2}}}\d t 
\int_0^\beta \frac{\cosh\left(\mu_0+\frac{3}{2}\right)t}{\left(\cosh \beta-\cosh t\right)^{\nu_0+\frac{1}{2}}}\d t
\\[3pt]
&\qquad-\int_0^\beta \frac{\cosh\left(\mu_0+\frac{1}{2}\right)t}{\left(\cosh \beta-\cosh t\right)^{\nu_0+\frac{1}{2}}}\d t 
\int_0^\beta \frac{\cosh\left(\mu_0+\frac{3}{2}\right)t}{\left(\cosh \beta-\cosh t\right)^{\nu_0+\frac{3}{2}}}\d t
\bigg).
\end{align*}

Finally, we will investigate the sign of
\begin{align*}
\tilde{\mathcal{H}}_{m, \dm}(\xi):=&
\int_0^\beta \frac{\cosh\left(\mu_0+\frac{1}{2}\right)t}{\left(\cosh \beta-\cosh t\right)^{\nu_0+\frac{3}{2}}}\d t \int_0^\beta \frac{\cosh\left(\mu_0+\frac{3}{2}\right)t}{\left(\cosh \beta-\cosh t\right)^{\nu_0+\frac{1}{2}}}\d t
\\[3pt]
&\quad-\int_0^\beta \frac{\cosh\left(\mu_0+\frac{1}{2}\right)t}{\left(\cosh \beta-\cosh t\right)^{\nu_0+\frac{1}{2}}}\d t
\int_0^\beta \frac{\cosh\left(\mu_0+\frac{3}{2}\right)t}{\left(\cosh \beta-\cosh t\right)^{\nu_0+\frac{3}{2}}}\d t,
\end{align*}
for $\beta \in (0,\infty)$.

The sign of 
\[
A:= \mu_0+1 = \frac{1}{m-1}+\frac{\dm-1}{2}
\]
will be crucial in determining the sign of $\tilde{\mathcal{H}}_{m, \dm}(\xi)$ (and hence, of $\mathcal{H}_{m, \dm}(\xi)$). We express $\tilde{\mathcal{H}}_{m, \dm}(\xi)$ using $A$ instead of $\mu_0$ as follows:
\begin{align*}
\tilde{\mathcal{H}}_{m, \dm}(\xi)=&
\int_0^\beta \frac{\cosh\left(A-\frac{1}{2}\right)t}{\left(\cosh \beta-\cosh t\right)^{\nu_0+\frac{3}{2}}}\d t \int_0^\beta \frac{\cosh\left(A+\frac{1}{2}\right)t}{\left(\cosh \beta-\cosh t\right)^{\nu_0+\frac{1}{2}}}\d t \\[3pt]
&\quad-\int_0^\beta \frac{\cosh\left(A-\frac{1}{2}\right)t}{\left(\cosh \beta-\cosh t\right)^{\nu_0+\frac{1}{2}}}\d t
\int_0^\beta \frac{\cosh\left(A+\frac{1}{2}\right)t}{\left(\cosh \beta-\cosh t\right)^{\nu_0+\frac{3}{2}}}\d t\\[3pt]
=& \int_0^\beta \frac{\cosh\left|A-\frac{1}{2}\right|t}{\left(\cosh \beta-\cosh t\right)^{\nu_0+\frac{3}{2}}}\d t \int_0^\beta \frac{\cosh\left|A+\frac{1}{2}\right|t}{\left(\cosh \beta-\cosh t\right)^{\nu_0+\frac{1}{2}}}\d t \\[3pt]
&\quad-\int_0^\beta \frac{\cosh\left|A-\frac{1}{2}\right|t}{\left(\cosh \beta-\cosh t\right)^{\nu_0+\frac{1}{2}}}\d t
\int_0^\beta \frac{\cosh\left|A+\frac{1}{2}\right|t}{\left(\cosh \beta-\cosh t\right)^{\nu_0+\frac{3}{2}}}\d t,
\end{align*}
where in the second equality we used that $\cosh(\cdot)$ is an even function. 

Now define, for $t >0$:
\[
\d \zeta:=\left(\frac{\cosh\left|A+\frac{1}{2}\right|t}{\left(\cosh \beta-\cosh t\right)^{\nu_0+\frac{3}{2}}}\right)\d t, \quad f(t)=\cosh\beta-\cosh t,\quad g(t)=\frac{\cosh\left|A-\frac{1}{2}\right|t}{\cosh\left|A+\frac{1}{2}\right|t}.
\]
Hence, we write
\[
\tilde{\mathcal{H}}_{m, \dm}(\xi)=\int_0^\beta f(t)\d\zeta\int_0^\beta g(t)\d\zeta-\int_0^\beta f(t)g(t)\d\zeta\int_0^\beta \d\zeta.
\]

The function $f$ is decreasing on $t\geq 0$. Also, $g'$ can be calculated from
\[
\frac{g'(t)}{g(t)} =\left|A-\frac{1}{2}\right|\tanh\left|A-\frac{1}{2}\right|t-\left|A+\frac{1}{2}\right|\tanh\left|A+\frac{1}{2}\right|t,
\]
which is positive for $A<0$, negative for $A>0$ and zero for $A=0$. 

The conclusion now follows from the integral form of the Chebyshev inequality \cite[Section 11.31]{gradshteyn2014table} (we account here for the monotonicities of the functions $f$ and $g$). We find: \\[3pt]
\noindent {\em Case 1: $A<0$.} Since $f$ is decreasing and $g$ is increasing, $\mathcal{H}_{m, \dm}(\xi)>0$ for all $\xi>1$.\\[2pt]
\noindent {\em Case 2: $A=0$.} In this case, $g(t)=1$ and $\mathcal{H}_{m, \dm}(\xi)\equiv0$ for all $\xi>1$.\\[2pt]
\noindent {\em Case 3: $A>0$.} Since $f$ and $g$ are both decreasing, $\mathcal{H}_{m, \dm}(\xi)<0$ for all $\xi>1$.

\end{proof}


\section{Proof of Lemma \ref{lem:Hlim}}
\label{appendix:Hlim}

We show that
\[
\lim_{ \eta\rightarrow \infty}H(-\eta)= \kappa_1^{-1},
\] 
where $\kappa_1$ is given by \eqref{eqn:kappa1} and the function $H$ by \eqref{eqn:H}. Compute
\begin{equation}
\label{eqn:H-meta-1}
\begin{aligned}
H(-\eta) &= \frac{m-1}{m} (\dm w_\dm)^{m-1} \left(\int_0^\pi (\cos\theta+\eta)^{\frac{1}{m-1}}\sin^{\dm-1}\theta \cos\theta\d\theta\right)
\left(\int_0^\pi (\cos\theta
+\eta)^{\frac{1}{m-1}}\sin^{\dm-1}\theta\d\theta\right)^{m-2} \\
&= \frac{m-1}{m} (\dm w_\dm)^{m-1} \eta
\left(\int_0^\pi \left( \frac{\cos\theta}{\eta}+1 \right)^{\frac{1}{m-1}}\sin^{\dm-1}\theta \cos\theta\d\theta\right)
\left(\int_0^\pi \left(\frac{\cos\theta}{\eta} +1 \right)^{\frac{1}{m-1}}\sin^{\dm-1}\theta\d\theta\right)^{m-2}.
\end{aligned}
\end{equation}

We will use Newton's generalized binomial theorem: for any $r\in \bbr$ and $x \in \bbr$ with $|x|<1$, it holds that
\begin{align*}
    (1+x)^r = \sum_{k=0}^\infty \binom{r}{k} x^k,
\end{align*}
where $ \binom{r}{k}= \frac{r(r-1)....(r-k+1)}{k!}$, with the convention $\binom{r}{0}=1$. Then, as the binomial series is uniformly convergent, we get from \eqref{eqn:H-meta-1}:
\begin{align}
H(-\eta) &= \frac{m-1}{m} (\dm w_\dm)^{m-1} \eta
\left(\sum_{k=0}^{\infty}\binom{\frac{1}{m-1}}{k}\frac{1}{\eta^k} \int_0^\pi \cos^{k+1}\theta \sin^{d-1}\theta \d\theta\right) \\
&\qquad \times
\left(\sum_{k=0}^{\infty}\binom{\frac{1}{m-1}}{k}\frac{1}{\eta^k} \int_0^\pi \cos^k\theta \sin^{d-1}\theta \d\theta\right)^{m-2}\\
&= \frac{m-1}{m} (\dm w_\dm)^{m-1} 
\left(\sum_{k=1}^{\infty}\binom{\frac{1}{m-1}}{k}\frac{1}{\eta^{k-1}} \int_0^\pi \cos^{k+1}\theta \sin^{d-1}\theta \d\theta\right) \\
&\qquad \times
\left(\sum_{k=0}^{\infty}\binom{\frac{1}{m-1}}{k}\frac{1}{\eta^k} \int_0^\pi \cos^k\theta \sin^{d-1}\theta \d\theta\right)^{m-2},
\end{align}
where for the second equal sign we used that the $k=0$ term in the first binomial sum, vanishes. 

Then, in the limit $\eta$ tends to infinity, we find
\begin{align} 
\lim_{\eta\rightarrow\infty} H(-\eta) &= \frac{m-1}{m} (dw_d)^{m-1} \left(\binom{\frac{1}{m-1}}{1}\int_0^\pi \cos^2\theta\sin^{d-1}\theta \d\theta\right)\left(\binom{\frac{1}{m-1}}{0}\int_0^\pi\sin^{d-1}\theta\d\theta\right)^{m-2}\\
&= \frac{m-1}{m} (dw_d)^{m-1} \frac{1}{m-1} \times \frac{1}{\dm+1} \left(\int_0^\pi\sin^{d-1}\theta\d\theta\right)^{m-1}\\
&= \frac{|\bbs^\dm|^{m-1}}{m(d+1)},
\end{align}
where for the second equal sign we used \eqref{eqn:cossin-rel}. This proves the claim.

\section{Proof of Lemma \ref{lem:F-monotone}}
\label{appendix:F-monotone}

Compute from \eqref{eqn:F}:
\begin{equation}
\label{eqn:dlnF}
\begin{aligned}
\frac{F'(\phi)}{F(\phi)}=&\frac{\sin\phi}{m-1}\times\frac{\int_0^\phi (\cos\theta-\cos\phi)^{\frac{1}{m-1}-1}\sin^{\dm-1}\theta\cos\theta \d\theta}{\int_0^\phi (\cos\theta-\cos\phi)^{\frac{1}{m-1}}\sin^{\dm-1}\theta \cos\theta \d\theta}\\[3pt]
&+\frac{(m-2)\sin\phi}{m-1}\times\frac{\int_0^\phi (\cos\theta-\cos\phi)^{\frac{1}{m-1}-1}\sin^{\dm-1}\theta \d\theta}{\int_0^\phi (\cos\theta-\cos\phi)^{\frac{1}{m-1}}\sin^{\dm-1}\theta \d\theta}.
\end{aligned}
\end{equation}
Hence, if we show that the following is positive:
\begin{align*}
G(\phi):=&\left(\int_0^\phi (\cos\theta-\cos\phi)^{\frac{1}{m-1}-1}\sin^{\dm-1}\theta\cos\theta \d\theta\right)\left(\int_0^\phi (\cos\theta-\cos\phi)^{\frac{1}{m-1}}\sin^{\dm-1}\theta \d\theta\right)\\
&+(m-2)\left(\int_0^\phi (\cos\theta-\cos\phi)^{\frac{1}{m-1}-1}\sin^{\dm-1}\theta \d\theta\right)\left(\int_0^\phi (\cos\theta-\cos\phi)^{\frac{1}{m-1}}\sin^{\dm-1}\theta \cos\theta \d\theta\right),
\end{align*}
then we get that $F'(\phi)$ is positive, and the conclusion follows.

To simplify $G(\phi)$, we use the following two identities obtained from integration by parts:
\begin{align*}
&\int_0^\phi(\cos\theta-\cos\phi)^{\frac{1}{m-1}-1}\sin^{\dm-1}\theta \cos\theta \d\theta\\
&\quad =\left[\frac{1}{\dm}(\cos\theta-\cos\phi)^{\frac{1}{m-1}-1}\sin^{\dm}\theta\right]_{\theta=0}^{\theta=\phi}+\frac{1}{\dm}\left(\frac{1}{m-1}-1\right)\int_0^\phi(\cos\theta-\cos\phi)^{\frac{1}{m-1}-2}\sin^{\dm+1}\theta\d\theta\\
& \quad =\frac{2-m}{\dm(m-1)}\int_0^\phi(\cos\theta-\cos\phi)^{\frac{1}{m-1}-2}\sin^{\dm+1}\theta\d\theta,
\end{align*}
and
\begin{align*}
&\int_0^\phi(\cos\theta-\cos\phi)^{\frac{1}{m-1}}\sin^{\dm-1}\theta \cos\theta \d\theta\\
& \quad =\left[\frac{1}{\dm}(\cos\theta-\cos\phi)^{\frac{1}{m-1}}\sin^{\dm}\theta\right]_{\theta=0}^{\theta=\phi}+\frac{1}{\dm}\left(\frac{1}{m-1}\right)\int_0^\phi(\cos\theta-\cos\phi)^{\frac{1}{m-1}-1}\sin^{\dm+1}\theta\d\theta\\
&\quad =\frac{1}{\dm(m-1)}\int_0^\phi(\cos\theta-\cos\phi)^{\frac{1}{m-1}-1}\sin^{\dm+1}\theta\d\theta.
\end{align*}
Then, we have
\begin{align}
G(\phi)&=\frac{2-m}{\dm(m-1)}\left(\int_0^\phi(\cos\theta-\cos\phi)^{\frac{1}{m-1}-2}\sin^{\dm+1}\theta\d\theta\right)\left(\int_0^\phi(\cos\theta-\cos\phi)^{\frac{1}{m-1}}\sin^{\dm-1}\theta\d\theta\right)\\
& \quad -\frac{2-m}{\dm(m-1)}\left(\int_0^\phi(\cos\theta-\cos\phi)^{\frac{1}{m-1}-1}\sin^{\dm-1}\theta\d\theta\right)\left(
\int_0^\phi(\cos\theta-\cos\phi)^{\frac{1}{m-1}-1}\sin^{\dm+1}\theta\d\theta
\right).
\end{align}
Since $\frac{2-m}{\dm(m-1)}>0$, we eventually need to show
\begin{equation}
\label{eqn:ineq-G}
\begin{aligned}
&\left(\int_0^\phi(\cos\theta-\cos\phi)^{\frac{1}{m-1}-2}\sin^{\dm+1}\theta\d\theta\right)\left(\int_0^\phi(\cos\theta-\cos\phi)^{\frac{1}{m-1}}\sin^{\dm-1}\theta\d\theta\right)\\
&\quad -\left(\int_0^\phi(\cos\theta-\cos\phi)^{\frac{1}{m-1}-1}\sin^{\dm-1}\theta\d\theta\right)\left(
\int_0^\phi(\cos\theta-\cos\phi)^{\frac{1}{m-1}-1}\sin^{\dm+1}\theta\d\theta
\right)\geq0.
\end{aligned}
\end{equation}

With $\phi \in (0,\pi]$ fixed, write
\begin{align*}
&2\left(\int_0^\phi(\cos\theta-\cos\phi)^{\frac{1}{m-1}-2}\sin^{\dm+1}\theta\d\theta\right)\left(\int_0^\phi(\cos\theta-\cos\phi)^{\frac{1}{m-1}}\sin^{\dm-1}\theta\d\theta\right)\\
&-2\left(\int_0^\phi(\cos\theta-\cos\phi)^{\frac{1}{m-1}-1}\sin^{\dm-1}\theta\d\theta\right)\left(
\int_0^\phi(\cos\theta-\cos\phi)^{\frac{1}{m-1}-1}\sin^{\dm+1}\theta\d\theta\right)\\
& \quad = \int_0^\phi\int_0^\phi(\cos\theta_1-\cos\phi)^{\frac{1}{m-1}-2}\sin^{\dm+1}\theta_1(\cos\theta_2-\cos\phi)^{\frac{1}{m-1}}\sin^{\dm-1}\theta_2\d\theta_1\d\theta_2\\
&\qquad + \int_0^\phi\int_0^\phi(\cos\theta_2-\cos\phi)^{\frac{1}{m-1}-2}\sin^{\dm+1}\theta_2(\cos\theta_1-\cos\phi)^{\frac{1}{m-1}}\sin^{\dm-1}\theta_1\d\theta_1\d\theta_2\\
&\qquad - \int_0^\phi\int_0^\phi(\cos\theta_1-\cos\phi)^{\frac{1}{m-1}-1}\sin^{\dm-1}\theta_1(\cos\theta_2-\cos\phi)^{\frac{1}{m-1}-1}\sin^{\dm+1}\theta_2\d\theta_1\d\theta_2\\
&\qquad -\int_0^\phi\int_0^\phi(\cos\theta_2-\cos\phi)^{\frac{1}{m-1}-1}\sin^{\dm-1}\theta_2(\cos\theta_1-\cos\phi)^{\frac{1}{m-1}-1}\sin^{\dm+1}\theta_1\d\theta_1\d\theta_2\\
&\quad =\int_0^\phi\int_0^\phi(\cos\theta_1-\cos\phi)^{\frac{1}{m-1}-2}(\cos\theta_2-\cos\phi)^{\frac{1}{m-1}-2}\sin^{\dm-1}\theta_1\sin^{\dm-1}\theta_2 \times\left(
(\cos\theta_2-\cos\phi)^2\sin^2\theta_1 \right.\\
&\qquad \left. +(\cos\theta_1-\cos\phi)^2\sin^2\theta_2-(\cos\theta_1-\cos\phi)(\cos\theta_2-\cos\phi)(\sin^2\theta_1+\sin^2\theta_2)
\right)\d\theta_1\d\theta_2.
\end{align*}

Using
\begin{align*}
&(\cos\theta_2-\cos\phi)^2\sin^2\theta_1+(\cos\theta_1-\cos\phi)^2\sin^2\theta_2-(\cos\theta_1-\cos\phi)(\cos\theta_2-\cos\phi)(\sin^2\theta_1+\sin^2\theta_2)\\
&\quad =\sin^2\theta_1(\cos\theta_2-\cos\phi)(\cos\theta_2-\cos\theta_1)+\sin^2\theta_2(\cos\theta_1-\cos\phi)(\cos\theta_1-\cos\theta_2)\\[5pt]
&\quad =(\cos\theta_2-\cos\theta_1)\sin^2\theta_1\sin^2\theta_2\left(\frac{\cos\theta_2-\cos\phi}{\sin^2\theta_2}-\frac{\cos\theta_1-\cos\phi}{\sin^2\theta_1}\right),
\end{align*}
we write the above as
\begin{equation}
\label{eqn:ineq-Gint}
\begin{aligned}
&2\left(\int_0^\phi(\cos\theta-\cos\phi)^{\frac{1}{m-1}-2}\sin^{\dm+1}\theta\d\theta\right)\left(\int_0^\phi(\cos\theta-\cos\phi)^{\frac{1}{m-1}}\sin^{\dm-1}\theta\d\theta\right)\\
&-2\left(\int_0^\phi(\cos\theta-\cos\phi)^{\frac{1}{m-1}-1}\sin^{\dm-1}\theta\d\theta\right)\left(
\int_0^\phi(\cos\theta-\cos\phi)^{\frac{1}{m-1}-1}\sin^{\dm+1}\theta\d\theta\right)\\
=&\int_0^\phi\int_0^\phi(\cos\theta_1-\cos\phi)^{\frac{1}{m-1}-2}(\cos\theta_2-\cos\phi)^{\frac{1}{m-1}-2}\sin^{\dm-1}\theta_1\sin^{\dm-1}\theta_2\\
&\times(\cos\theta_2-\cos\theta_1)\sin^2\theta_1\sin^2\theta_2\left(\frac{\cos\theta_2-\cos\phi}{\sin^2\theta_2}-\frac{\cos\theta_1-\cos\phi}{\sin^2\theta_1}\right)\d\theta_1\d\theta_2.
\end{aligned}
\end{equation}

Note that $\frac{\cos\theta-\cos\phi}{\sin^2\theta}$ is decreasing on $\theta \in (0,\pi)$, as 
\[
\frac{\d}{\d\theta}\left(\frac{\cos\theta-\cos\phi}{\sin^2\theta}\right)=\frac{-\cos^2\theta+2\cos\phi\cos\theta-1}{\sin^3\theta}=-\frac{(\cos\theta-\cos\phi)^2+\sin^2\phi}{\sin^3\theta}\leq0.
\]
Since $\cos\theta$ is also decreasing on $[0,\pi]$, we have
\[
(\cos\theta_2-\cos\theta_1)\left(\frac{\cos\theta_2-\cos\phi}{\sin^2\theta_2}-\frac{\cos\theta_1-\cos\phi}{\sin^2\theta_1}\right)\geq0, \qquad\forall \theta_1,\theta_2\in[0,\phi].
\]
Therefore, from \eqref{eqn:ineq-Gint} we can conclude that \eqref{eqn:ineq-G} holds, which completes the proof.


\section{Calculations in support of Remark \ref{rmk:kappa2}}
\label{appendix:kappa2}

By \eqref{eqn:kappa2}, we have
\begin{equation*}
\kappa_2^{-1}=\frac{m-1}{m}(\dm w_\dm)^{m-1}\left(\int_0^\pi (\cos\theta+1)^{\frac{1}{m-1}}\sin^{\dm-1}\theta \cos\theta\d\theta
\right)
\left(\int_0^\pi (\cos\theta+1)^{\frac{1}{m-1}}\sin^{\dm-1}\theta\d\theta\right)^{m-2},
\end{equation*}

Recall the identity
\[
\int_0^\pi \cos^a\left(\frac{\theta}{2}\right)\sin^b\left(\frac{\theta}{2}\right)\d\theta=\frac{\Gamma\left(\frac{a+1}{2}\right)\Gamma\left(\frac{b+1}{2}\right)}{\Gamma\left(\frac{a+b+2}{2}\right)}.
\]
Using several basic trigonometry formulas 
\[
\cos\theta+1=2\cos^2\frac{\theta}{2}, \quad \sin\theta=2 \sin \left(\frac{\theta}{2}\right)\cos\left(\frac{\theta}{2}\right),\quad \cos\theta=\cos^2\left(\frac{\theta}{2}\right)-\sin^2\left(\frac{\theta}{2}\right),
\]
and the property $\Gamma(z+1)=z \Gamma(z)$ for all $z>0$, we calculate
\begin{align*}
&\int_0^\pi (\cos\theta+1)^{\frac{1}{m-1}}\sin^{\dm-1}\theta \cos\theta\d\theta\\
&=2^{\frac{1}{m-1}}\int_0^\pi \cos^{\frac{2}{m-1}}\left(\frac{\theta}{2}\right)\sin^{\dm-1}\theta \cos\theta\d\theta\\
&=2^{\frac{1}{m-1}+\dm-1}\int_0^\pi \left(\cos^{\frac{2}{m-1}+\dm+1}\left(\frac{\theta}{2}\right)\sin^{\dm-1}\left(\frac{\theta}{2}\right)-\cos^{\frac{2}{m-1}+\dm-1}\left(\frac{\theta}{2}\right)\sin^{\dm+1}\left(\frac{\theta}{2}\right)\right) \d\theta\\
&=2^{\frac{1}{m-1}+\dm-1}\left(\frac{\Gamma\left(\frac{1}{m-1}+\frac{\dm+2}{2}\right)\Gamma\left(\frac{\dm}{2}\right)-\Gamma\left(\frac{1}{m-1}+\frac{\dm}{2}\right)\Gamma\left(\frac{\dm+2}{2}\right)}{\Gamma\left(\frac{1}{m-1}+\dm+1\right)}\right)\\
&=2^{\frac{1}{m-1}+\dm-1}\frac{\Gamma\left(\frac{1}{m-1}+\frac{\dm}{2}\right)\Gamma\left(\frac{\dm}{2}\right)}{\Gamma\left(\frac{1}{m-1}+\dm\right)}\left(\frac{\frac{1}{m-1}+\frac{\dm}{2}-\frac{\dm}{2}}{\frac{1}{m-1}+\dm}\right)\\
&=2^{\frac{1}{m-1}+\dm-1}\frac{\Gamma\left(\frac{1}{m-1}+\frac{\dm}{2}\right)\Gamma\left(\frac{\dm}{2}\right)}{\Gamma\left(\frac{1}{m-1}+\dm\right)}\times\frac{1}{1+\dm(m-1)}.
\end{align*}
Similarly, 
\begin{align*}
&\int_0^\pi (\cos\theta+1)^{\frac{1}{m-1}}\sin^{\dm-1}\theta \d\theta\\
&=2^{\frac{1}{m-1}}\int_0^\pi \cos^{\frac{2}{m-1}}\left(\frac{\theta}{2}\right)\sin^{\dm-1}\theta\d\theta\\
&=2^{\frac{1}{m-1}+\dm-1}\int_0^\pi\cos^{\frac{2}{m-1}+\dm-1}\left(\frac{\theta}{2}\right)\sin^{\dm-1}\left(\frac{\theta}{2}\right) \d\theta\\
&=2^{\frac{1}{m-1}+\dm-1}\frac{\Gamma\left(\frac{1}{m-1}+\frac{\dm}{2}\right)\Gamma\left(\frac{\dm}{2}\right)}{\Gamma\left(\frac{1}{m-1}+\dm\right)}
\end{align*}

Therefore, we can find an explicit expression of $\kappa_2$ in terms of the Gamma function:
\begin{align*}
\kappa_2^{-1}=\frac{m-1}{m}(\dm w_\dm)^{m-1}2^{1+(\dm-1)(m-1)}\left(\frac{\Gamma\left(\frac{1}{m-1}+\frac{\dm}{2}\right)\Gamma\left(\frac{\dm}{2}\right)}{\Gamma\left(\frac{1}{m-1}+\dm\right)}\right)^{m-1}\times\frac{1}{1+\dm(m-1)}.
\end{align*}
Finally, using 
\[
|\bbs^\dm|=\frac{\dm w_\dm\sqrt{\pi}\Gamma\left(\frac{\dm}{2}\right)}{\Gamma\left(\frac{\dm+1}{2}\right)},
\]
one arrives at \eqref{eqn:kappa2-alt}.


\bibliographystyle{abbrv}
\def\url#1{}
\bibliography{lit.bib}

\end{document}